\documentclass[11pt]{article}

\usepackage[margin=3cm]{geometry}

\usepackage{mathrsfs}
\usepackage{amsmath}
\usepackage{amsthm}
\usepackage{color}
\usepackage{amsfonts,amssymb}
\usepackage{graphicx}

\theoremstyle{plain}
\newtheorem{Theorem}{Theorem}[section]
\newtheorem{Corollary}[Theorem]{Corollary}
\newtheorem{Lemma}[Theorem]{Lemma}
\newtheorem{Proposition}[Theorem]{Proposition}
\newtheorem{Conjecture}[Theorem]{Conjecture}
\newtheorem{Claim}[Theorem]{Claim}
\newtheorem{Fact}[Theorem]{Fact}

\theoremstyle{definition}

\renewcommand{\P}{\mathbb{P}}
\newcommand{\E}{\mathbb{E}}

\newcommand{\G}{\mathcal{G}}

\newcommand{\V}{\mathrm{V}_N}
\newcommand{\D}{\mathcal{D}_R}
\newcommand{\eps}{\varepsilon}

\numberwithin{equation}{section}

\title{Clustering and the hyperbolic geometry of complex networks\footnote{An extended abstract of this paper appeared in the
Proceedings of the 11th Workshop on Algorithms and Models for the Web Graph (WAW '14).}}

\author{Elisabetta Candellero\footnote{Department of Statistics, University of Warwick, Coventry CV4 7AL, United Kingdom} 
\and {Nikolaos Fountoulakis\footnote{School of Mathematics, University of Birmingham, Edgbaston B15 2TT, United Kindgom. 
Research supported by a Marie Curie Career Integration
Grant PCIG09-GA2011-293619.}}}

\begin{document}
\maketitle

\begin{abstract}
Clustering is a fundamental property of complex networks and it is the mathematical expression of a ubiquitous phenomenon that arises
in various types of self-organized networks such as biological networks, computer networks or social networks. In this paper, we consider
what is called the \emph{global clustering coefficient} of random graphs on the hyperbolic plane. This model of random graphs was proposed
recently by Krioukov et al.~\cite{ar:Krioukov} as a mathematical model of complex networks, under the fundamental assumption that hyperbolic
geometry underlies the structure of these networks. We give a rigorous analysis of clustering and characterize the global clustering
coefficient in terms of the parameters of the model. We show how the global clustering coefficient can be tuned by these parameters
and we give an explicit formula for this function.
\end{abstract}

% \begin{keyword}[class=MSC]
% \kwd[Primary: ]{05C82}% Small world graphs, complex networks
% \kwd{05C80} %Random graphs
% \kwd{68R05} %Combinatorics
% \kwd{91D30} % Social networks
% \kwd[; secondary: ]{60C05} %Combinatorial probability
% \kwd{90B15} %Network models, stochastic
% \end{keyword}

% \begin{keyword}
% \kwd{clustering coefficient}
% \kwd{random geometric graphs}
% \kwd{hyperbolic plane}
% \end{keyword}

%\end{frontmatter}

\tableofcontents

\section{Introduction}

The theory of complex networks was developed during the last 15 years mainly as a unifying mathematical framework for modeling
a variety of networks such as biological networks or large computer networks among which is the Internet, the World Wide Web as well as
social networks that have been recently developed over these platforms. A number of mathematical models have emerged whose aim is to
describe fundamental characteristics of these networks as these have been described by experimental evidence -- see for
example~\cite{ar:StatMechs}. Loosely speaking, the notion of a complex network refers to a class of large networks which exhibit the
following characteristics:
\begin{enumerate}
\item[1.] they are \emph{sparse}, that is, the number of their edges is proportional to the number of nodes; 
\item[2.] they exhibit the \emph{small world phenomenon}: most pairs of vertices which belong to the same component 
are within a short distance from each other;
\item[3.] \emph{clustering}: two nodes of the network that have a common 
neighbour are somewhat more likely to be connected with each other;  
\item[4.] the tail of their degree distribution follows a \emph{power law}. In particular, experimental evidence 
(see~\cite{ar:StatMechs}) indicates that many networks that emerge in applications follow power law degree distribution with exponent between 
2 and 3. 
\end{enumerate}
The books of Chung and Lu~\cite{ChungLuBook+} and of Dorogovtsev~\cite{Dor} are excellent references for a detailed discussion of these
properties. 

Among the most influential models was the Watts-Strogatz model of small worlds~\cite{ar:WatStrog98}
and the Barab\'asi-Albert model~\cite{ar:BarAlb}, that is also known as the preferential attachment model. The main typical characteristics
of these networks have to do with the distribution of the degrees (e.g., power-law distribution), the existence of clustering as well as the
typical distances between vertices (e.g., the small world effect).
These models as well as other known models, such as the Chung-Lu model (defined by Chung and Lu~\cite{ChungLu1+},
\cite{ChungLuComp+}) fail to capture \emph{all} the above features simultaneously or if they do so, they do it in a way that is difficult 
to tune these features independently. 
For example, the Barab\'asi-Albert model does exhibit a power law 
degree distribution, with exponent between 2 and 3, and average distance of order $O(\log \log N)$, when it is suitably parameterized. 
However, it is locally tree-like around a typical vertex (cf.~\cite{ar:BolRiordan},~\cite{ar:EggNoble}). 
On the other hand, the Watts-Strogatz model, although it exhibits clustering and small distances between the vertices, has degree 
distribution that decays exponentially~\cite{ar:BarrWeigt2000}.

The notion of clustering formalizes the property that two nodes of a network that share
a neighbor (for example two individuals that have a common friend) are more likely to  be joined by an edge (that is, to be friends of
each other). In the context of social networks, sociologists
have explained this phenomenon through the notion of \emph{homophily}, which refers to the tendency of individuals to be related with
similar individuals, e.g.\ having similar socioeconomic background or similar educational background. 
There have been numerous attempts to define models where clustering is present -- see for example the work of Coupechoux and 
Lelarge~\cite{ar:LelCoup} or that of Bollob\'as, Janson and Riordan~\cite{ar:BollJanRior} where this is combined with the general notion 
of inhomogeneity. In that context, clustering is \emph{planted} in a sparse random graph. 
Also, it is even more rare to quantify clustering precisely (see for example the work of~\cite{ar:Bloznelis2013} on random intersection 
graphs). This is the case as the presence of clustering is the outcome of heavy dependencies between the edges of the random graphs and, 
in general, these are not easy to handle. 

However, clustering is naturally present on random graphs that are created on a metric space, as is the case of a random geometric graph 
on the Euclidean plane.
The theory of random geometric graphs was initiated by Gilbert~\cite{Gilbert61} already in 1961 (see also~\cite{bk:MeesterRoy}) and
started taking its present form later by Hafner~\cite{ar:Hafner72}.
In its standard form a geometric random graph is created as follows: $N$ points are sampled within a subset of $\mathbb{R}^d$
following a particular distribution (most usually this is the uniform distribution or the distribution of the point-set of a Poisson
point process) and any two of them are joined when their Euclidean distance is smaller than some threshold value which, in general, is
a function of $N$. During the last two decades, this kind of random graphs was studied in depth by several researchers -- see
the monograph of Penrose~\cite{bk:Penrose} and the references therein.
Numerous typical properties of such random graphs have been investigated, such as the chromatic number~\cite{McDiarmid}, 
Hamiltonicity~\cite{Balogh} etc.

There is no particular reason why a random geometric graph on a Euclidean space would be intrinsically associated with the 
formation of a complex network. 
Real-world networks consist of heterogeneous nodes, which can be classified into groups. In turn, these groups
can be classified into larger groups which belong to bigger subgroups and so on. For example, if we consider the network of citations, 
whose set of nodes is the set of research papers and there is a link from one paper to another if one cites the other, there is a 
natural classification of the nodes according to the scientific fields each paper belongs to (see for example~\cite{Boerner04+}). In the
case of the network of web pages, a similar classification can be considered in terms of the similarity between two web pages. 
That is, the more similar two web pages are, the more likely it is that there exists a hyperlink between them (see~\cite{Mencezer02+}).  

This classification can be approximated by tree-like structures representing the hidden hierarchy of the network. 
The tree-likeness suggests that in fact the geometry of this hierarchy is \emph{hyperbolic}. 
One of the basic features of a hyperbolic space is that the volume growth is exponential which is also the case, for example, 
when one considers a $k$-ary tree, that is, a rooted tree where every vertex has $k$ children. 
Let us consider for example the Poincar\'e unit disc model (which we will discuss in more detail in the next section). 
If we place the root of an infinite $k$-ary tree at the centre of 
the disc, then the hyperbolic metric provides the necessary room to embed the tree into the disc so that every edge has unit length
in the embedding. 

Recently Krioukov et al.~\cite{ar:Krioukov} introduced a model which implements this idea. 
In this model, a random network is created on the hyperbolic plane (we will see the detailed definition shortly). In particular, 
Krioukov et al.~\cite{ar:Krioukov} determined the degree distribution for \emph{large} degrees showing that it is \emph{scale free} 
and its tail follows a power law, whose exponent is determined by some of the parameters of the model. 
Furthermore, they consider the clustering properties of the resulting random network. A numerical approach in~\cite{ar:Krioukov} suggests
that the (local) clustering coefficient\footnote{This is defined as the average density of the neighbourhoods of the vertices} is
positive and it is determined by one of the parameters of the model. In fact, as we will discuss in Section~\ref{sec:geom_asp}, 
this model corresponds to the sparse regime of random geometric graphs on the hyperbolic plane and hence is of independent interest 
within the theory of random geometric graphs. 

This paper investigates \emph{rigorously} the presence of clustering in this model, through the notion of the \emph{clustering
coefficient}. Our first contribution is that we manage to determine exactly the value of the clustering coefficient as a function of the
parameters of the model. More importantly, our results imply that in fact the exponent of the power law, the density of the random graph 
and the amount of clustering can be tuned \emph{independently} of each other, through the parameters of the random graph.  
Furthermore, we should point out that the clustering coefficient we consider is the so-called \emph{global clustering coefficient}.
Its calculation involves tight concentration bounds on the number of triangles in the random graph. Hence, our analysis initiates an
approach to the small subgraph counting problem in these random graphs, which is among the central problems in the 
general theory of random graphs~\cite{Bollobas},\cite{JLR} and of random geometric graphs~\cite{bk:Penrose}. 

We now proceed with the definition of the model of random geometric graphs on
the hyperbolic plane.

\subsection{Random geometric graphs on the hyperbolic plane}
The most common representations of the hyperbolic plane are the upper-half plane representation $\{z \ : \ \Im z > 0 \}$ as
well as the Poincar\'e unit disc which is simply the open disc of radius one, that is, $\{(u,v) \in \mathbb{R}^2 \ : \ 1-u^2-v^2 > 0 \}$.
Both spaces are equipped with the hyperbolic metric; in the former
case this is $\frac{dx^2+ dy^2} {(\zeta y)^2}$ whereas in the latter this is ${4\over \zeta^2}~{du^2 + dv^2\over (1-u^2-v^2)^2}$, where
$\zeta$ is some positive real number.
It can be shown that the (Gaussian) curvature in both cases is equal to $-\zeta^2$ and the two spaces are isometric, i.e., there
is a bijection between the two spaces that preserves (hyperbolic) distances. In fact, there are more representations of the 2-dimensional
hyperbolic space of curvature $-\zeta^2$ which are isometrically equivalent to the above two. We will denote by $\mathbb{H}^2_\zeta$ the
class of these spaces.

In this paper, following the definitions in~\cite{ar:Krioukov}, we shall be using the \emph{native} representation of $\mathbb{H}^2_\zeta$.
Under this representation, the ground space of $\mathbb{H}^2_{\zeta}$ is $\mathbb{R}^2$ and every point $x \in \mathbb{R}^2$ whose
polar coordinates are $(r,\theta)$ has hyperbolic distance from the origin equal to $r$. 
More precisely, the native representation can be viewed as a mapping of the Poincar\'e unit disc to $\mathbb{R}^2$, where the origin 
of the unit disc is mapped to the origin of $\mathbb{R}^2$ and every point $v$ in the Poincar\'e disc is mapped to a point  
$v'\in \mathbb{R}^2$, where $v' = (r,\theta)$ in polar coordinates. 
More specifically, $r$ is the hyperbolic distance of $v$ from the origin of the Poincar\'e disc and $\theta$ is its angle.

Also, a circle of radius $r$ around the origin
has length equal to ${2\pi\over \zeta} \sinh \zeta r$ and area equal to ${2\pi \over \zeta^2} (\cosh \zeta r - 1)$.

We are now ready to give the definitions of the two basic models introduced in~\cite{ar:Krioukov}.
Consider the native representation of the hyperbolic plane of curvature $K = - \zeta^2$, for some $\zeta > 0$.
For some constant $\nu >0$, we let $N= \nu e^{\zeta R/2}$ -- thus $R$ grows logarithmically as a function of $N$. 
We shall explain the role of $\nu$ shortly. 
{Let $V_N = \{ v_1,\ldots, v_N \}$ be the set of vertices of the random graphs, where $v_i$ is a random point 
on the disc of radius $R$ centered at the origin $O$, which we denote by $\D$. In other words, the vertex set $V_N$ 
is a set of \emph{i.i.d.} random variables that take values on $\D$. }

Their distribution is as follows.  Assume that $u$ has
polar coordinates $(r, \theta)$. The angle $\theta$ is uniformly distributed in $(0,2\pi]$ and  the probability density function of
$r$, which we denote by $\rho_N (r)$, is determined by a parameter $\alpha >0$ and is equal to
\begin{equation} \label{eq:pdf}
 \rho_N (r) = \begin{cases}
\alpha {\sinh  \alpha r \over \cosh \alpha R - 1}, & \mbox{if $0\leq r \leq R$} \\
0, & \mbox{otherwise}
\end{cases}.
\end{equation}
Note that when $\alpha = \zeta$, this is simply the uniform distribution.

An alternative way to define this distribution is as follows. Consider $\mathbb{H}^2_{\alpha}$ and select arbitrarily a point $O'$ 
on this space, which will be assumed to be its origin. Consider the disc $\D'$ of radius $R$ around $O'$ and select $N$ points 
within $\D'$ uniformly at random. Subsequently, the selected points are projected onto $\D$ preserving their polar coordinates.   
The projections of these points, which we will be denoting by $\V$, will be the vertex set of the random graph.
We will be also treating the vertices as points in the hyperbolic space indistinguishably.

{  More precisely, 
given $\alpha, \zeta>0$, we consider two different Poincar\'e disk representations using two different metrics: one being $(ds)^2:={4\over \zeta^2}~{du^2 + dv^2\over (1-u^2-v^2)^2}$, and the other $(ds')^2:={4\over \alpha^2}~{du^2 + dv^2\over (1-u^2-v^2)^2}$.
Subsequently, we consider the disk around the origin with (hyperbolic) radius $R$ in the metric $(ds')^2$.
There, we select the $N=\nu e^{\zeta R/2}$ vertices of $V_N$ uniformly at random.
Thereafter, we project these vertices onto the disk of  (hyperbolic) radius $R$ under the metric $(ds)^2$ on which we create 
the random graph. }

Note that the curvature in this case determines the rate of growth of the space. Hence, when $\alpha < \zeta$, the $N$ points 
are distributed on a disc (namely $\D'$) which has smaller area compared to $\D$. This naturally increases the density of those points 
that are located closer to the origin. Similarly, when $\alpha > \zeta$ the area of the disc $\D'$ is larger than that of $\D$, and most of the 
$N$ points are significantly more likely to be located near the boundary of $\D'$, due to the exponential growth of the volume. 

Given the set $\V$ on $\D$ we define the following two models of random graphs. 
\begin{enumerate}
\item[1.] \emph{The disc model}:
this model is the most commonly studied in the theory of random geometric graphs on Euclidean spaces.
We join two vertices if they are within (hyperbolic) distance $R$ from each other.
Figure~\ref{fig:disc} illustrates a disc of radius $R$ around a vertex $v \in \D$. 
\begin{figure}[htp] 
 \centering
 \includegraphics[scale=0.4]{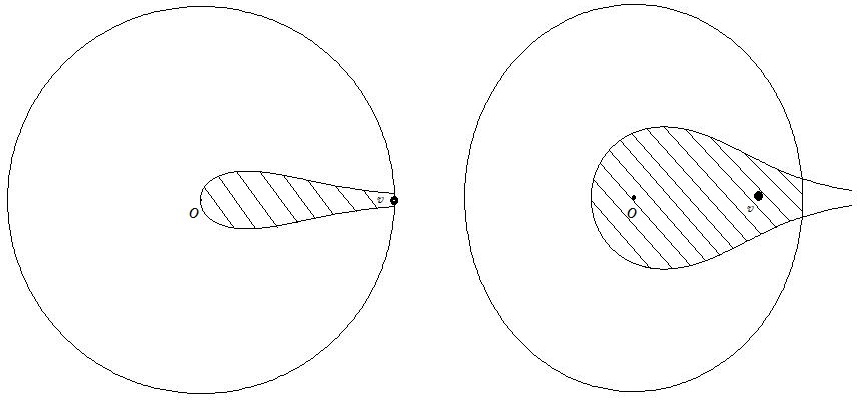}
 \caption{The disc of radius $R$ around $v$ in $\D$}
 \label{fig:disc}
 \end{figure}

\item[2.] \emph{The binomial model}:
we join any two distinct vertices $u, v$ with probability
$$ p_{u,v} = {1\over \exp\left(\beta ~ {\zeta \over 2}(d(u,v)- R)\right) + 1},$$
independently of every other pair, where $\beta >0$ is fixed and $d(u,v)$ is the hyperbolic distance between $u$ and $v$.
We denote the resulting random graph by $\G (N;\zeta, \alpha, \beta,\nu)$.
\end{enumerate}
The binomial model is in some sense a \emph{soft} version of the disc model. In the latter, two vertices become adjacent 
if and only if their hyperbolic distance is at most $R$. This is \emph{approximately} the case in the former model. 
If $d(u,v)= (1+\delta) R$, where $\delta >0$ is some constant, then $p_{u,v} \rightarrow 0$, whereas if $d(u,v) = (1-\delta) R$, then
$p_{u,v} \rightarrow 1$, as $N \rightarrow \infty$. 
{(Recall that $R\to \infty$ as $N\to \infty$.)}

Also, the disc model can be viewed as a limiting case of the binomial model as 
$\beta \rightarrow \infty$. Assume that the positions 
of the vertices in $\D$ have been realized. If $u,v \in V_N$ are such that $d(u,v) < R$, then when $\beta \rightarrow \infty$ the 
probability that $u$ and $v$ are adjacent tends to 1; however, if $d(u,v)>R$, then this probability converges to 0 as $\beta$ grows. 
%Rigorous results for the disc model were obtained by Gugelmann et al.~\cite{ar:Kosta}, regarding their degree distribution as well as 
%the clustering coefficient.  

The binomial model has also a statistical-mechanical interpretation where the parameter $\beta$ can be seen as the inverse temperature of
a fermionic system where particles correspond to the edges of the random graph -- see~\cite{ar:Krioukov} for a detailed discussion.  

Krioukov et al.~\cite{ar:Krioukov} provide an argument which indicates that in both models the degree distribution has a power law 
tail with exponent that is equal to $2\alpha /\zeta +1$. Hence, when $0 < \zeta /\alpha < 2$, any exponent greater than $2$ can 
be realised. This has been shown rigorously by Gugelmann et al.~\cite{ar:Kosta}, for the disc model, and by the second 
author~\cite{ar:Foun13+}, for the binomial model. In the latter case, the average degree of a vertex depends on all four parameters 
of the model. For the disc model in particular, having fixed $\zeta$ and $\alpha$, which determine the exponent of the power law, the 
parameter $\nu$ determines the average degree. In the binomial model, there is an additional dependence on $\beta$.  
However, our main results show that clustering \emph{does not depend} on $\nu$. Therefore, in the binomial model the ``amount" of
clustering and the average degree can be tuned \emph{independently}. 

The second author has shown~\cite{ar:Foun13+} that $\beta =1$ is a critical point around which there is a transition on the 
density of $\G (N;\zeta, \alpha, \beta,\nu)$. 
When $\beta \leq 1$, the average degree of the random graph grows at least logarithmically in $N$. 
The case where $\beta > 1$ and $0 < \zeta /\alpha < 2$ becomes of particular interest as the random graph obtains the characteristics 
that are ascribed to complex networks.
More specifically, in this regime  $\G (N; \zeta, \alpha, \beta,\nu)$ has constant (i.e., not
depending on $N$) average degree that depends on $\nu, \zeta, \alpha$ and $\beta$, whereas the degree distribution follows the tail of
a power law with exponent $2\alpha / \zeta +1$. {Theorem 1.2 in~\cite{ar:Foun13+} implies that the average degree is 
proportional to $\nu^{2\alpha /\zeta} f(\zeta, \alpha, \beta)$. 
In this paper, we show that the (global) clustering coefficient is only a function of $\zeta, \alpha$ and $\beta$. }

\subsection{Notation}
Let $\{ X_N \}_{N \in \mathbb{N}}$ be a sequence of real-valued random variables on a sequence of probability spaces
$\{ (\Omega_N, \mathbb{P}_N)\}_{N \in \mathbb{N}}$, and let $\{ a_N \}_{N \in \mathbb{N}}$ be a sequence of real numbers that tends to infinity as $N \rightarrow \infty$.

We write $X_N = o_p (a_N)$, if $|X_N|/a_N$ \emph{converges to 0 in probability}.
That is, for any $\eps >0$, we have $\mathbb{P}_N (|X_N /a_N|>\eps) \rightarrow 0$ as $N \rightarrow \infty$.
Additionally, we write $X_N = \Theta_C (a_N)$ if there exist positive real numbers $C_1 , C_2$ such that
we have $\mathbb{P} (C_1 a_N \leq |X_N| \leq C_2 a_N ) = 1- o(1)$.
Finally, if $\mathcal{E}_N$ is a measurable subset of $\Omega_N$, for any $N \in \mathbb{N}$, we say that the sequence
$\{ \mathcal{E}_N \}_{N \in \mathbb{N}}$ occurs \emph{asymptotically almost surely (a.a.s.)} if $\mathbb{P} (\mathcal{E}_N) = 1-o(1)$,
as $N\rightarrow \infty$. However, with a slight abuse of terminology, we will be saying that an \emph{event occurs a.a.s.} implicitly referring
to a sequence of events.

For two functions $f,g : \mathbb{N} \rightarrow \mathbb{R}$ we write $f(N) \ll g(N)$ if $f(N)/g(N) \rightarrow 0$ as $N \rightarrow \infty$.
Similarly, we will write $f(N)\asymp g(N)$, meaning that there are positive constants $c_1, c_2$ such that for all $N\in \mathbb{N}$
we have $c_1 g(N) \leq f(N)\leq c_2 g(N)$.
Analogously, we write $f(N)\lesssim g(N)$ (resp.\ $f(N)\gtrsim g(N)$) if there is a positive constant $c$ such that for all $N\in \mathbb{N}$ %large enough,
we have $f(N)\leq cg(N) $ (resp.\ $f(N)\geq cg(N) $). These notions could have been expressed through the standard Landau notation, but
we chose to express them as above in order to make our calculations more readable.

Finally, we write $X_N \lesssim g(N)$ \emph{a.a.s.}, if there is a positive constant $c$ such that $X_N \leq c g(N)$ a.a.s.
An analogous interpretation is used for $X_N \gtrsim g(N)$ \emph{a.a.s.}

\subsection{Some geometric aspects of the two models} \label{sec:geom_asp}

The disc model on the hyperbolic plane can be also viewed within the framework of random geometric graphs. Within this 
framework, the disc model may be defined for \emph{any} threshold distance $r_N$ and not merely for threshold distance equal to $R$. 
However, only taking $r_N = R$ yields a random graph with constant average degree that is bounded away from 0. 
More specifically for any $\delta \in (0,1)$, if $r_N = (1-\delta ) R$, then the resulting 
random graph becomes rather trivial and most vertices have no neighbours. On the other hand, if $r_N = (1+\delta )R$, the resulting 
random graph becomes too dense and its average degree grows polynomially fast in $N$.

The proof of these observations relies on the following lemma which provides a characterization of what it means for two points $u,v$ to
have $d(u,v) \leq (1+\delta) R$, for $\delta \in (-1,1)$, in terms of their \emph{relative angle}, which we denote by $\theta_{u,v}$. 
For this lemma, we need the notion of the \emph{type} of a vertex. 
For a vertex $v \in \V$, if $r_v$ is the distance of  $v$ from the origin, that is, the radius of $v$, then we set $t_v = R - r_v$ -- 
we call this quantity the \emph{type} of vertex $v$. 
\begin{Lemma}\label{lem:relAngle}
Let $\delta \in (-1,1) $ be a real number.
For any $\eps>0$ there exists an $N_0>0$ and a $c_0>0$ such that for any $N>N_0$ and $u,v\in\D$ with $t_u+t_v < (1-|\delta |) R-c_0$
the following hold.
\begin{itemize}
			\item If $\theta_{u,v} < 2 (1-\eps )\exp\left(\frac{\zeta}{2}(t_u+t_v-(1-\delta)R)\right)$, then $d(u,v) < (1+\delta ) R$.
			\item If $\theta_{u,v}> 2(1+\eps )\exp\left(\frac{\zeta}{2}(t_u+t_v-(1-\delta ) R)\right)$,
				then $d(u,v) > (1+\delta ) R$.
		\end{itemize}
\end{Lemma}
The proof of this lemma can be found in Appendix~\ref{app:B}. 

Let us consider temporarily the (modified) disc model, where we assume that two vertices are joined precisely 
when their hyperbolic distance is at most $(1+ \delta ) R$.   
Let $u \in \V$ be a vertex and assume that $t_u < C$ (in fact, by Claim~\ref{clm:density_approx}, if $C$ is large enough, then most
vertices will satisfy this). 
We will show that if $\delta <0$, then the expected degree of $u$, in fact, tends to 0. 
Let us consider a simple case where $\delta$ satisfies $\frac{\zeta}{2 \alpha} < 1-|\delta | <1$. 
As we will see later (Corollary~\ref{cor:x0}), a.a.s. there are no vertices of 
type much larger than ${\zeta \over 2\alpha} R$.  Hence, since $t_u < C$, if $N$ is sufficiently large, then we have 
${\zeta \over 2\alpha} R < (1-|\delta | ) R - t_u - c_0$. 
By Lemma~\ref{lem:relAngle}, the probability that a vertex $v$ has type at most ${\zeta \over 2 \alpha} R$ and it is adjacent to $u$ 
(that is, its hyperbolic distance from $u$ is at most $(1+\delta ) R$) is proportional to
$e^{{\zeta \over 2}(t_u+t_v - (1-\delta)R)}/\pi$.
If we average this over $t_v$, and denote by $u\sim v$ the relation of neighborhood between the two vertices $u$ and $v$, we obtain
\begin{equation*}
\begin{split}
\mathrm{Pr} \left[ u \sim v | t_u \right]
&\asymp {e^{\zeta t_u /2} \over e^{{\zeta \over 2}(1- \delta ) R}}  \int_{0}^{{\zeta \over 2 \alpha} R} e^{\zeta t_v /2} {\alpha \sinh (\alpha (R-t_v)) \over \cosh (\alpha R) -1} dt_v \\
& \lesssim
{e^{\zeta t_u /2} \over e^{{\zeta \over 2}(1- \delta ) R}} \int_0^R e^{\zeta t_v /2}~ {e^{\alpha (R-t_v)} \over 
\cosh (\alpha R) - 1} dt_v \\
& \asymp {e^{\zeta t_u /2} \over e^{{\zeta \over 2}(1- \delta ) R}} \int_0^R e^{(\zeta/2 - \alpha ) t_v} dt_v 
\stackrel{0< \zeta /\alpha <2}{\asymp}  { e^{\zeta t_u/2}\over N^{1-\delta}} \stackrel{\delta < 0}{=} 
o \left( {1\over N} \right). 
\end{split}
\end{equation*}
Hence, the probability that there is such a vertex is $o(1)$. Markov's inequality implies that with high probability 
most vertices will have no neighbors. 

A similar calculation can actually show that the above probability is $\Omega \left({ e^{\zeta t_u/2}\over N^{1-\delta}} \right)$. 
Thereby, if $0< \delta <1$, then the expected degree of $u$ is of order $N^{\delta}$. A more detailed argument can 
show that the resulting random graph is too dense in the sense that the number of edges is \emph{no longer proportional} to the number of
vertices but grows much faster than that. 

{Of course to create a random graph with constant average degree one could have chosen a different scaling. 
For example, one could choose the threshold distance to be a positive constant $r$. In that case, the hyperbolic law of cosines 
(see Fact~\ref{Fact_I}) implies that two typical vertices which are close to the boundary of $\D$ would be within distance 
$r$ is their relative angle is of order $e^{-\zeta R}$. In other words, the probability that two typical vertices 
are adjacent would be proportional to $e^{-\zeta R}$. Hence, such a vertex would have constant expected degree if 
the total number of vertices scaled as $e^{\zeta R}$ (which is the volume of $\D$). However, such a choice would create a random 
model which is not suitable for complex networks. For example, the degree distribution would lose its heavy tail. In particular, 
the vertices which are close to the centre would not act as the high-degree 
hubs that hold the network together. Such vertices in general are fairly important as they give rise to short typical distances.}

The above heuristic also suggests that when $\delta =0$, the expected degree of vertex $u$ is proportional to $e^{\zeta t_u/2}$. 
In the present paper we focus on this case. 
{As we shall see in later in Claim~\ref{clm:density_approx}, the type of a vertex $u$ follows approximately the exponential distribution
$e^{-\alpha t_u}$. This implies that $e^{\zeta t_u/2}$ follows a power law degree distribution with exponent equal to $2\alpha/\zeta +1$.
This is precisely the effect of selecting the vertices on the hyperbolic plane of curvature $-\alpha^2$ and thereafter projecting them 
onto the plane of curvature $-\zeta^2$. The parameter $\alpha$ comes up as the parameter of the exponential distribution that 
(partially) determines the position of a vertex. One may view the quantity $e^{\zeta t_u/2}$ as the weight of vertex $u$. Let us point out 
that this is precisely the degree distribution in the \emph{Chung-Lu} model~\cite{ChungLu1+,ChungLuComp+}, in which 
vertices have associated weights that follow a power law distribution and any two vertices are joined with probability proportional to 
the product of these weights \emph{independently} of the other pairs. In our context, however, due to the underlying geometry 
dependencies are present. 
}

\subsection{The clustering coefficient}

The theme of this work is the study of clustering in $\G (N;\zeta, \alpha, \beta,\nu)$. The notion of clustering was introduced
by Watts and Strogatz~\cite{ar:WatStrog98}, as a { measure of the ``\emph{cliquishness of a typical neighborhood}'' (quoting
the authors). More specifically, it measures the expected density of the neighbourhood of a randomly chosen vertex}. 
In the context of biological or social networks,
this is expressed as the likelihood of two vertices that have a common neighbor to be joined with each other. More specifically, for each
vertex $v$ of a graph, the \emph{local clustering coefficient} $C(v)$ is
defined to be the density of the neighborhood of $v$. In~\cite{ar:WatStrog98}, the \emph{clustering coefficient} of a graph $G$, which
we denote by $C_1(G)$, is defined as the average of the local clustering coefficients over all vertices of $G$.
The clustering coefficient $C_1 (\G(N; \zeta, \alpha, \beta,\nu))$, as a function of $\beta$ is discussed in~\cite{ar:Krioukov}, where simulations
and heuristic calculations indicate that $C_1$ can be tuned by $\beta$. For the disc model, Gugelmann et al.~\cite{ar:Kosta} have shown
rigorously that this quantity is asymptotically with high probability bounded away from 0 when $0 < \zeta /\alpha <2$.

As we already mentioned, there has been significant experimental evidence which shows that many networks which arise in applications have
degree distributions that follow a power law usually with exponent between 2 and 3 (cf.~\cite{ar:StatMechs} for example). Also, such
networks are typically sparse with only a few nodes of very high degree which are the \emph{hubs} of the network. Thus, in the regime where
$\beta > 1$ and { $1 < \zeta/\alpha < 2$} the random graph $\G (N;\zeta,\alpha,\beta,\nu)$ appears to exhibit these characteristics.
In this work, we explore further the potential of this random graph model as a suitable model for complex networks focusing
on the notion of \emph{global clustering} and how this is determined by the parameters of the model.

A first attempt to define this notion was made by Luce and Perry~\cite{ar:LuPer49}, but it was rediscovered more recently by Newman,
Strogatz and Watts~\cite{ar:NewStroWat2001}. Given a graph $G$, we let $T=T(G)$ be the number
of triangles of $G$ and let $\Lambda=\Lambda(G)$ denote the number of \emph{incomplete triangles} of $G$; this is simply the number of
the (not necessarily induced) paths having length 2. Then the \emph{global clustering coefficient} $C_2 (G)$ of a graph $G$ is defined as
\begin{equation}\label{def_clustering}
C_2(G) := {3T(G) \over \Lambda(G)}.
\end{equation}
This parameter measures the likelihood that two vertices which share a neighbor are themselves adjacent.

The present work has to do with the %\emph{typical value}
value of  $C_2 (\G(N; \zeta ,\alpha , \beta,\nu))$. %, meaning that the position of the vertices satisfies some geometrical constraint (the rigorous definition will be given in a few lines).
Our results show exactly how clustering can be tuned by the parameters $\beta, \zeta$ and $\alpha$ only.
More precisely, our main result states that this undergoes an abrupt change as $\beta$ crosses the critical value 1.
\begin{Theorem} \label{thm:globalclustering}
Let $0 < \zeta / \alpha < 2$. If $\beta > 1$, then
$$
C_2 (\G (N; \zeta ,\alpha, \beta,\nu) ) \stackrel{p}{\rightarrow}
\left \{ \begin{array}{ll}
L_\infty(\beta, \zeta,\alpha), & \mbox{if $0 < \zeta/\alpha < 1$} \\
0, & \mbox{if $1 \leq \zeta /\alpha < 2$}
\end{array} \right .,
$$
where 
\[
\begin{split}
 &L_\infty (\beta, \zeta,\alpha) = \\
&{3\over 2}~\frac{(\zeta - 2\alpha)^2(\alpha - \zeta)}{(\pi C_\beta)^2} 
\int_{[0,\infty)^3}  e^{{\zeta\over 2}(t_u+t_v) + \zeta t_w} g_{t_u,t_v,t_w}(\beta,\zeta)
e^{-\alpha(t_u+t_v+t_w)} dt_u dt_v dt_w \ { \in (0,1]},
\end{split}
\]
with
$$g_{t_u,t_v,t_w}(\beta,\zeta) = \int_{[0,\infty)^2} \frac{1}{z_1^\beta+1}\frac{1}{z_2^\beta+1}\frac{1}
{\left(e^{{\zeta\over 2}(t_w-t_v)}z_1+e^{{\zeta\over 2}(t_w-t_u)}z_2\right)^\beta+1}
dz_1 dz_2$$
and
$C_\beta:= \frac{2}{\beta \sin(\pi/\beta)}$. 

\noindent
If $\beta \leq 1$, then
$$C_2 (\G (N; \zeta ,\alpha, \beta,\nu) ) \stackrel{p}{\rightarrow} 0. $$
\end{Theorem}
{Our results complement those of Krioukov et al.~\cite{ar:Krioukov}, whose estimates suggest that $C_1$ is 
tuned by $\beta$. From a different point of view, as it is suggested in~\cite{ar:BolRiordan}, the global clustering coefficient sometimes
gives more accurate information about the network. Bollob\'as and Riordan~\cite{ar:BolRiordan} give the example of a graph that is 
a double star with two adjacent centres 1 and 2, where each of them is joined with $n-1$ other vertices. Clearly this graph 
exhibits no clustering as it is a tree, nevertheless $C_1$ approaches 1 as $n$ grows, whereas $C_2$ vanishes. Furthermore, the 
study of $C_2$ is associated with subgraph counting which is a central topic in the theory of random graphs. Our analysis may need 
significant modification in order to be applied to $C_1$.} 

The fact that the global clustering coefficient asymptotically vanishes when $\zeta /\alpha \geq 1$ is due to the following: when
$\zeta / \alpha$ crosses 1 vertices of very high degree appear, which incur an abrupt increase on the number of incomplete triangles
with no similar increase on the number of triangles.
{In other words, vertices of high degree do not create many 
triangles. In our context, this is the case because the edges that are incident to a vertex of high degree are spread out and, as 
a result of this, most of their other endvertices are not close enough to each other. The same phenomenon appears in very 
different contexts. Namely, in random intersection graphs, where this was shown by Bloznelis~\cite{ar:Bloznelis2013}, as well as 
in a spatial model of preferential attachment introduced by M\"orters and Jacob~\cite{MoertJac15}. In both papers, it is shown that the
global clustering coefficient in that model is positive if and only if the second moment of the degree sequence is finite. 
Foudalis et al.~\cite{FJPS11} find this a natural feature of social networks: two neighbours of a very high degree vertex are 
somewhat less likely to be adjacent to each other that if they were neighbours of a low degree vertex.}

Recall that for a vertex $u \in \V$, its
\emph{type} $t_u$ is defined to be equal to $R- r_u$ where $r_u$ is the radius (i.e., its distance from the origin) of $u$ in $\D$. When 
$1\leq \zeta / \alpha < 2$, vertices of type larger than $R/2$ appear, which affect the tail of the degree sequence of 
$\G (N;\zeta, \alpha, \beta,\nu)$. For $\beta > 1$, it was shown in~\cite{ar:Foun13+} that when $1\leq \zeta /\alpha <2$ 
the degree sequence follows approximately a power law with exponent in $(2,3]$. More precisely, asymptotically as $N$ grows, 
the degree of a vertex $u \in \V$ conditional on its type follows a Poisson distribution with parameter equal to $K e^{\zeta t_u /2}$, 
where $K = K(\zeta, \alpha, \beta,  \nu) >0$. 
Also, a simple calculation reveals that most vertices have small types. This is made more precise in Claim~\ref{clm:density_approx}, where
we show that the types have approximately exponential density $e^{-\alpha t}$.
In fact, when $\zeta / \alpha < 1$, a.a.s.\ all vertices have type less than $R/2$ (cf.\ Corollary \ref{cor:x0}).

Let us consider, for example, more closely the case $\zeta = \alpha$, where the $N$ points are uniformly distributed on $\D$. 
In this case, the type of a vertex $u$ is approximately exponentially distributed  with density 
$\zeta e^{-\zeta t_u}$. Hence, there are about $N e^{- \zeta (R/2 - \omega (N))} \asymp e^{\zeta \omega (N)}$ vertices of type between 
$R/2 - \omega (N)$ and $R/2 + \omega (N)$; here $\omega (N)$ is assumed to be a slowly growing function. Now, each of these vertices 
has degree that is at least  $e^{{\zeta \over 2}~({R \over 2} - \omega (N))} = N^{1/2}e^{- \zeta \omega (N)/2}$. Therefore, 
these vertices' contribution to $\Lambda$ is at least $e^{\zeta \omega (N)} \times \left( N^{1/2}e^{- \zeta \omega (N)/2} \right)^2 = N$. 

Now, if vertex $u$ is of type less than $R/2 - \omega (N)$, its contribution to $\Lambda$ in expectation is proportional to 
\begin{equation} \label{eq:1stheuristic}
\int_{0}^{R/2 - \omega (N)} \left( e^{\zeta t_u /2} \right)^2 e^{-\zeta t_u } dt_u \asymp R.
\end{equation}
As most vertices are indeed of type less than $R/2 - \omega (N)$, it follows that these vertices contribute $R N$ on average to $\Lambda$. 

However, the amount of triangles these vertices contribute is asymptotically much smaller. Recall that for any two vertices $u,v$ the 
probability that these are adjacent is bounded away from 0 when $d(u,v) < R$. 
{ Consider three vertices $w, u, v$ which without loss of generality they satisfy $t_v < t_u < t_w < R/2 - \omega (N)$.  As we shall see later (cf. Fact~\ref{Fact_I}) 
having $d(u,v) < R$  is almost equivalent to having $\theta_{u,v} \lesssim e^{{\zeta \over 2} (t_u + t_v -R)}$. }
Since the relative angle between 
$u$ and $v$ is uniformly distributed in $[0,\pi]$, it turns out that the probability that $u$ is adjacent to $w$ is proportional to 
$e^{{\zeta \over 2} (t_w + t_u -R)}$; similarly, the probability that $v$ is adjacent to $w$ is proportional to 
$e^{{\zeta \over 2} (t_w + t_v -R)}$. Note that these events are independent. Now, conditional on these events, the relative angle 
between $u$ and $w$ is approximately uniformly distributed in an interval of length $e^{{\zeta \over 2} (t_w + t_u -R)}$. 
Similarly, the relative angle between $v$ and $w$ is approximately uniformly distributed in an interval of length 
$e^{{\zeta \over 2} (t_w + t_v -R)}$. Hence, the (conditional) probability that $u$ is adjacent to $v$ is bounded by a quantity that is 
proportional to $e^{{\zeta \over 2}(t_u + t_v)}/ e^{{\zeta \over 2}(t_v + t_w)} = e^{{\zeta \over 2}(t_u - t_w)}$. 
This implies that the probability that $u,v$ and $w$ form a triangle is proportional to $e^{{\zeta \over 2}t_w + \zeta t_u 
+ {\zeta \over 2}t_v} /N^2$. Averaging over the types of these vertices we have 
$${1\over N^2} \int_0^{R/2 - \omega (N)} \int_0^{t_w} \int_0^{t_u} 
e^{{\zeta \over 2}t_w + \zeta t_u + {\zeta \over 2}t_v - \zeta (t_v+t_u + t_w)} 
 dt_v dt_u dt_w \asymp {1\over N^2}. $$
Hence the expected number of triangles that have all their vertices of type at most $R/2  - \omega (N)$ is only proportional to $N$. 
Note that if we take $\alpha > \zeta$, then the above expression is still proportional to $N$, whereas 
(\ref{eq:1stheuristic}) becomes asymptotically constant giving contribution to $\Lambda$ that is also proportional to $N$. 
This makes the clustering coefficient be bounded away from 0 when $\zeta /\alpha < 1$. 
Our analysis will make the above heuristics rigorous and generalize them for all values of $\zeta /\alpha < 2$ and $\beta > 0$.

It turns out that the situation is somewhat different if we do not take into consideration high-degree vertices (or, equivalently, vertices 
that have large type). For any fixed $t>0$, we will consider the global clustering coefficient of the subgraph of 
$\G (N; \zeta, \alpha,\beta,\nu)$ that is induced by those vertices that have type at most $t$. We will denote this by $\widehat{C_2}(t)$.
We will show that when $\beta >1$ then for all $0 < \zeta /\alpha < 2$, the quantity $\widehat{C_2}(t)$ remains bounded away
from 0 with high probability. Moreover, we determine its dependence on $\zeta, \alpha, \beta$.
\begin{Theorem} \label{thm:typicalglobalclustering}
Let $0 < \zeta / \alpha < 2$ and let $t>0$ be fixed. If $\beta > 1$, then
%$$ \widehat{C_2} (t) = \Theta_C (1). $$
%In particular, there exist positive constants $c_1, c_2$ that depend on $\zeta$ and $\alpha$ such that for any $\beta >1$ a.a.s.
%$$c_1 g(\beta) <  \widehat{C_2} (\G (N; \zeta ,\alpha, \beta) )  < c_2 g (\beta), $$
%where $g (\beta ) = ??$.
\begin{equation}\label{eq:Lbeta}
\widehat{C_2} (t) \stackrel{p}{\rightarrow} L(t;\beta,\zeta,\alpha),
\end{equation}
where%More precisely, denote by $D_t:=[0,R/2-\omega(N)]^3$, by $ D_z:=[0,\infty)^2$ and $C_\beta:= \frac{2}{\beta \sin(\pi/\beta)}$. Set
\[
L(t;\beta,\zeta,\alpha):=%{6\over \left(\pi C_{\beta}\right)^2}
\frac{ 6\int_{[0,t)^3}  e^{{\zeta\over 2}(t_u+t_v) + \zeta t_w} g_{t_u,t_v,t_w}(\beta,\zeta)
e^{-\alpha(t_u+t_v+t_w)} dt_u dt_v dt_w }{\left(\pi C_{\beta}\right)^2\int_{[0,t)^3}
e^{{\zeta\over 2}(t_u+t_v) + \zeta t_w} e^{-\alpha(t_u+t_v+t_w)} dt_u dt_v dt_w},
\]
where $g_{t_u,t_v,t_w}(\beta,\zeta)$ and $C_\beta$ are as in Theorem~\ref{thm:globalclustering}.
\end{Theorem}

The most involved part of the proofs, which may be of independent interest, has to do with counting triangles in 
$\G (N; \zeta , \alpha, \beta,\nu)$, that is, with estimating $T(\G (N; \zeta, \alpha, \beta,\nu))$.  In fact, most of our effort is devoted to the calculation of the probability that three vertices form
a triangle. Thereafter, a second moment argument, together with the fact that the degree of high-type vertices is concentrated around
its expected value, implies that
$T(\G (N; \zeta, \alpha, \beta,\nu))$ is close to its expected value (see Section~\ref{sect:2nd_moment}).
%Our analysis shows that the number of triangles in $\G (N; \zeta, \alpha, \beta)$ exhibits a sudden jump around $\zeta / \alpha = 1$.

The paper is organized as follows.
In Section~\ref{sect:proofs} we state a series of results that imply Theorem~\ref{thm:globalclustering}.
Section~\ref{sect:2nd_moment} is mainly devoted to showing that the random variables counting the number 
of \emph{typical incomplete and complete triangles} (i.e., whose vertices have type at most $R/2-\omega(N)$, for a suitable growing function $\omega(N)$) %in the system
are concentrated around their expected values. We find precise (asymptotic) expressions for these values in
Sections~\ref{sect:proof_prop_21},~\ref{sect:probab_triangles} (where we deduce Theorem~\ref{thm:typicalglobalclustering})
and~\ref{sect:proof_prop_22}. Section~\ref{sect:proof_prop_26} takes care
of the atypical (i.e., non-typical) case, together with the calculations shown in Appendix~\ref{sect_aux_calculations}.
\medskip

\section{Triangles and concentration: %proofs of Theorems~\ref{thm:globalclustering} and~\ref{thm:typicalglobalclustering}
proof of Theorem~\ref{thm:globalclustering}}\label{sect:proofs}

The main ingredient of the proofs of Theorems~\ref{thm:globalclustering} and~\ref{thm:typicalglobalclustering} is a collection
of concentration results regarding the number of triangles as well as that of incomplete triangles.
For a function $\omega : \mathbb{N} \rightarrow \mathbb{N}$ such that $\omega (N) \rightarrow \infty$ as $N \rightarrow \infty$, we call a
vertex $u$ \emph{typical} if $t_u \leq R/2 - \omega (N)$. The function $\omega$ grows slowly enough so that our calculations work. 
For example, we may set { $\omega (N) = \ln \ln \ln (N)$}.
We consider two classes of triangles, namely those which consist of typical vertices and those that contain at least one vertex that
is not typical.

In particular, we introduce the random variables $\widehat{T}$ and $\widehat{\Lambda}$ which denote the numbers of triangles and
incomplete triangles, respectively, with all their 3 vertices being typical. We also introduce the random variables $\widetilde{T}$ and
 $\widetilde{\Lambda}$ which denote the numbers of triangles and incomplete triangles, respectively, which have at least one vertex that is
 not typical, or, as we shall be saying, \emph{atypical}. Hence,
$$T=\widehat{T} + \widetilde{T}, \ \mbox{and} \ \Lambda = \widehat{\Lambda} + \widetilde{\Lambda}. $$

We now give a series of propositions that describe how the expected values of $\widehat{T}$ and $\widehat{\Lambda}$ vary
according to the parameters $\zeta, \alpha$ and $\beta$.

\begin{Proposition} \label{prop:incomplete_triangles}
Let $0<\zeta /\alpha < 2$. Then the following hold:
\begin{enumerate}
\item[(i)] for $\beta > 1$
\begin{equation}\label{expected_D}
\E (\widehat{\Lambda } )\asymp
\left \{ \begin{array}{ll}
%{N \choose 3} \Theta \left ( \frac{1}{N^2}\right )=
 N ,  & \textnormal{ if }\frac{\zeta}{\alpha}<1\\
%{N \choose 3} \Theta \left ( \frac{R}{N^2}\right )=
 RN,  & \textnormal{ if }\frac{\zeta}{\alpha}=1\\
%{N \choose 3} \Theta \left ( \frac{N^{\gamma}e^{-\gamma \zeta \omega(N)}}{N^2}\right )=
 N^{2-\alpha/\zeta}e^{-( \zeta-\alpha)\omega(N)},  & \textnormal{ if }\frac{\zeta}{\alpha}>1
\end{array} \right ..
\end{equation}
\item[(ii)] for $\beta = 1$
\begin{equation}\label{expected_D_1}
\E (\widehat{\Lambda } )\asymp
\left \{ \begin{array}{ll}
 R^2 N,  & \textnormal{ if }\frac{\zeta}{\alpha}<1\\
 R^3 N,  & \textnormal{ if }\frac{\zeta}{\alpha}=1\\
 R^2 N^{2-\alpha/\zeta}e^{-( \zeta-\alpha) \omega(N)},  & \textnormal{ if }\frac{\zeta}{\alpha}>1
\end{array} \right ..
\end{equation}
\item[(iii)] for $\beta < 1$
\begin{equation}\label{expected_D<1}
\E (\widehat{\Lambda } )\asymp
\left \{ \begin{array}{ll}
 N^{3-2\beta},  & \textnormal{ if }\frac{\beta \zeta}{\alpha}<1\\
R N^{3-2\beta},  & \textnormal{ if }\frac{\beta \zeta}{\alpha}=1\\
N^{3-\beta-\alpha/\zeta}e^{(\alpha-\beta\zeta) \omega(N)},  & \textnormal{ if }\frac{\beta \zeta}{\alpha}>1
\end{array} \right ..
\end{equation}
\end{enumerate}
\end{Proposition}
The following proposition is the counterpart of the above for triangles.
\begin{Proposition} \label{prop:complete_triangles}
Let $0<\zeta /\alpha < 2$. % and let $\gamma$ be as in Proposition~\ref{prop:incomplete_triangles}.
Then the following hold:
\begin{enumerate}
\item[(i)] for $\beta > 1$
\begin{equation}\label{expected_T}
\E (\widehat{T})\asymp
\left \{ \begin{array}{ll}
%{N \choose 3} \Theta \left ( \frac{1}{N^2}\right )=
\E (\widehat{\Lambda}),  & \textnormal{ if }\frac{\zeta}{\alpha}<1\\
o \left( \E (\widehat{\Lambda}) \right),  & \textnormal{ if } \frac{\zeta}{\alpha} \geq 1
%{N \choose 3} \Theta \left ( \frac{R}{N^2}\right )=
%%%%%RN,  & \textnormal{ if }\frac{\zeta}{\alpha}=1\\
%{N \choose 3} \Theta \left ( \frac{N^{\gamma}e^{-\gamma \zeta \omega (N)}}{N^2}\right )=
%%%%%N^{2-\alpha/\zeta}e^{-( \zeta-\alpha) \omega(N)},  & \textnormal{ if }\frac{\zeta}{\alpha}>1
\end{array} \right ..
\end{equation}
\item[(ii)] for $\beta \leq 1$
\begin{equation}\label{expected_T_1}
\E (\widehat{T}) = o \left( \E (\widehat{\Lambda}) \right).
\end{equation}
 \end{enumerate}
\end{Proposition}

The following proposition states that the random variables $\widehat{\Lambda}$ and $\widehat{T}$ are concentrated around their
expected values $\E (\widehat{\Lambda })$ and $\E (\widehat{T })$, respectively.
\begin{Proposition}\label{prop_concentration}
Let $0<\zeta/\alpha<2$. Then for all $\beta >0$
\[
\widehat{\Lambda }=\E (\widehat{\Lambda }) \bigl ( 1+o_p(1)\bigr ) \]
and for $0 < \zeta /\alpha < 1$ and $\beta >1$
\[\widehat{T }=\E (\widehat{T }) \bigl ( 1+o_p(1)\bigr ).
\]
\end{Proposition}
The next result deals with the number of \emph{atypical} triangles. 
Note that, since each triangle contains three incomplete triangles, we always have $\widetilde{T} \leq \widetilde{\Lambda}/3$.
\begin{Proposition}\label{prop_asymptotic_atypical_triangles}
%Let $\beta'$ be defined as in (\ref{eq:beta_prime}).
For $0<\zeta/\alpha<2$ and any $\beta > 0$ we have:
\medskip 

\noindent
if $\zeta /\alpha < 1$, then
\[
\widetilde{T} \leq {1\over 3}\widetilde{\Lambda} = o_p(1);
\]
{ thus, a.a.s.\ we have $ \widetilde{T}=\widetilde{\Lambda}=0$.}

If $\zeta /\alpha \geq 1$, then
\[ \widetilde{T} = o_p(\E (\widehat \Lambda )). \]
\end{Proposition}

\subsection{Proof of Theorem \ref{thm:globalclustering}}
We begin with the case where $0<\zeta/\alpha < 1$ and $\beta > 1$.
From Proposition~\ref{prop_concentration}, it follows that $\widehat{T}$ and $\widehat{\Lambda}$ are concentrated around their expected
values $\E (\widehat{T})$ and $\E (\widehat{\Lambda} )$, respectively.
Also, the first part of Proposition~\ref{prop_asymptotic_atypical_triangles} implies that
$\widetilde{T}, \widetilde{\Lambda } = o_p (1)$. Thus, when $0<\zeta/\alpha< 1 $ and
$\beta > 1$ we have
\[
{T \over \widehat{T}} = 1 + o_p(1) \quad \textnormal{and}\quad {\Lambda \over \widehat{\Lambda}} = 1 + o_p(1).
\]
Now, the first part of the theorem follows from~(\ref{expected_T}). The value of $L_\infty(\beta, \zeta,\alpha)$ will be deduced in 
Section~\ref{sec:proof_of_theorem_2}.
%together with Propositions~\ref{prop:za1} and~\ref{prop:atypical_cherries_conc}

For $\zeta /\alpha \geq 1$ and $\beta>1$, the statement of the theorem follows from the fact that $\widehat{T} + \widetilde{T} = 
o_p (\E (\widehat{\Lambda}))$ (cf.~(\ref{expected_T}) and Proposition~\ref{prop_asymptotic_atypical_triangles})
together with the fact that $\widehat{\Lambda} = \E (\widehat{\Lambda}) (1+o_p(1))$ (cf. Proposition~\ref{prop_concentration}).
The above argument also works for $\beta \leq 1$.

\section{Preliminary results}\label{sect:prel_results}
For two vertices $u,v$ of types $t_u$ and $t_v$, respectively, we define
\begin{equation}\label{def_A}
A_{u,v}:= {N \over \nu}~e^{-\frac{\zeta}{2}(t_u+t_v)}{ =e^{-\frac{\zeta}{2}(R-t_u-t_v)}}.
\end{equation}
For two vertices $u,v$ we write $u\sim v$ to indicate that $u$ is adjacent to $v$. 
The following lemma is a special case of Lemma 2.4 in~\cite{ar:Foun13+}.
\begin{Lemma} \label{lemma_2.4_evolution}
Let $\beta >0$. There exists a constant $C_{\beta}>0$ such that uniformly for all distinct pairs $u,v \in \V$ such that
{ $t_u+ t_v < R - 2\omega (N)$} we have
$$ \P \left( u \sim v \mid t_u,t_v \right) =
\left \{
\begin{array}{ll}
(1+o(1)){C_{\beta}\over A_{u,v}}~, & \mbox{if $\beta>1 $} \\[1.5ex]
%& \\
(1+o(1)){C_\beta \ln A_{u,v} \over A_{u,v}},  & \mbox{if $\beta = 1 $} \\[1.5ex]
%& \\
(1+o(1)){C_\beta \over A_{u,v}^{\beta}}, & \mbox{if $\beta < 1 $}
\end{array}
\right ..$$
In particular,
$$ C_\beta =
\left \{
\begin{array}{ll}
{2\over \beta}~\sin^{-1} \left( {\pi \over \beta}\right), & \mbox{if $\beta >1 $} \\[1.1ex]
{2 \over \pi}, & \mbox{if $\beta = 1 $} \\ [1.05ex]
{1 \over \sqrt{\pi}}~{ \Gamma \left({1-\beta \over 2} \right) \over \Gamma\left(1 -{\beta \over 2} \right)}, & \mbox{if $\beta < 1$}
\end{array}
\right .$$
\end{Lemma}

We will need the following fact, which expresses the hyperbolic distance between two points of given types as a function of
their~\emph{relative} angle.
For two points $u$ and $v$ in $\D$, we denote by $\theta_{u,v} \in [0,\pi]$ the relative angle between $u$ and $v$ (with the center of $\D$
being the reference point).
\begin{Lemma} \label{lem:probs}
Let $u,v \in \D$ be two distinct points of types $t_u$ and $t_v$, respectively.
Moreover, set
\begin{equation}\label{theta_c}
\bar{\theta}_{u,v}:=\left (e^{-2\zeta (R-t_u)}+e^{-2\zeta (R-t_v)}\right ) ^{1/2}=\nu ^2 \left \{ \frac{e^{2\zeta t_u}}{N^4}+\frac{e^{2\zeta t_v}}{N^4} \right \}^{1/2}
\end{equation}
and assume that $\theta_{u,v} \gg \bar{\theta}_{u,v}$. Then
$$p_{u,v}= \frac{1}{C\, A^\beta_{u,v}\sin^\beta (\theta_{u,v}/2)+1},$$
where $C=1+o(1)$, uniformly for any $t_u, t_v < R/2 - \omega (N)$.
\end{Lemma}
\begin{proof}
We will use the following fact, which relates the hyperbolic distance between two points of given types with their relative angle.
\begin{Fact}[Lemma 2.3,~\cite{ar:Foun13+}] \label{Fact_I}
For every $\beta>0$ and $0<\zeta/\alpha<2$, let $u,v$ be two distinct points of $\mathcal{D}_R$.
If $\bar{\theta}_{u,v}\ll \theta_{u,v}\leq \pi$, then
\begin{equation}\label{d_uv_formula}
d(u,v)=2R-(t_u+t_v)+\frac{2}{\zeta} \log \sin \left (\frac{\theta_{u,v}}{2} \right )+\Theta  \left ( \left ( \frac{\bar{\theta}_{u,v}}{\theta_{u,v}}\right )^2 \right),
\end{equation}
uniformly for all $u,v$ with $t_u,t_v < R/2 - \omega (N)$.
\end{Fact}
% Using Claim 2.5 from~\cite{ar:Foun13+}, we have that  $\bar{\theta}_{u,v} = o(\tilde{\theta}_{u,v})$ uniformly for any
% $t_u, t_v < R/2 - \omega (N)$.
Hence, whenever $\theta_{u,v} \gg \bar{\theta}_{u,v}$ we have
\[
e^{\beta \frac{\zeta}{2}(d(u,v)-R)}=C e^{\beta \frac{\zeta}{2}\bigl (R-(t_u+t_v)\bigr )}e^{\beta \log \sin (\theta_{u,v}/2)},
\]
where $C=1+o(1)$, uniformly for any $t_u, t_v < R/2 - \omega (N)$; { note that the error term does depend on $u$ and $v$, 
but is uniformly bounded by a function that is $o(1)$.} 
The lemma follows from the definition of $p_{u,v}$ together with (\ref{def_A}).
\end{proof}

Finally, note that the density function of the type of a vertex $u$ is
\begin{equation}\label{def_bar_rho}
\bar{\rho}_N(t_u):=\alpha \frac{\sinh (\alpha (R-t_u))}{\cosh(\alpha R)-1}.
\end{equation}
We will be using this density quite frequently, as we will often be conditioning on the types of the vertices under consideration.
It is not hard to see that this density can be approximated by an exponential density.
\begin{Claim} \label{clm:density_approx} %If $u$ is a vertex, then uniformly for any $t_u < R/2$  we have
For any vertex $u\in \D$, uniformly for $t_u<R$ we have
\[
\bar{\rho}_N(t_u)\leq \bigl (1+o(1)\bigr )\alpha e^{-\alpha t_u}.
\]
Moreover, uniformly for all $0<t_u\leq \lambda R $, where $0<\lambda<1$ we have
$$ \bar{\rho}_N (t_u) = \bigl (1+o(1)\bigr ) \alpha e^{-\alpha t_u} . $$
\end{Claim}
\begin{proof}
Starting from the definition we get:
\[
\begin{split}
\bar{\rho}_N(t_u) & =\alpha \frac{\sinh (\alpha (R-t_u))}{\cosh(\alpha R)-1}\leq
2\alpha \frac{\sinh (\alpha (R-t_u))}{e^{\alpha R}} (1+o(1))\\
& = \alpha \frac{e^{\alpha (R-t_u)}}{e^{\alpha R}}\bigl (1+o(1)\bigr )= \alpha e^{-\alpha t_u}\bigl (1+o(1)\bigr ).
\end{split}
\]
The lower bound under the condition $t_u\leq \lambda R$ can be proven arguing as follows:
\[
\sinh (\alpha (R-t_u))=\cosh(\alpha (R-t_u))-e^{-\alpha(R-t_u)}= \cosh(\alpha (R-t_u))\bigl ( 1-o(1)\bigr),
\]
uniformly over $t_u\leq \lambda R$.

Furthermore, we have:
\[
\frac{\cosh(\alpha (R-t_u))}{\cosh(\alpha R)-1} \geq \frac{\cosh(\alpha (R-t_u))}{\cosh(\alpha R)}\bigl ( 1+o(1)\bigr ).
\]
Hence,
\[
\begin{split}
\bar{\rho}_N(t_u)
& =\alpha \frac{\sinh(\alpha (R-t_u))}{\cosh(\alpha R)-1} \geq \alpha \frac{\cosh(\alpha (R-t_u))}{\cosh(\alpha R)}\bigl (1+o(1)\bigr )\\
& \geq \alpha \frac{e^{\alpha (R-t_u)}}{e^{\alpha R}}\bigl (1+o(1)\bigr )= \alpha e^{-\alpha t_u}\bigl (1+o(1)\bigr ).
\end{split}
\]
At this point the statement follows.
\end{proof}
Now, Claim~\ref{clm:density_approx} together with the fact that $N\asymp e^{\zeta R/2}$ imply the following.
\begin{Corollary}\label{cor:x0}
For any $\zeta , \alpha > 0$, conditional on $u \in \V$ we have
\[
\P \left (t_u \leq \frac{\zeta}{2\alpha}R + \omega(N) \right )=1-o (N^{-1}).
\]
In particular, a.a.s.\ for all $u \in \V$
$$ t_u \leq \frac{\zeta}{2\alpha}R + \omega(N). $$
\end{Corollary}

\section{On incomplete triangles: proof of Proposition~\ref{prop:incomplete_triangles}}\label{sect:proof_prop_21}

In this section, we calculate the expected number of incomplete triangles with all their vertices having type less than $R/2 - \omega (N)$, that is,
being \emph{typical vertices}. In particular, we give the proofs of (\ref{expected_D}), (\ref{expected_D_1}) and (\ref{expected_D<1}).

%For every triple of distinct vertices $u,v,w$, recall that $\Lambda(u,v;w)$ denotes the event that the vertices $u,v,w$ form an incomplete triangle with $w$ being the pivoting vertex.

Let us introduce the following notation: for every triple of vertices $u,v,w$ such that $t_u,t_v,t_w<R/2-\omega(N) $ (being $\omega(N)$ an arbitrary slowly growing function),
we denote by $\Lambda(u,v;w)$ the event that the triple $u,v,w$ forms an incomplete triangle
\emph{pivoted} at $w$. In other words, $u,v$ and $w$ form a path of length 2 with $w$ being the middle vertex.
{ Note that there are exactly $3{N \choose 3}$ such choices.}

Similarly, for every triple of vertices $u,v,w$ such that $t_u,t_v,t_w<R/2-\omega(N) $, we denote by $\Delta(u,v,w)$ the event that the triple $u,v,w$ forms a triangle.
%We have
%{ 
%\begin{equation} \label{eq:TypLambda_Exp}
%\begin{split}
%\E (\widehat{\Lambda}^2)
%& =\E \left [ \left ( \sum_{\{u,v,w\}}\mathbf{1}_{\{\Lambda(u,v;w)\}}\right )^2\right ]\\
%& = \E \left [ \sum_{\{u,v,w\}}\mathbf{1}_{\{\Lambda(u,v;w)\}}^2 + \sum_{\{u_2,v_2,w_2\}\neq \{u_1,v_1,w_1\}}\mathbf{1}_{\{\Lambda(u_1,v_1;w_1)\}} \mathbf{1}_{\{\Lambda(u_2,v_2;w_2)\}}\right ]\\
%& = \E (\widehat{\Lambda} )+ \sum_{\{u_2,v_2,w_2\}\neq \{u_1,v_1,w_1\}}
%\E \left ( \mathbf{1}_{\{\Lambda(u_1,v_1;w_1)\}} \mathbf{1}_{\{\Lambda(u_2,v_2;w_2)\}}\right ),
%\end{split}
%\end{equation}
%where the double sum ranges over ordered pairs of distinct subsets.
%}
%Similarly, we write
%\[
%\E (\widehat{T}^2)=\E (\widehat{T})+ \sum_{\{u_1,v_1,w_1\}}\sum_{\{u_2,v_2,w_2\}}
%\E \left ( \mathbf{1}_{\{\Delta(u_1,v_1,w_1)\}} \mathbf{1}_{\{\Delta(u_2,v_2,w_2)\}}
%%\Delta(u_1,v_1,w_1)\Delta(u_2,v_2,w_2)
%\right ).
%\]
At this point we introduce the following parameters, which will be used throughout the paper:
\begin{equation}\label{eq:beta_prime}
\beta':=
\left \{
\begin{array}{ll}
1, & \textnormal{ if }\beta\geq 1\\
\beta, & \textnormal{ if }\beta< 1
\end{array}
\right .,
\end{equation}
as well as
\begin{equation}\label{eq:def_delta_1}
\delta=\delta(\beta,1):=
\left \{
\begin{array}{ll}
1 & \textnormal{ if } \beta=1\\
0 & \textnormal{ otherwise}
\end{array}
\right ..
\end{equation}
We start from a simple observation: conditional on the types of $u,v,w$, the presence of the edges $\{uw\} $ and $\{vw\} $ are
independent events. Thereby, we have
\[
\P \bigl ({\Lambda}(u,v;w) \mid t_u,t_v,t_w\bigr )=\P \bigl ( u\sim w \mid t_u, t_w\bigr )\P \bigl ( v\sim w \mid t_v, t_w\bigr ).
\]
{ This implies:
\begin{equation}\label{eq_lambda_general}
\P \bigl ( {\Lambda}(u,v;w) \mid t_u,t_v,t_w\bigr ) \asymp
%
%\begin{cases}
%\frac{1}{A_{u,w}A_{v,w}} = \frac{e^{\zeta t_w+(\zeta/2)(t_u+t_v)}}{N^2}, & \textnormal{if }\beta>1\\[1.5ex]
% \frac{\ln(A_{u,w})\ln(A_{v,w})}{A_{u,w}A_{v,w}}= & \\[1.5ex]
% \quad = \frac{\ln(A_{u,w})\ln(A_{v,w})}{N^2}e^{\zeta t_w+(\zeta/2)(t_u+t_v)},& \textnormal{if }\beta=1\\[1.5ex]
%\frac{1}{A_{u,w}^\beta A_{v,w}^\beta} = \frac{e^{\beta \zeta t_w+(\beta \zeta/2)(t_u+t_v)}}{N^{2\beta}},  & \textnormal{if }\beta<1
%\end{cases}.
%
\frac{(\ln(A_{u,w})\ln(A_{v,w}))^\delta}{A_{u,w}^{\beta'} A_{v,w}^{\beta'}} =\frac{(\ln(A_{u,w})\ln(A_{v,w}))^\delta \, e^{\beta' \zeta t_w+(\beta' \zeta/2)(t_u+t_v)}}{N^{2\beta'}}.
\end{equation}
Hence, to determine the expected value of $\widehat{\Lambda}$ we need to integrate the above expressions with respect to $t_u, t_v, t_w$ (using the density given by Claim~\ref{clm:density_approx}) and subsequently multiply the outcome by $3{N \choose 3}$.

More precisely, we have
\[
\begin{split}
\P ({\Lambda}(u,v;w)) & \asymp N^{-2\beta'}  \int_0^{R/2-\omega(N)}\int_0^{R/2-\omega(N)} e^{(\beta'\zeta/2-\alpha)t_u}e^{(\beta'\zeta/2-\alpha)t_v} \times \\
& \times \left (\int_0^{R/2-\omega(N)} (R-t_u-t_w)^\delta  (R-t_v-t_w)^\delta  e^{(\beta'\zeta-\alpha)t_w}dt_w \right )dt_udt_v.
%
%\\
%& \asymp  N^{-2\beta'} \int_0^{R/2-\omega(N)} e^{(\beta'\zeta-\alpha)t_w}dt_w .
\end{split}
\]
Now observe that a lower bound for the above integral can be obtained by reducing the domain of integration of $t_u$ and $t_v$ 
up to $R/4$.  Thus, if $t_u,t_v \leq R/4$, we have  $(R-t_u-t_w)^\delta  (R-t_v-t_w)^\delta \geq (3R/4-t_w)^{2\delta}$.
This simplifies the calculations considerably, yielding:
\[
\begin{split}
\P ({\Lambda}(u,v;w)) & \gtrsim  N^{-2\beta'} \left ( \int_0^{R/4} e^{(\beta'\zeta/2-\alpha)t_u}dt_u\right )^2 \int_0^{R/2-\omega(N)} (3R/4-t_w)^{2\delta} e^{(\beta'\zeta-\alpha)t_w}dt_w \\
& \stackrel{t_w < R/2}{\gtrsim}R^{2\delta} N^{-2\beta'} \int_0^{R/2-\omega(N)} e^{(\beta'\zeta-\alpha)t_w}dt_w .
\end{split}
\]
We obtain an upper bound using $(R-t_u-t_w)^\delta  (R-t_v-t_w)^\delta < (R-t_w)^{2\delta}$.
This yields
\[
\begin{split}
\P ({\Lambda}(u,v;w)) & \lesssim  N^{-2\beta'} \left ( \int_0^{R/2-\omega(N)} e^{(\beta'\zeta/2-\alpha)t_u}dt_u\right )^2 \int_0^{R/2-\omega(N)} (R-t_w)^{2\delta} e^{(\beta'\zeta-\alpha)t_w}dt_w \\
& \lesssim R^{2\delta} N^{-2\beta'} \int_0^{R/2-\omega(N)} e^{(\beta'\zeta-\alpha)t_w}dt_w .
\end{split}
\]
}
At this point we need a case distinction according to the value of $\beta'\zeta/\alpha$, leading to
\[
\P  ({\Lambda}(u,v;w)) \asymp
\left \{ \begin{array}{ll}
{ R^{2\delta}}N^{-2\beta'}, & \textnormal{if }\beta'\zeta/\alpha<1\\
{ R^{1+2\delta}} N^{-2\beta'}, & \textnormal{if }\beta'\zeta/\alpha=1\\
{ R^{2\delta}}N^{-\beta'-\alpha/\zeta}e^{-(\beta'\zeta-\alpha)\omega(N)}, & \textnormal{if }\beta'\zeta/\alpha>1
\end{array} \right ..
\]
Multiplying the above expressions by $3{N \choose 3}$, we deduce the statement for $\E(\widehat{\Lambda})$.

\section{Proof of Proposition \ref{prop_concentration}}\label{sect:2nd_moment}
In this section, we show the concentration of the random variables $\widehat{\Lambda}$ and $\widehat{T}$ around
$\E (\widehat{\Lambda}) $ and $\E (\widehat{T})$, respectively.
We will do so by bounding their second moment.
For every triple of distinct vertices $u,v,w$, recall that $\Lambda(u,v;w)$ denotes the event that the vertices $u,v,w$ form an incomplete triangle with $w$ being the pivoting vertex.
{For a triple of distinct vertices $u,v$ and $w$ we denote by $\{u,v;w\}$ the set $\{ \{u,v\},w \}$}. 
%Let us introduce the following notation: for every triple of vertices $u,v,w$ such that $t_u,t_v,t_w<R/2-\omega(N) $ (being $\omega(N)$ an arbitrary slowly growing function),
%we denote by $\Lambda(u,v;w)$ the event that the triple $u,v,w$ forms an incomplete triangle
%\emph{pivoted} at $w$. In other words, $u,v$ and $w$ form a path of length 2 with $w$ being the middle vertex.
%{ Note that there are exactly $3{N \choose 3}$ such choices.}
%
%Similarly, for every triple of vertices $u,v,w$ such that $t_u,t_v,t_w<R/2-\omega(N) $, we denote by $\Delta(u,v,w)$ the event that the triple $u,v,w$ forms a triangle.
We have
{ 
\begin{equation} \label{eq:TypLambda_Exp}
\begin{split}
\E (\widehat{\Lambda}^2)
& =\E \left [ \left ( \sum_{\{u,v;w\}}\mathbf{1}_{\{\Lambda(u,v;w)\}}\right )^2\right ]\\
& = \sum_{\{u,v;w\}}\E (\mathbf{1}_{\{\Lambda(u,v;w)\}}^2) + \sum_{\{u_2,v_2;w_2\}\neq \{u_1,v_1;w_1\}}\E (\mathbf{1}_{\{\Lambda(u_1,v_1;w_1)\}} \mathbf{1}_{\{\Lambda(u_2,v_2;w_2)\}})\\
& = \E (\widehat{\Lambda} )+ \sum_{\{u_2,v_2;w_2\}\neq \{u_1,v_1;w_1\}}
\E \left ( \mathbf{1}_{\{\Lambda(u_1,v_1;w_1)\}} \mathbf{1}_{\{\Lambda(u_2,v_2;w_2)\}}\right ),
\end{split}
\end{equation}
where the double sum is taken over pairs of distinct triples of vertices.
}
{Similarly, with $\{ u,v,w \}$ denoting a set of pairwise distinct vertices, we write
\[
\E (\widehat{T}^2)=\E (\widehat{T})+ \sum_{\{u_1,v_1,w_1\} \neq \{u_2,v_2,w_2\}}
\E \left ( \mathbf{1}_{\{\Delta(u_1,v_1,w_1)\}} \mathbf{1}_{\{\Delta(u_2,v_2,w_2)\}}
%\Delta(u_1,v_1,w_1)\Delta(u_2,v_2,w_2)
\right ).
\]
}
%At this point we introduce the following parameter, which will be used throughout the paper:
%\begin{equation}\label{eq:beta_prime}
%\beta':=
%\left \{
%\begin{array}{ll}
%1, & \textnormal{ if }\beta\geq 1\\
%\beta, & \textnormal{ if }\beta< 1
%\end{array}
%\right ..
%\end{equation}
In order to show concentration of $\widehat{\Lambda} $, we will show the following statement.
\begin{Lemma}\label{ED2} For any $0 < {\zeta \over \alpha} < 2$ and $\beta >0$ we have
\[
\E (\widehat{\Lambda}^2)=\E^2 (\widehat{\Lambda}) \bigl ( 1+o(1)\bigr ).
\]
\end{Lemma}
\begin{proof}
To evaluate the second moment of $\widehat{\Lambda}$, we need to control the dependencies between every two triples of vertices $u_1,v_1,w_1 $
and $u_2,v_2,w_2 $, assuming that the vertices $w_1$ and $w_2$ are the pivoting vertices. More precisely, we have { eight} possible configurations:
\begin{itemize}
\item[1.] $ \{u_1,v_1,w_1\} \cap \{u_2,v_2,w_2 \} =\emptyset$;
\item[2.] %$\{u_1,v_1,w_1\} \cap \{u_2,v_2,w_2 \} =\{u_1\}, \{v_1,w_1\} \cap \{v_2,w_2 \} =\emptyset $ (or $v_1=v_2$ with $\{u_1,w_1\} \cap \{u_2,w_2 \} =\emptyset $);
    $u_1=u_2$ with $ \{v_1,w_1\} \cap \{v_2,w_2 \} =\emptyset $
(or, analogously, $v_1=v_2$ and $\{u_1,w_1\} \cap \{u_2,w_2 \} =\emptyset $);
\item[3.] $w_1=w_2$ and $\{u_1,v_1\} \cap \{u_2,v_2 \} =\emptyset $;
\item[4.] $u_1=u_2, v_1=v_2, w_1\neq w_2$;
\item[5.] $w_1 = u_2$, $v_1 = w_2$ and $v_1 \not = u_1$;
\item[6.] $w_1=w_2,u_1=u_2, v_1\neq v_2 $ (or $w_1=w_2,v_1=v_2, u_1\neq u_2$).
{ 
\item[7.] $ v_1=w_2$ and $\{u_1,w_1\}\cap \{u_2,v_2\}=\emptyset $.
\item[8.] $w_1=u_2, v_1=v_2 $ and $ u_1\neq w_2$.
}
\end{itemize}
We denote by $\widehat{\Lambda}_i$ the contribution in (\ref{eq:TypLambda_Exp}) of the terms that correspond to case $i$ for $i=1,\ldots, 8$.
Our aim is to show that
\begin{equation} \label{eq:ToShowI}
\widehat{\Lambda}_1  = \E^2 \left( \widehat{\Lambda} \right)(1+o_p(1))
\end{equation}
and for each $i=2,\ldots, 8$ we have
\begin{equation} \label{eq:ToShowII}
\widehat{\Lambda}_i  = o_p\left( \E^2 \left( \widehat{\Lambda} \right)\right).
\end{equation}

Note that Proposition~\ref{prop:incomplete_triangles} implies that
\begin{equation} \label{eq:LowTypLambda}
\E (\widehat{\Lambda}) \gtrsim
\begin{cases}
N^{3-2 \beta'}, & \mbox{if $\beta'\zeta/\alpha \leq 1$} \\
N^{3-\beta' - \alpha /\zeta} e^{-(\beta' \zeta - \alpha)\omega (N)}, & \mbox{if $\beta'\zeta/\alpha > 1$}
\end{cases},
\end{equation}
where $\beta'$ is defined in (\ref{eq:beta_prime}).
To deduce (\ref{eq:ToShowII}), we will compare the polynomial terms that bound $\widehat{\Lambda}_i$ with the polynomial terms 
in (\ref{eq:LowTypLambda}), { hence from now on we only consider the polynomial terms in $N$.

Now, the probability that the event described in cases 4 and 8 occurs can be bounded from above by the probability of having a path of length three.
Hence, all events listed in cases 2--8 can be reduced to the event of certain trees appearing in the system.
More precisely, the probability of any of the events of cases 2--8 is bounded from above by the probability of a certain tree with $3$ (resp.\ $4$) edges and $4$ (resp.\ $5$) vertices to be present in the graph.
In general, 
%by applying the same proof of Lemma \ref{lemma_2.4_evolution} in an iterative way (which works in the same way as the proof of \cite[Lemma 2.4]{ar:Foun13+}, hence we omit the details) 
we obtain the following result.
\begin{Claim}\label{claim:embeddings}
Let $\mathcal{T}_k$ be a tree on $k\geq 2$ vertices $u_1,u_2,\ldots,u_{k}$ that have degrees $ n_1,n_2,\ldots,n_{k}$, where $k$ 
is fixed.
% and types $ t_{u_1},t_{u_2},\ldots,t_{u_{k+1}}$ respectively.
Then %the expected number of embeddings is given by
\[
\begin{split}
\E & (\# \textnormal{ copies of }\mathcal{T}_k\textnormal{ in }\G(N;\zeta,\alpha,\beta,\nu))\\
& \asymp N^{k-\beta'(k-1)} R^{\delta (k-1)}\prod_{i=1}^{k} \max \left \{ 1, R, N^{\beta' n_i/2-\alpha/\zeta} 
e^{-\omega(N)(\beta'\zeta n_i/2-\alpha)}.
\right \},
\end{split}
\]
%conditional probability that such vertices represent an embedding of a tree $\mathcal{T}_k $ with $k$ edges and $k+1$ vertices is given by 
\end{Claim}
\begin{proof}
We will calculate the probability that vertices $v_1,\ldots, v_k$ form a tree that is isomorphic to $\mathcal{T}_k$, where 
$v_i$ is identified with $u_i$. We shall give an estimate first on the probability that the $v_i$s form this tree conditional on their types. 
So suppose that the vertices $v_1,v_2,\ldots,v_{k}$ have types $ t_{u_1},t_{u_2},\ldots,t_{u_{k}}$, respectively. 

We shall expose/embed the positions of $v_1,v_2,\ldots, v_k$ according to a breadth-first ordering of the vertices of $\mathcal{T}_k$.
 Without loss of generality, let us assume that 
$v_1$ is the root of $\mathcal{T}_k$. The breadth-first search algorithm discovers in each \emph{round} 
the children of an internal vertex, starting at $v_1$. If the rooted  $\mathcal{T}_k$ has $\ell$ leaves, then there will be $k-\ell$ rounds.
Also, assume for simplicity, that during round $i$ the children of vertex $v_i$ will be exposed, for $i=1,\ldots, k-\ell$. Let
$\mathcal{T}_k^{(i)}$ denote the subtree of $\mathcal{T}_k$ that has been revealed after $i$ rounds of the breadth-first search. 
With a slight abuse of notation, we also use the symbol $\mathcal{T}_k^{(i)}$ to denote the \emph{event} that the subtree
$\mathcal{T}_k^{(i)}$ has been embedded. 

Consider $\mathcal{T}_k^{(1)}$ first. If $v_{1}$ has $v_{i_1},\ldots, v_{i_{n_1}}$ as its children, 
then (with $\sim$ denoting vertex adjacency)
\begin{equation}\label{eq:ind_base}
 \begin{split}
 \P ( \mathcal{T}_k^{(1)} \mid t_{v_1},\ldots , t_{v_k}) &= \prod_{j=1}^{n_1} 
 \P (v_{i_j} \sim v_1 \mid t_{v_1}, t_{v_{i_j}}) \\
& \stackrel{Lemma~\ref{lemma_2.4_evolution}}{\asymp}  N^{-\beta'n_1} R^{\delta n_1} e^{\beta' \zeta t_{v_1} n_1 /2} 
\prod_{j=1}^{n_1} e^{\beta' \zeta t_{v_{i_j}}/2}.
 \end{split}
\end{equation} 
Suppose that $\mathcal{T}_k^{(i)}$ has $e_i$ edges and let $V_i$ denote its set of vertices with $\mathcal{L}_i \subset V_i$ 
being its set of leaves. 
Assume that for $i$ that satisfies $1\leq i < k-\ell$ we have shown that 
\begin{equation}\label{eq:ind_hyp} 
\P (\mathcal{T}_k^{(i)} \mid t_{v_1},\ldots, t_{v_k}) \asymp  N^{-\beta' e_i} R^{\delta e_i}
\prod_{v_j \in V_i \setminus \mathcal{L}_i} e^{\beta' \zeta t_{v_j} n_j /2} 
\prod_{v_j \in \mathcal{L}_i} e^{\beta' \zeta t_{v_{i_j}}/2}.
\end{equation}
In the above expression, the hidden constants do depend on $i$. 

We shall derive a similar expression for $\mathcal{T}_k^{(i+1)}$. Assume that 
$v_{i+1}$, which has to be a leaf in $\mathcal{T}_k^{(i)}$ has $v_{j_1},\ldots, v_{j_{n_{i+1}-1}}$ as its children that are 
to be exposed. 
\[
\begin{split} 
\P ( \mathcal{T}_k^{(i+1)} \mid  t_{v_1},\ldots, t_{v_k} ) 
&= \P ( \{ \cap_{s=1}^{n_{i+1}-1} v_{i+1} \sim v_{j_s} \} \cap \mathcal{T}_k^{(i)} \mid  t_{v_1},\ldots, t_{v_k} ) \\
&= \P ( \{ \cap_{s=1}^{n_{i+1}-1} v_{i+1} \sim v_{j_s} \} \mid  t_{v_1},\ldots, t_{v_k} )~
\P ( \mathcal{T}_k^{(i)} \mid  t_{v_1},\ldots, t_{v_k}).
\end{split}
\]
The first term on the right-hand side can be calculated as in (\ref{eq:ind_base}):
\[ 
\P ( \{ \cap_{s=1}^{n_{i+1}-1} v_{i+1} \sim v_{j_s} \} \mid  t_{v_1},\ldots, t_{v_k} ) 
=  N^{-\beta'( n_{i+1}-1 )} R^{\delta ( n_{i+1} -1)} e^{\beta' \zeta t_{v_{i+1}} ( n_{i+1}-1 ) /2} 
\prod_{s=1}^{n_{i+1}-1} e^{\beta' \zeta t_{v_{j_s}}/2}.
\]
The second term is given by (\ref{eq:ind_hyp}). 
Now, observe that $\mathcal{L}_{i+1} = \mathcal{L}_i \cup \{v_{j_1},\ldots, v_{j_{n_{i+1}-1}} \} \setminus v_{i+1}$ and 
$v_{i+1} \in V_{i+1} \setminus \mathcal{L}_{i+1}$.
Also, $e_{i+1} = e_i + (n_{i+1}-1)$. 
Therefore, 
\begin{equation*}
\begin{split}
\lefteqn{\P ( \mathcal{T}_k^{(i+1)} \mid  t_{v_1},\ldots, t_{v_k} ) \asymp} \\
&\left( N^{-\beta'( n_{i+1}-1 )} R^{\delta ( n_{i+1} -1)} e^{\beta' \zeta t_{v_{i+1}} ( n_{i+1}-1 ) /2} 
\prod_{s=1}^{n_{i+1}-1} e^{\beta' \zeta t_{v_{j_s}}/2} \right)\times \\
& \hspace{2cm}
\left( N^{-\beta' e_i} R^{\delta e_i}
\prod_{v_j \in V_i \setminus \mathcal{L}_i} e^{\beta' \zeta t_{v_j} n_j /2} 
\prod_{v_j \in \mathcal{L}_i} e^{\beta' \zeta t_{v_{i_j}}/2}\right) \\
&= N^{-\beta' (e_i + n_{i+1}-1)} R^{\delta ( e_i + n_{i+1}-1 )}
\prod_{v_j \in V_{i+1} \setminus \mathcal{L}_{i+1}} e^{\beta' \zeta t_{v_j} n_j /2} 
\prod_{v_j \in \mathcal{L}_{i+1}} e^{\beta' \zeta t_{v_j}/2} \\
&= N^{-\beta' e_{i+1}} R^{\delta e_{i+1}} 
\prod_{v_j \in V_{i+1} \setminus \mathcal{L}_{i+1}} e^{\beta' \zeta t_{v_j} n_j /2} 
\prod_{v_j \in \mathcal{L}_{i+1}} e^{\beta' \zeta t_{v_j}/2}.
\end{split}
\end{equation*}
Thus, 
\begin{equation*}
\begin{split}
&\P ( \mathcal{T}_k^{(k-\ell)} \mid  t_{v_1},\ldots, t_{v_k} ) \asymp N^{-\beta' e_{k-\ell}} R^{\delta e_{k-\ell}}
\prod_{v_j \in V_i \setminus \mathcal{L}_{k-\ell}} e^{\beta' \zeta t_{v_j} n_j /2} 
\prod_{v_j \in \mathcal{L}_{k-\ell }} e^{\beta' \zeta t_{v_{i_j}}/2} \\
&\stackrel{ \mathcal{T}_k^{(k-\ell)} = \mathcal{T}_k}{=}N^{-\beta'(k-1)} R^{\delta (k-1)}\prod_{i=1}^{k} e^{\beta'\zeta n_i t_{v_i}/2}.
\end{split}
\end{equation*}

Integrating over the types, we get
\[
\begin{split}
\lefteqn{\P  (v_1,v_2,\ldots,v_{k}\textnormal{ form an embedding of }\mathcal{T}_k) \asymp}\\
& \left( \frac{R^{\delta}}{N^{\beta'}} \right)^{k-1} \int_0^{R/2-\omega (N)} \cdots \int_0^{R/2-\omega (N)}e^{\beta'\zeta n_1 t_{u_1}/2-\alpha t_{u_1}}\cdots e^{\beta'\zeta n_{k} t_{u_{k}}/2-\alpha t_{u_{k}}}dt_{u_1}\cdots dt_{u_{k}}\\
& \asymp  N^{-\beta'(k-1)} R^{\delta (k-1)}\prod_{i=1}^{k}\max \left \{ 1, R, e^{(R/2-\omega(N))(\beta'\zeta n_i/2-\alpha)}\right \}.
\end{split}
\]
The number of choices for the vertices $v_1,\ldots, v_k$ and their arrangement into a copy of $\mathcal{T}_k$ is  proportional 
to $N^k$.  Hence, multiplying the above expression by $N^{k}$ yields the statement of the claim. 
\end{proof}
Now, suppose that there are exactly $\ell\geq 1$ vertices with degrees larger than $ 2\alpha/(\beta'\zeta)$.
Without loss of generality, we can re-label the vertices in such a way that $u_1,u_2,\ldots,u_{\ell}$ are such that 
$n_i> 2\alpha/(\beta'\zeta)$ for $i\in \{1, 2, \ldots,\ell\}$. Furthermore, assume that there are $\ell'$ vertices with 
$n_i = 2\alpha/(\beta'\zeta)$.
This implies that 
\[
\prod_{i=1}^{k}\max \left \{ 1, R, e^{(R/2-\omega(N))(\beta'\zeta n_i/2-\alpha)}\right \}= R^{\ell'} e^{(R/2-\omega(N))\Bigl (\beta'\zeta \bigl (\sum_{j=1}^\ell n_j\bigr )/2-\ell \alpha\Bigr )}.
\]
Since the $n_j$s are degrees of vertices in $\mathcal{T}_k$ we have $\sum_{j=1}^\ell n_j\leq 2(k-1)$, since there are $k-1$ edges in
$\mathcal{T}_k$.
%The best estimate that we can obtain under these two conditions is given by the case when $\mathcal{T}_k$ is a star, in fact this is the case when $\ell=1$ and $n_1=k$.
%To see that the star provides indeed an upper bound, we can think in the following way.

Thus, 
\[
\beta'\zeta \bigl (\sum_{j=1}^\ell n_j\bigr )/2-\ell \alpha \leq \beta'\zeta 2(k-1)/2-\ell \alpha=\beta'\zeta (k-1)-\ell \alpha
\stackrel{\ell\geq 1}{\leq} \beta'\zeta (k-1)- \alpha.
\]
The last quantity corresponds to the case when $\mathcal{T}_k$ is a star, in fact this is the case when $\ell=1$ and $n_1=k-1$.
(Note that even if $\ell=0$, then the star provides an upper bound.) 
This implies that for $i=2,\ldots, 8$ we have $\E (\widehat{\Lambda}_i) = o \left( \max \{ \E (\widehat{\Lambda}_3) , 
\E (\widehat{\Lambda}_6) \} \right)$. 
So to show (\ref{eq:ToShowII}), it suffices to prove it for Cases 3 and 6. 
However, we begin Case 1 which shows (\ref{eq:ToShowI}).
}

\noindent
\emph{Case 1:} the two incomplete triangles are clearly independent, since they are disjoint. In other words, the realization of the event
$\{ u_1\sim w_1\sim v_1\}$ gives no information about the triple $\{u_2,v_2,w_2 \}$.
Thus,
\[
\begin{split}
& \P \bigl (\Lambda(u_1,v_1;w_1)\cap \Lambda(u_2,v_2;w_2) \bigr )=\P \bigl (\Lambda(u_1,v_1;w_1)\bigr )
\P \bigl( \Lambda(u_2,v_2;w_2) \bigr ).
\end{split}
\]
%that can be evaluated by (\ref{eq_lambda_general}).
There are $3{N \choose 3}$ ways to select three distinct vertices $u_1, v_1, w_1$ and distinguish among them the central vertex of the path.
Having chosen those, there $3{N-3 \choose 3}$ ways for selecting the second triple. 
Therefore,
\[
\E (\widehat{\Lambda}_1) =9
{N \choose 3}{N-3 \choose 3}\P \bigl (\Lambda(u_1,v_1;w_1)\bigr ) \P \bigl( \Lambda(u_2,v_2;w_2) \bigr )=
\E^2 ( \widehat{\Lambda})(1-o(1)),
\]
which is (\ref{eq:ToShowI}).

\smallbreak

\noindent
\emph{Case 3:} 
By Claim~\ref{claim:embeddings}, the contribution of \emph{Case 3} to $\E (\widehat{\Lambda}^2)$ is
\[
\E (\widehat{\Lambda}_3) \lesssim
\left \{ \begin{array}{ll}
N^5\frac{R^{4\delta}}{N^{4\beta'}} \asymp  R^{4\delta} N^{5-4\beta'}, & \textnormal{ if }\frac{\beta'\zeta}{\alpha}<\frac{1}{2}\\[1.1ex]
N^5\frac{R^{4\delta+1}}{N^{4\beta'}} \asymp N^{5-4\beta'} R^{4\delta+1}, & \textnormal{ if }\frac{\beta'\zeta}{\alpha}=\frac{1}{2} \\[1.1ex]
R^{4\delta}\, N^5\frac{N^{2\beta'-\alpha /\zeta}e^{-(2\beta'\zeta - \alpha )\omega (N)}}{N^{4\beta'}}, & \textnormal{ if }\frac{\beta'\zeta}{\alpha}
> \frac{1}{2}
\end{array} \right ..
\]
Now it suffices to take into account the exponents of the terms in $N$.

When $\beta' \zeta / \alpha \leq 1/2$, then (\ref{eq:ToShowII}) holds since $5-4\beta' < 2(3-2\beta')$.
When $\beta'\zeta /\alpha > 1/2$, we need to have
$$ 5-2\beta' - \alpha / \zeta \leq 2(3-2\beta'), \ \mbox{if $1/2 < \beta' \zeta /\alpha \leq 1$}.$$
This suffices in order to show (\ref{eq:ToShowII}) because of the $e^{-(2\beta'\zeta - \alpha )\omega (N)}$ factor.
This is equivalent to $2\beta' \leq \alpha /\zeta + 1$, which holds since $\beta' \leq {\alpha/\zeta}$ and $\beta' \leq 1$.
Also, we need to have
$$ 5-2\beta' - \alpha / \zeta < 2(3-\beta' - \alpha /\zeta), \ \mbox{if $1 < \beta' \zeta /\alpha $ }.$$
But this is equivalent to $\alpha /\zeta < 1$, which also holds since $\alpha /\zeta < \beta'\leq 1$.
\smallbreak

\noindent
\emph{Case 6:} 
%here, after conditioning on the types, one has
% \[
% \begin{split}
% \P (v_1\sim & w_1\sim u_1, w_1\sim v_2 \mid t_{v_1}, t_{u_1}, t_{w_1}, t_{v_2}) \lesssim \frac{R^{3\delta}}{A_{v_1, w_1}^{\beta'} A_{w_1, u_1}^{\beta'}A_{w_1, v_2 }^{\beta'}} \\
% & \asymp \frac{R^{3\delta}e^{(\beta'\zeta /2)(t_{u_1}+t_{v_1}+t_{v_2})+(3\beta'\zeta /2)t_{w_1}}}{N^{3\beta'}},
% \end{split}
% \]
% where the first step follows again from Lemma~\ref{lemma_2.4_evolution} and (\ref{eq:ln_A}). Integrating over the types of the four vertices
% involved, we get
% \[ \P \left(\Lambda (u_1,v_1;w_1) \cap \Lambda (u_1,v_2;w_1) \right) \asymp
% \int_0^{R/2-\omega(N)} \frac{R^{3\delta}\, e^{(3\beta'\zeta /2-\alpha )t_{w_1}}}{N^{3\beta'}}dt_{w_1}.
% \]
Claim~\ref{claim:embeddings} yields:
\[
\E(\widehat{\Lambda}_6) \asymp
\left \{ \begin{array}{ll}
N^4~\frac{R^{3\delta}}{N^{3\beta'}}, & \textnormal{ if }\frac{\beta'\zeta}{\alpha}<\frac{2}{3}\\[1.1ex]
N^4\frac{R^{3\delta+1}}{N^{3\beta'}}, & \textnormal{ if }\frac{\beta'\zeta}{\alpha}=\frac{2}{3}\\[1.1ex]
N^4\frac{R^{3\delta}\,N^{(3\beta'/2-\alpha/\zeta)}e^{-(3\zeta\beta'/2-\alpha)\omega(N)}}{N^{3\beta'}} = & \\
\quad = R^{3\delta}\,N^{4-3\beta'/2-\alpha/\zeta}e^{- (3\beta'\zeta/2-\alpha)\omega(N)},
& \textnormal{ if }\frac{\beta'\zeta}{\alpha}>\frac{2}{3}
\end{array} \right ..
\]
To verify (\ref{eq:ToShowII}) when $\beta'\zeta/ \alpha\leq 2/3$, note that
$4-3\beta' < 2(3-2\beta')$ (which is equivalent to $\beta' < 2$) holds. When $2/3 < \beta' \zeta /\alpha \leq 1$,
it suffices to verify that
$$ 4- 3 \beta' /2 -\alpha /\zeta < 2(3 - 2\beta'), $$
which is equivalent to $5 \beta' /2 - \alpha /\zeta < 2$. But $\alpha /\zeta \geq \beta'$, which implies that
$$ {5 \beta' \over 2} - {\alpha \over \zeta}  \leq {3\beta' \over 2} \leq {3 \over 2} < 2.$$
Finally, assume that $\beta' \zeta /\alpha > 1$.
Here it suffices to show that
$$ 4- 3 \beta' /2 -\alpha /\zeta < 2(3- \beta' - \alpha /\zeta),$$
which is equivalent to
$$ {\beta' \over 2} + {\alpha \over \zeta} < 2.$$
But $\alpha / \zeta < \beta'$ whereby the left-hand side is at most $3 \beta' /2 \leq 3/2 < 2$. Hence, (\ref{eq:ToShowII}) holds also in this case.
\end{proof}

Now we have that $\operatorname{Var}(\widehat{\Lambda})= o (\E^2 \widehat{\Lambda}) $ and, therefore, the concentration of $\widehat{\Lambda} $ follows immediately from Chebyshev's inequality together with the fact that $\E (\widehat{\Lambda}) \rightarrow \infty$ as $N \rightarrow \infty$
(cf. (\ref{eq:LowTypLambda})).

Next we show the analogue of Lemma~\ref{ED2} for $\widehat{T}$.
\begin{Lemma}\label{ET2} For any $0 < {\zeta \over \alpha} < 1$ and $\beta > 1$ we have
\[
\E (\widehat{T}^2)= \E^2 \widehat{T} \bigl ( 1+o(1)\bigr ).
\]
\end{Lemma}
\begin{proof} The proof of this fact is almost identical to that of Lemma~\ref{ED2}, which we apply to a version of the random variable $\widehat{T}$.
More specifically, we let $\widehat{T}_r$ denote the number of \emph{rooted} typical triangles, that is, the typical triangles with one distinguished vertex which we call the root.
Note that $\widehat{T}_r = 3 \widehat{T}$, whereby $\E (\widehat{T}) \asymp \E (\widehat{T}_r)$.

To estimate the second moment of $\widehat{T}_r$, we also need to consider all different cases that cover all possible
ways of intersection of the distinct triples $u_1,v_1,w_1$ and $u_2,v_2,w_2$. We denote by $\widehat{T}_i$ the contribution of each
case in $\E (\widehat{T}^2)$.

For three distinct vertices $u,v,w$ we let $\Delta (u,v;w)$ denote the event that the vertices $u,v,w$ form a triangle that is rooted at $w$.
If $u_1,v_1,w_1$ and $u_2,v_2,w_2$ are two disjoint triples of vertices, that is, for \emph{Case 1} we have
\[
\begin{split}
& \P \bigl (\Delta (u_1,v_1;w_1) \Delta (u_2,v_2;w_2) \bigr )=\P \bigl (\Delta (u_1,v_1;w_1)\bigr ) \P \bigl( \Delta (u_2,v_2;w_2) \bigr ),
\end{split}
\]
which implies that
\[
\E (\widehat{T}_1)=\E^2 (\widehat{T})(1-o(1)).
\]
For the remaining { two cases (either when the triangles share a vertex or when they share an edge)} we can deduce that $\E (\widehat{T}_i)/\E^2 ( \widehat{T}_r ) = o(1)$, since $\E (\widehat{T}_i) \leq
\E (\widehat{\Lambda}_i)$ (as every rooted triangle is contained an incomplete triangle whose pivoting vertex is the root of the triangle) and
$\E(\widehat{T}_r) \asymp \E (\widehat{\Lambda})$
(cf. Propositions~\ref{prop:incomplete_triangles},~\ref{prop:complete_triangles}).
\end{proof}
Again, concentration of $\widehat{T}$ follows by applying Chebyshev's inequality together with the fact that
$\E (\widehat{T}) \rightarrow \infty$ as $N \rightarrow \infty$ (cf.~(\ref{expected_T})).

\section{On the probability of triangles: proof of Theorem~\ref{thm:typicalglobalclustering}}\label{sect:probab_triangles}
In this section compute the probability that three \emph{typical} vertices $u,v,w$ form a triangle. In particular, we find an explicit expression for such quantity, which allows us to prove Theorem \ref{thm:typicalglobalclustering}.

For any two distinct points $u$ and $v$, we denote the~\emph{angle of $u$ with respect to $v$}
by $\vartheta_{u,v} \in { (-\pi,\pi)}$; here we assume that $v$ is the reference point and $\vartheta_{u,v}$ is positive
when the straight line that joins $u$ with the center of $\D$ is reached if we rotate the straight line that joins $v$ with the center of $\D$ in
the counterclockwise direction.

We will approximate $\widehat{T}$ by a random variable that counts triangles whose vertices are arranged in a certain way on $\D$.
More specifically, for three distinct vertices $u,v$ and $w$ we let $T(u,v;w)$ denote the indicator random variable that is equal to
1 if and only if $u,v$ and $w$ form a triangle with
$0 \leq \vartheta_{u,w} \leq \pi$ and $- \pi \leq  \vartheta_{v,w} \leq 0$ and $t_u, t_v, t_w < R/2 - \omega (N)$.
We let
\begin{equation} \label{eq:hatT'def}
\widehat{T}' := \sum_{w \in V_N} \sum_{\stackrel{(u,v): }{ u,v \in \V \setminus \{ w\} }} T(u,v;w),
\end{equation}
where the second sum ranges over the set of ordered pairs of distinct vertices in $\V \setminus \{w \}$.
Note that in general $\widehat{T}'$ is not equal to $\widehat{T}$. However,
\begin{equation}  \label{eq:equivalence} \widehat{T} \leq \widehat{T}' \leq 3 \widehat{T}.
\end{equation}
We will work with this random variable as its analysis is somewhat easier compared to that of $\widehat{T}$ due to the fact
that the indicators $T(u,v;w)$ are associated with a certain arrangement of the vertices $u,v,w$ on $\D$.
For any two distinct vertices $u$ and $v$, we define the following quantities:
\begin{equation}\label{tilda_hat_definition}
%\begin{split}
\tilde{\theta}_{u,v}:=
\begin{cases}
\bigl ( \omega(N)A_{u,v}\bigr )^{-1} & \textnormal{if }\beta \geq  1\\
\omega(N)A_{u,v}^{-1} & \textnormal{if }\beta<1
\end{cases}
\quad \textnormal{and} \quad
\hat{\theta}_{u,v}:=
\begin{cases}
\omega(N)A_{u,v}^{-1} & \textnormal{if }\beta> 1\\
\pi-\tilde{\theta}_{u,v}  & \textnormal{if }\beta \leq 1
\end{cases}.
\end{equation}
Informally, these define an area such that when the (relative) angle between $u$ and $v$ is within these bounds, the probability
that $u$ is adjacent to $v$ is maximized.

We split the expected value of $T(u,v;w)$ according to the value of the angle between $u$ and $w$ and that between $v$ and $w$.
Our aim is to show that the main contribution to the expected value of $T(u,v;w)$ comes from the case when $|\vartheta_{u,w}|$ and
$|\vartheta_{v,w}|$ are within the bounds $\tilde{\theta}_{u,w}$, $\hat{\theta}_{u,w}$ and $\tilde{\theta}_{v,w}$, $\hat{\theta}_{v,w}$  respectively. %, defined in (\ref{tilda_hat_definition}).
We have
\begin{equation} \label{eq:ProbSplit}
\begin{split}
\P & \Big( T  (u,v;w)=1 \mid t_u, t_v, t_w \Big) \\
& = \E \left( T(u,v;w)\mathbf{1}_{\{ \tilde{\theta}_{u,w}\leq \vartheta_{u,w}\leq \hat{\theta}_{u,w},
-\tilde{\theta}_{v,w}\geq \vartheta_{v,w}\geq - \hat{\theta}_{v,w}\} }\mid t_u, t_v, t_w \right)\\
& + \E \left(T(u,v;w) \left( 1 - \mathbf{1}_{\{ \tilde{\theta}_{u,w}\leq \vartheta_{u,w}\leq \hat{\theta}_{u,w},
-\tilde{\theta}_{v,w}\geq \vartheta_{v,w}\geq - \hat{\theta}_{v,w}\} } \right) \mid t_u, t_v, t_w \right).
\end{split}
\end{equation}
We set
\begin{equation}\label{eq:def_vartheta}
\vartheta(u,v;w):= \mathbf{1}_{\{ \tilde{\theta}_{u,w}\leq \vartheta_{u,w}\leq \hat{\theta}_{u,w},
-\tilde{\theta}_{v,w}\geq \vartheta_{v,w}\geq - \hat{\theta}_{v,w}\} }.
\end{equation}

Using Lemma~\ref{lem:probs}, we express $p_{u,v}$,  $p_{u,w}$ and $p_{v,w}$ as
\begin{equation}\label{exp_sin}
\begin{split}
& p_{u,v}= \frac{1}{C\, A^\beta_{u,v}\sin^\beta (\theta_{u,v}/2)+1},\\[1.1ex]
& p_{u,w}= \frac{1}{C\, A^\beta_{u,w}\sin^\beta (\theta_{u,w}/2)+1},\\[1.1ex]
& p_{v,w}= \frac{1}{C\, A^\beta_{v,w}\sin^\beta (\theta_{v,w}/2)+1},
\end{split}
\end{equation}
as long as $\theta_{u,v} \gg \bar{\theta}_{u,v}$, $\theta_{u,w} \gg \bar{\theta}_{u,w}$ and $\theta_{v,w} \gg \bar{\theta}_{v,w}$,
respectively.
Note that the relative angle between two points is uniformly distributed in the interval $[0,\pi]$.
Moreover, if $\vartheta_{u,v} \in [-\pi, \pi]$, then $\theta_{u,v} = |\vartheta_{u,v}|$ (recall that $\theta_{u,v}$ denotes the relative
angle between the points $u$ and $v$).

\noindent
Now, we focus on the first term of { the right hand side of} (\ref{eq:ProbSplit}).
{ In order to simplify the notation, we set 
\[
C_{u,v}:=e^{\frac{\zeta}{2}(t_u+t_v)},
\]
where $u,v$ are any two elements of $\V$.}
\begin{Lemma} \label{lem:TMain}
Let $u,v,w \in { \V}$ be three distinct vertices. Then uniformly for $t_u, t_v, t_w < R/2 - \omega (N)$ the following hold:

\noindent
for $\beta>1$
\begin{equation}\label{prop:equ_prob_b>1}
\begin{split}
\E & ( T(u,v;w) \vartheta(u,v;w)\mid t_u,t_v, t_w)\\
& = {(1+ o(1))\over \pi^2\, A_{u,w}A_{v,w}}\int_{[0,\infty)^2} \frac{1}{z_1^\beta+1}\frac{1}{z_2^\beta+1}\frac{1}{ (C_{w,v} z_1+C_{w,u} z_2)^\beta+1 } \, d z_1 d z_2,
\end{split}
\end{equation}

\noindent
for $\beta=1$
\begin{equation}\label{prop:equ_prob_b=1}
\begin{split}
% & \frac{1}{A_{uw}A_{vw}}\int_{1/(2\omega(N))}^{A_{uw}\pi/2} \int_{1/(2\omega(N))}^{A_{vw}\pi/2} \frac{1}{z_1+1}
% \frac{1}{z_2+1}\frac{1}{C_{w,v} z_1+C_{w,u} z_2+1 } \, d z_1 d z_2\\
% &\leq
\E & ( T(u,v;w) \vartheta(u,v;w) \mid t_u,t_v, t_w) \\
& \lesssim \frac{1}{A_{u,w}A_{v,w}}\int_{1/(2\pi \omega(N))}^{A_{v,w}/2} \int_{1/(2\pi \omega(N))}^{A_{u,w} /2} \frac{1}{z_1+1}
\frac{1}{z_2+1}\frac{1}{ C_{w,v} z_1+C_{w,u} z_2+1 } \, d z_1 d z_2,
\end{split}
\end{equation}
and for $\beta<1$ we have
\begin{equation}\label{prop:equ_prob_b<1}
\begin{split}
% & \frac{4}{A_{uw}A_{vw}}\int_{\omega(N)/2}^{A_{uw}\pi/2} \int_{\omega(N)/2}^{A_{vw}\pi/2} \frac{1}{z_1^\beta+1}\frac{1}{z_2^\beta+1}\frac{1}{ (C_{w,v} z_1+C_{w,u} z_2)^\beta+1 } \, d z_1 d z_2  \\
\E & (T(u,v;w) \vartheta(u,v;w) \mid t_u,t_v, t_w) \\
& \lesssim \frac{1}{A_{u,w}A_{v,w}}~\int_{\omega(N)/2\pi}^{A_{v,w} /2} \int_{\omega(N)/2\pi}^{A_{u,w} /2}
\frac{1}{z_1^\beta+1}\frac{1}{z_2^\beta+1}\frac{1}{ (C_{w,v} z_1+C_{w,u} z_2)^\beta+1 } \, d z_1 d z_2 .
\end{split}
\end{equation}
\end{Lemma}
\begin{proof}
For all $\beta>0$, we can use the expressions in (\ref{exp_sin}) obtaining:
\begin{equation} \label{eq:exp_int}
\begin{split}
\E & \left ( T(u,v;w)\mathbf{1}_{\{\tilde{\theta}_{u,w}\leq \vartheta_{u,w}\leq \hat{\theta}_{u,w}, -\hat{\theta}_{v,w}\leq \vartheta_{v,w}\leq -\tilde{\theta}_{v,w}\}} \mid t_u,t_v,t_w \right )\\
& = \frac{1}{4\pi^2}\int_{\tilde{\theta}_{u,w}}^{\hat{\theta}_{u,w}} 
\int_{-\hat{\theta}_{v,w}}^{-\tilde{\theta}_{v,w}} p_{u,w}\, p_{v,w}\, p_{u,v}\, d \vartheta_{v,w} d \vartheta_{u,w} \\
& = \frac{1}{4\pi^2} \int_{\tilde{\theta}_{u,w}}^{\hat{\theta}_{u,w}} \int_{-\hat{\theta}_{v,w}}^{-\tilde{\theta}_{v,w}} 
\frac{1}{C\, A^\beta_{u,w}\sin^\beta (\theta_{u,w}/2)+1}
\frac{1}{C\, A^\beta_{v,w}\sin^\beta (\theta_{v,w}/2)+1}~p_{u,v} \, d \vartheta_{v,w} d \vartheta_{u,w}.
\end{split}
\end{equation}
In this case the relative angle between $u$ and $v$ is either $\theta_{u,w} + \theta_{v,w}$, if this sum is at most $\pi$; otherwise,
it is equal to $2\pi - (\theta_{u,w} + \theta_{v,w})$ (this may happen when $\beta \leq 1$).
In this case, $\sin ({ \theta_{u,v}}/2)= \sin \left( \pi - {\theta_{u,w} + \theta_{v,w} \over 2} \right) =
\sin \left( {\theta_{u,w} + \theta_{v,w}  \over 2} \right)$. Also, by the definition of $\hat{\theta}_{u,w}$ and $\hat{\theta}_{v,w}$ for
$\beta \leq 1$, it follows that $\theta_{u,v}\geq \hat{\theta}_{u,w} + \hat{\theta}_{v,w}$. (Note that this issue does not come up when
$\beta >1$;  for $N$ sufficiently large, $\theta_{u,w} + \theta_{v,w}$ is much smaller than $\pi$.)

Thus, in order to use the expression in (\ref{exp_sin}) for $p_{u,v}$, we need to show that the angle $\theta_{u,v}$ is asymptotically
larger than the critical value $\bar{\theta}_{u,v}$ defined by Equation (\ref{theta_c}).
The following fact shows that this is the case.
\begin{Fact}
Let $\beta > 0$ and let $u$ and $v$ be two distinct vertices of the graph with $t_u, t_v < R/2 - \omega (N)$.
Then uniformly for any such $t_u$ and $t_v$ we have
\[
\tilde{\theta}_{u,w}+\tilde{\theta}_{v,w} \gg \bar{\theta}_{u,v}.
\]
\end{Fact}
\begin{proof}
Starting from the definitions of $ \tilde{\theta}_{u,w}$ and $\tilde{\theta}_{v,w}$ (see (\ref{tilda_hat_definition})), we have for
any $\beta >0$:
\[
\begin{split}
\bigl (\tilde{\theta}_{u,w}+\tilde{\theta}_{v,w}\bigr )^2
& \geq \frac{1}{\omega(N)^2}\bigl ( A_{u,w}^{-1}+ A_{v,w}^{-1}\bigr )^2 \geq  \frac{1}{\omega(N)^2}
\bigl( A_{u,w}^{-2}+ A_{v,w}^{-2}\bigr) \\
& =  \frac{\nu^2}{\omega(N)^2}\left ( \frac{e^{\zeta (t_u+t_w)}}{N^2}+ \frac{e^{\zeta (t_v+t_w)}}{N^2}\right )
 \geq  \frac{\nu^2}{\omega(N)^2}\left ( \frac{e^{\zeta t_u}}{N^2}+ \frac{e^{\zeta t_v}}{N^2}\right).
\end{split}
\]
Hence, the statement of the lemma holds if
\[
\frac{1}{\omega(N)^2}\left ( \frac{e^{\zeta t_u}}{N^2}+ \frac{e^{\zeta t_v}}{N^2}\right ) \gg  \frac{e^{2\zeta t_u}}{N^4}+\frac{e^{2\zeta t_v}}{N^4}.
\]
It suffices to show that
\begin{equation*}
\frac{1}{\omega(N)^2} \frac{e^{\zeta t_u}}{N^2} \gg  \frac{e^{2\zeta t_u}}{N^4} \quad \textnormal{ and } \quad \frac{1}{\omega(N)^2} \frac{e^{\zeta t_v}}{N^2} \gg  \frac{e^{2\zeta t_v}}{N^4}.
\end{equation*}
The first condition (the argument for the second is identical) is equivalent to
\begin{equation*}
\frac{1}{\omega(N)^2} \gg  \frac{e^{\zeta t_u}}{N^2} \quad \Leftrightarrow \quad \frac{1}{\omega(N)} \gg  \frac{e^{\frac{\zeta}{2}
t_u}}{N}.
\end{equation*}
The latter holds as $t_u < R/2 - \omega (N)$.
\end{proof}
Now, since $ C=1+o(1)$ uniformly for all $\tilde{\theta}_{u,w}\leq \theta_{u,w}\leq \hat{\theta}_{u,w} $ and $\tilde{\theta}_{v,w}\leq \theta_{v,w}\leq \hat{\theta}_{v,w}$, the last integral in Equation (\ref{eq:exp_int}) is equal to:
\[
\begin{split}
\int_{\tilde{\theta}_{u,w}}^{\hat{\theta}_{u,w}}& \int_{-\hat{\theta}_{v,w}}^{-\tilde{\theta}_{v,w}}
\frac{1}{C\,A^\beta_{u,w}
\sin^\beta (\theta_{u,w}/2)+1}
\frac{1}{C\, A^\beta_{v,w}\sin^\beta (\theta_{v,w}/2)+1}\times \\
& \qquad \frac{1}{C\,A^\beta_{u,v}\sin^\beta ((\theta_{u,w}+\theta_{v,w})/2)+1} d \vartheta_{v,w} d \vartheta_{u,w}  .
\end{split}
\]
To bound this integral we will use a first-order approximation of the $\sin(\cdot)$ function:
\begin{equation} \label{eq:sin_bounds}
{\theta \over \pi}\leq \sin \left( {\theta \over 2} \right) \leq {\theta \over 2}.
\end{equation}
Moreover, if $\theta$ is sufficiently small, then the upper bound is a tight approximation. More precisely,
for any $\delta > 0$ there exists $\eps > 0$ such that for every $\theta < \eps$ we have
\begin{equation*}
\sin \left( {\theta \over 2} \right) > (1-\delta)~{\theta \over 2}.
\end{equation*}
Hence, the above integral can be bounded from above and below by integrals of the form
\begin{equation*}
\begin{split}
\int_{\tilde{\theta}_{u,w}}^{\hat{\theta}_{u,w}}& \int_{-\hat{\theta}_{v,w}}^{-\tilde{\theta}_{v,w}} 
\frac{1}{(\lambda\,A^\beta_{u,w}(
\theta_{u,w}/2)^\beta+1)}\frac{1}{(\lambda\,A^\beta_{v,w}(\theta_{v,w}/2)^\beta+1)}\times \\
& \quad \frac{1}{
\left( \lambda\,A^\beta_{u,v} (({ \theta_{u,w}+\theta_{v,w}})/2)^\beta+1 \right )} d \vartheta_{v,w} d \vartheta_{u,w},
\end{split}
\end{equation*}
where $\lambda>0$ is constant when $\beta \leq 1$, but in fact $\lambda = 1+ o(1)$, when $\beta >1$.
For $\beta \leq 1$, we will only need the upper bound, which is obtained using the lower bound in (\ref{eq:sin_bounds}). Hence, in this case
$\lambda = 1/\pi^{\beta}$. In other words, we can take
$$ \lambda = \begin{cases} 1, & \mbox{if $\beta>1$} \\
{1\over \pi^\beta}, & \mbox{if $\beta \leq 1$} \end{cases}. $$

At this point we can make a convenient change of variables, setting
\[
{ z_1 :=\lambda^{1/\beta}A_{u,w} \frac{\theta_{u,w}}{2}, \quad \textnormal{ and } \quad z_2 :=\lambda^{1/\beta}A_{v,w} \frac{\theta_{v,w}}{2}.}
\]
Note that for each triple of vertices $u,v,w$ we have
\begin{equation}\label{c1_c2}
\begin{split}
& C_{w,v}=\frac{A_{u,v}}{A_{u,w}}=e^{\frac{\zeta}{2}(t_w-t_v)}, \quad \quad C_{w,u} =\frac{A_{u,v}}{A_{v,w}}=e^{\frac{\zeta}{2}(t_w-t_u)},\\
& C_{v,w}=\frac{A_{u,w}}{A_{u,v}}=e^{\frac{\zeta}{2}(t_v-t_w)}, \quad \quad C_{v,u} =\frac{A_{u,w}}{A_{v,w}}=e^{\frac{\zeta}{2}(t_v-t_u)},\\
& C_{u,w}=\frac{A_{v,w}}{A_{u,v}}=e^{\frac{\zeta}{2}(t_u-t_w)}, \quad \quad C_{u,v} =\frac{A_{v,w}}{A_{u,w}}=e^{\frac{\zeta}{2}(t_u-t_v)}.
\end{split}
\end{equation}
Thus, we obtain:
\medskip

\noindent
For $\beta>1$
\begin{equation*}
\begin{split}
\E & ( T(u,v;w)  \mathbf{1}_{\{ \tilde{\theta}_{u,w}\leq \vartheta_{u,w}\leq \hat{\theta}_{u,w},
-\hat{\theta}_{v,w}\leq \vartheta_{v,w}\leq -\tilde{\theta}_{v,w}\} } \mid t_u,t_v, t_w)   \\
&= {(1+ o(1)) \over \pi^2\,A_{u,w}A_{v,w}}{ \int_{\omega(N)^{-1}/2}^{\omega(N)/2} \int_{\omega(N)^{-1}/2}^{\omega(N)/2}\, } \frac{1}{z_1^\beta+1}\frac{1}{z_2^\beta+1}\frac{ d z_1 d z_2}{ (C_{w,v} z_1+C_{w,u} z_2)^\beta+1 } \\
& = {(1+ o(1)) \over \pi^2\,A_{u,w}A_{v,w}}\int_{0}^{\infty} \int_{0}^{\infty} \frac{1}{z_1^\beta+1}\frac{1}{z_2^\beta+1}\frac{1}{ (C_{w,v} z_1+C_{w,u} z_2)^\beta+1 } \, d z_1 d z_2,
\end{split}
\end{equation*}
(where the last equality follows from the fact that the latter integral is finite).
\medskip

\noindent
For $\beta=1$
\[
\begin{split}
% & \frac{1}{A_{uw}A_{vw}}\int_{1/(2\omega(N))}^{A_{uw}\pi/2} \int_{1/(2\omega(N))}^{A_{vw}\pi/2} \frac{1}{z_1+1}
% \frac{1}{z_2+1}\frac{1}{C_{w,v} z_1+C_{w,u} z_2+1 } \, d z_1 d z_2\\
% &\leq
\E & ( T(u,v;w) \mathbf{1}_{\{ \tilde{\theta}_{u,w}\leq \vartheta_{u,w}\leq \hat{\theta}_{u,w},
-\hat{\theta}_{v,w}\leq \vartheta_{v,w}\leq -\tilde{\theta}_{v,w}\} }  \mid t_u,t_v, t_w)   \\
& \lesssim \frac{1}{A_{u,w}A_{v,w}}\int_{1/(2\pi \omega(N))}^{A_{v,w}/2} \int_{1/(2\pi \omega(N))}^{A_{u,w}/2} \frac{1}{z_1+1}
\frac{1}{z_2+1}\frac{1}{ C_{w,v} z_1+C_{w,u} z_2+1 } \, d z_1 d z_2.
\end{split}
\]
Finally, for $\beta<1$ we have
\[
\begin{split}
% & \frac{4}{A_{uw}A_{vw}}\int_{\omega(N)/2}^{A_{uw}\pi/2} \int_{\omega(N)/2}^{A_{vw}\pi/2} \frac{1}{z_1^\beta+1}\frac{1}{z_2^\beta+1}\frac{1}{ (C_{w,v} z_1+C_{w,u} z_2)^\beta+1 } \, d z_1 d z_2  \\
\E & (T(u,v;w)  \mathbf{1}_{\{ \tilde{\theta}_{u,w}\leq \vartheta_{u,w}\leq \hat{\theta}_{u,w},
-\hat{\theta}_{v,w}\leq \vartheta_{v,w}\leq -\tilde{\theta}_{v,w}\} }  \mid t_u,t_v, t_w)  \\
& \lesssim \frac{1}{A_{u,w}A_{v,w}}~\int_{\omega(N)/2\pi}^{A_{v,w}/2} \int_{\omega(N)/2\pi}^{A_{u,w}/2}
\frac{1}{z_1^\beta+1}\frac{1}{z_2^\beta+1}\frac{1}{ (C_{w,v} z_1+C_{w,u} z_2)^\beta+1 } \, d z_1 d z_2 .
\end{split}
\]
The proof of the lemma is now complete.
\end{proof}
\medskip

\noindent
We now focus on the second term in (\ref{eq:ProbSplit}). Recall the definition of $\vartheta(u,v;w)$ from (\ref{eq:def_vartheta}).
\begin{Lemma} \label{lem:TCompl}
{ Recall the definitions of $\beta'$ and $\delta$ from \eqref{eq:beta_prime} and \eqref{eq:def_delta_1} respectively, and}
let $u,v,w \in \D$ be three distinct vertices. Then uniformly for $t_u, t_v, t_w < R/2 - \omega (N)$
\[
\E  \Bigl  ( T(u,v;w) \left ( 1- \vartheta(u,v;w) \right ) \mid t_u,t_v,t_w \Bigr ) =
%\left \{
%\begin{array}{ll}
%o\left( {1 \over A_{u,w} A_{v,w}} \right), & \textnormal{if }\beta > 1 \\[1.2ex]
%o \left( {\ln (A_{v,w}) \over A_{u,w} A_{v,w}} \right), & \textnormal{if } \beta =1 \\
%o\left( \frac{1}{A_{u,w}^{\beta}A_{v,w}^{\beta}} \right), & \textnormal{if }\beta< 1
%\end{array}
%\right ..
{ o\left ( \frac{\left ( \ln A_{v,w}\right )^\delta}{( A_{u,w} A_{v,w})^{\beta'}}\right ).}
\]
\end{Lemma}
\begin{proof}
Using the union bound, we can bound this term as follows:
\begin{equation}\label{p_delta_1}
\begin{split}
\E & \Bigl  ( T(u,v;w) \left ( 1- \mathbf{1}_{\{ \tilde{\theta}_{u,w}\leq \vartheta_{u,w}\leq \hat{\theta}_{u,w}, -\tilde{\theta}_{v,w}\geq\vartheta_{v,w}
\geq -\hat{\theta}_{v,w}\} }\right ) \mid t_u,t_v,t_w \Bigr )\\
& \leq  \E \left (  T(u,v;w)  \mathbf{1}_{\{ \vartheta_{u,w}> \hat{\theta}_{u,w}\} } \mid t_u,t_v,t_w \right ) \\
& \quad +  \E \left (  T(u,v;w)  \mathbf{1}_{\{ \vartheta_{v,w}< -\hat{\theta}_{v,w}\} }\mid t_u,t_v,t_w \right ) \\
&  \quad + \E \left ( T(u,v;w)  \mathbf{1}_{\{ 0<\vartheta_{u,w}< \tilde{\theta}_{u,w}\} } \mid t_u,t_v,t_w \right )\\
& \quad + \E \left (  T(u,v;w)  \mathbf{1}_{\{ 0>\vartheta_{v,w} > -\tilde{\theta}_{v,w}\} } \mid t_u,t_v,t_w \right ).
\end{split}
\end{equation}
We now obtain upper bounds on each term in the right-hand side of the above.

We bound the first term as follows:
\[
\P  \left ( u\sim w,  \hat{\theta}_{u,w} < \vartheta_{u,w} \leq \pi \mid  t_u,t_v,t_w \right ) =
{1\over 2 \pi}\int_{\hat{\theta}_{u,w}}^{\pi}p_{u,w}\, d\vartheta_{u,w}.
\]
When $\beta \leq 1$, by (\ref{tilda_hat_definition}) we have
\begin{equation} \label{eq:1stTerm_beta<=1}
\int_{\hat{\theta}_{u,w}}^{\pi}p_{u,w}\, d\vartheta_{u,w} \leq \tilde{\theta}_{u,w} =
\left \{ \begin{array}{ll}
{1\over \omega (N) A_{u,v}} = o \left( {1\over A_{u,w}}\right), & \mbox{if $\beta =1$} \\
{\omega (N) \over A_{u,v}} = o \left( {1\over A_{u,w}^\beta} \right),& \mbox{if $\beta <1$}
\end{array} \right ..
\end{equation}
For $\beta >1$ we need to work slightly more.
We have
\[
\begin{split}
\P & \left ( u\sim w,  \hat{\theta}_{u,w} < \vartheta_{u,w} \leq \pi \mid  t_u,t_v,t_w \right ) =
{1\over 2 \pi}\int_{\hat{\theta}_{u,w}}^{\pi}p_{u,w}\, d\vartheta_{u,w} \\
&\stackrel{(\ref{exp_sin})}{=}{1\over 2\pi} \int_{\hat{\theta}_{u,w}}^{\pi} \frac{1}{C\, A^\beta_{u,w}\sin^\beta (\theta_{u,w}/2)+1}\, d\vartheta_{u,w} \
 \leq { \frac{1+o(1)}{2\pi A^\beta_{u,w}}} \int_{\hat{\theta}_{u,w}}^{\pi} \frac{d\vartheta_{u,w}}{\sin^\beta (\theta_{u,w}/2),} 
\end{split}
\]
{ since $C=1+o(1)$ uniformly for $u,v$.}
Using the well-known inequality
\[
\sin \left (\frac{\theta}{2} \right )\geq \frac{\theta}{\pi} \quad \textnormal{for all } \theta \in [0,\pi],
\]
we can bound the previous term by
\[
\begin{split}
 & \frac{1}{2\pi C\, A^\beta_{u,w}}\int_{\hat{\theta}_{u,w}}^{\pi} \frac{1}{\sin^\beta (\theta_{u,w}/2)}\, d\theta_{u,w}
 \leq \frac{\pi ^{\beta-1}}{2  C\, A^\beta_{u,w}} \int_{\hat{\theta}_{u,w}}^{\pi} \theta_{u,w}^{-\beta} \, d\theta_{u,w} \\
& = \frac{\pi ^{\beta - 1}}{2C(\beta-1) A^\beta_{u,w}} \left ( \hat{\theta}_{u,w}^{-\beta+1}-\pi^{-\beta+1}\right ) \asymp \frac{1}{A^\beta_{u,w}}  \frac{A^{\beta-1}_{u,w}}{\omega(N)^{\beta-1}}\asymp  \frac{A^{-1}_{u,w}}{\omega(N)^{\beta-1}}.
\end{split}
\]
Hence, for $\beta>1$, we have that
\begin{equation} \label{int_hat_theta}
\P \left ( u\sim w, \vartheta_{u,w}> \hat{\theta}_{u,w} \mid  t_u,t_v,t_w \right )  \lesssim  \frac{1}{A_{u,w}\, \omega(N)^{\beta-1}}.
\end{equation}
We can now return to Equation (\ref{p_delta_1}) and use the above estimates to bound the first term of the right-hand side:
\[
\begin{split}
\E & \left ( T(u,v;w) \mathbf{1}_{ \vartheta_{u,w}> \hat{\theta}_{u,w}} | t_u,t_v,t_w \right )  \leq \P \left ( u\sim w, v\sim w, \vartheta_{u,w}> \hat{\theta}_{u,w} |  t_u,t_v,t_w \right )\\
& =  \P \left ( u\sim w, \vartheta_{u,w}> \hat{\theta}_{u,w} \mid  t_u,t_v,t_w \right ) \P \left ( v\sim w \mid  t_u,t_v,t_w \right )\\
& \stackrel{Lemma~\ref{lemma_2.4_evolution}, (\ref{eq:1stTerm_beta<=1}), (\ref{int_hat_theta})}{\lesssim}
\left \{ \begin{array}{ll}o\left( \frac{1}{A_{u,w}A_{v,w}} \right), & \mbox{if $\beta > 1$}\\
o\left( {\ln (A_{v,w}) \over A_{u,w} A_{v,w}} \right), & \mbox{if $\beta= 1$}\\
o\left( \frac{1}{A_{u,w}^{\beta} A_{v,w}^{\beta}} \right), & \mbox{if $\beta < 1$}
\end{array} \right ..
\end{split}
\]
The first equality is due to the independence of the relative positions of $u$ and $v$ with respect to $w$, together with the independence of
the edges, given the positions of the vertices. The very same calculation can be used in order to deduce the same bound on
$\E \left (  T(u,v;w) \mathbf{1}_{\{ \vartheta_{v,w}< - \hat{\theta}_{v,w}\} } \mid t_u,t_v,t_w \right ) $.

Now we use the fact that $p_{u,w}(\theta)\leq 1 $ to bound
\[
\begin{split}
\E & \left (  T(u,v;w) \mathbf{1}_{\{\vartheta_{u,w}< \tilde{\theta}_{u,w}\} }\mid t_u, t_v, t_w \right )
\leq \P (v\sim w,  0 < \vartheta_{u,w}<\tilde{\theta}_{u,w}\mid t_u, t_v, t_w ) \\ 
& = \P (v\sim w \mid t_u, t_v, t_w) \P (0 < \vartheta_{u,w}<\tilde{\theta}_{u,w}\mid t_u, t_v, t_w ) \\
&=
\left \{ \begin{array}{ll}
o( A_{u,w}^{-1} A_{v,w}^{-1}), & \textnormal{if }\beta\geq 1 \\
o( A_{u,w}^{-\beta}A_{v,w}^{-\beta}), & \textnormal{if }\beta< 1
\end{array} \right ..
%& = {1\over 2 \pi}~\int_0^ {\tilde{\theta}_{u,w}}p_{u,w}(\theta_{u,w})d \vartheta_{u,w} \leq  \tilde{\theta}_{u,w}.%=o( A_{u,w}^{-1}), \\
\end{split}
\]
In the last two equalities we have used the fact that the event $\{ v \sim w \}$ is independent of the event $\{\vartheta_{u,w}< \tilde{\theta}_{u,w}\}$ together with Lemma~\ref{lemma_2.4_evolution}.
The same bound can be deduced for
$\E \left (  T(u,v;w) \mathbf{1}_{\{0 > \vartheta_{v,w} > -\tilde{\theta}_{v,w} \} } \mid t_u,t_v,t_w \right ) $.
\end{proof}
The above lemma implies that
\begin{Corollary} \label{cor:typical_large_angles}
For $\beta >0$, we have
$$ \E  \Bigl  ( T(u,v;w) \left ( 1- \vartheta(u,v;w)\right ) \mid t_u,t_v,t_w \Bigr ) = o\left( \P \left( u\sim w, v \sim w \mid t_u,t_v, t_w\right ) \right ),$$
uniformly over all $t_u,t_v, t_w \leq R/2- \omega (N)$.
\end{Corollary}
Hence, for any $\beta > 0$ the contribution of these terms to $\E (\widehat{T}')$ is $o \left( \E (\widehat{\Lambda}) \right)$.

Thus we need to focus on the terms that are covered by Lemma~\ref{lem:TMain}.

\noindent
Using (\ref{c1_c2}), we write
\[
\frac{1}{A_{u,w}A_{v,w}}=\left({\nu \over N}\right)^2~\frac{e^{2\zeta t_w}}{C_{w,v}C_{w,u}}.
\]
Thus we can rewrite the right-hand sides of (\ref{prop:equ_prob_b>1}), (\ref{prop:equ_prob_b=1}) and (\ref{prop:equ_prob_b<1}) as follows:
\begin{equation}\label{Iuvw}
I(u,v;w):= \left({\nu \over N}\right)^2 \frac{e^{2\zeta t_w}}{C_{w,v}C_{w,u}}\int_{D_z} \frac{1}{z_1^\beta+1}\frac{1}{z_2^\beta+1}\frac{d z_1 d z_2}{  (C_{w,v} z_1+C_{w,u} z_2)^\beta+1 } ,
\end{equation}
where $D_z$ is the domain where $z_1$ and $z_2$ range; we have
{ 
$$D_z= \begin{cases}[0,\infty)^2, & \mbox{if $\beta > 1$} \\
 \bigl [(2\pi \omega(N))^{-1},A_{v,w}/2\bigr ]\times \bigl [(2\pi \omega(N))^{-1},A_{u,w}/2\bigr ], & \mbox{if $\beta=1 $} \\
\bigl [(\omega(N)/(2\pi), A_{v,w}/2\bigr ]\times \bigl [\omega(N)/(2\pi), A_{u,w}/2\bigr ] , & \mbox{if $\beta < 1$}
\end{cases}.
$$ }
The above together with Corollary~\ref{cor:typical_large_angles} imply the following statement.
\begin{Lemma} \label{lem:Final_Expressions}
For any distinct $u,v,w \in \V$ we have the following:\\ if $\beta>1$:
\[
\begin{split}
 \E & (T(u,v;w)\vartheta(u,v;w))=(1+o(1)) \int_{D_t} I(u,v;w) \bar{\rho}_N(t_u)\bar{\rho}_N(t_v)\bar{\rho}_N(t_w) dt_u dt_v dt_w;
\end{split}
\]
while if $\beta \leq 1$
\[
 \E (T(u,v;w)\vartheta(u,v;w)) \lesssim \int_{D_t} I(u,v;w) \bar{\rho}_N(t_u)\bar{\rho}_N(t_v)\bar{\rho}_N(t_w) dt_u dt_v dt_w.
\]
where $D_t:=[0, R/2-\omega(N)]^3$. Also, for any $\beta >0$
$$ N^3 \E (T(u,v;w)(1-\vartheta(u,v;w)))= o \left( \E (\widehat{\Lambda}) \right).  $$
\end{Lemma}

\subsection{Proof of Theorem \ref{thm:typicalglobalclustering}} \label{sec:proof_of_theorem_2}

{ Recall that $\beta>1$.
Let $t<R$ be a positive constant, and consider $N$ and $R$ so large, that $R/2-\omega(N)>t$.} 
Let $\widehat{T}_t$ denote the number of triangles in $\G(N;\zeta,\alpha,\beta,\nu)$ whose vertices have type at most $t$.
{ Note that in this part the triangles are automatically typical, hence $\widehat{T}_t=T_t $ a.a.s..}
Similarly, we let $\widehat{\Lambda}_t$ denote the number of incomplete triangles all of whose vertices have type at most $t$.
For three distinct vertices $u,v$ and $w$ we let $T_t(u,v;w)$ be defined as $T(u,v;w)$ but with the additional restriction that
$t_u,t_v,t_w \leq t$.

Finally, let $\widehat{T}'_t$ be defined as $\widehat{T}'$ with the variables $T(u,v;w)$ replaced by $T_t(u,v;w)$.
Note that the analogue of (\ref{eq:equivalence}) holds  between $\widehat{T}_t$ and $\widehat{T}_t'$.

By Lemma~\ref{lemma_2.4_evolution}, it follows that
\begin{equation} \label{eq:Lambda_t_Exp}
\begin{split}
\E & \left( \widehat{\Lambda}_t \right) = (1+o(1))3 {N \choose 3}\left(C_\beta\right)^2 \int_{[0,t]^3} {1\over A_{w,u} A_{w,v}} \bar{\rho}_{N}(t_u)
\bar{\rho}_{N}(t_v) \bar{\rho}_{N}(t_w) dt_u dt_v dt_w \\
&\stackrel{Claim~\ref{clm:density_approx}}{=} (1+o(1)) { \frac{1}{2}}N \left(\nu C_{\beta}\right)^2  \alpha^3
 \int_{[0,t]^3} e^{{\zeta \over 2}(t_u + t_v) + \zeta t_w} e^{-\alpha (t_u + t_v + t_w)} dt_u dt_v dt_w.
\end{split}
\end{equation}
One can show that $\widehat{\Lambda}_t$ is concentrated around its expected value through a second moment argument that is
very similar to that in Section~\ref{sect:2nd_moment} and we omit.
Thus,
\begin{equation} \label{eq:Lambda_t}
 \widehat{\Lambda}_t = \E \left( \widehat{\Lambda}_t \right) (1+o_p(1)).
\end{equation}
Regarding $\widehat{T}_t$, we use the following fact.
\begin{Claim}
For $\beta > 1$ and for $N$ large enough we have
\[
 \widehat{T}_t' - \widehat{T}_t \leq  \sum_{w \in \V} \sum_{\stackrel{(u,v): }{ u,v \in \V \setminus \{ w\} }} 
T_t(u,v;w) \left( 1 - \vartheta(u,v;w) \right).
\]
\end{Claim}
\begin{proof}
We write
\begin{equation} \label{eq:T'}
\begin{split}
\widehat{T}_t' &= \sum_{w \in \V} \sum_{\stackrel{(u,v): }{ u,v \in \V \setminus \{ w\} }} T_t (u,v;w) \\
& = \sum_{w \in \V} \sum_{\stackrel{(u,v): }{ u,v \in \V \setminus \{ w\} }} T_t (u,v;w) \left( \vartheta (u,v;w) + 1- \vartheta (u,v;w) \right).
\end{split}
\end{equation}
The definitions in (\ref{tilda_hat_definition}) imply that for $\beta > 1$ we have $\hat{\theta}_{u,w}, \hat{\theta}_{v,w},
 \hat{\theta}_{u,v}=o(1)$.  Therefore, for $N$ sufficiently large, if $\vartheta(u,v;w)=1$, then $\vartheta(u,w;v) = \vartheta(v,w;u) = 0$.
Furthermore, if $\vartheta (u,v;w) = 1$, then $\vartheta (v,u;w) = 0$ and also $\vartheta (w,u;v) = \vartheta (w,v;u) = 0$.

Hence, if $T_t(u,v,w)$ denotes the indicator random variable that is equal to 1 if and only if the vertices $u,v,w$ form a triangle and all have types at most $t$, we have
\begin{equation} \label{eq:T}
\begin{split}
\widehat{T}_t & \geq \sum_{w \in \V} \sum_{\stackrel{(u,v): }{ u,v \in \V \setminus \{ w\} }} T_t (u,v,w)  \vartheta (u,v;w) \\
& \geq \sum_{w \in \V} \sum_{\stackrel{(u,v): }{ u,v \in \V \setminus \{ w\} } } T_t (u,v;w)  \vartheta (u,v;w).
\end{split}
\end{equation}
Now, subtracting (\ref{eq:T}) from (\ref{eq:T'}), the claim follows.
\end{proof}

Therefore, by Corollary~\ref{cor:typical_large_angles}, we have
\begin{equation} \label{eq:Exp_differences}
\E \left ( \widehat{T}_t' - \widehat{T}_t \right) = o \left( \E (\widehat{\Lambda}_t) \right).
\end{equation}
Now, by (\ref{prop:equ_prob_b>1}) in Lemma~\ref{lem:TMain} we have
\begin{equation} \label{eq:T'_t_Exp}
\begin{split}
\E \left( \widehat{T}_t' \right) &= (1+o(1))6 {N \choose 3}{1\over \pi^2}
 \int_{[0,t]^3} I(u,v;w)\bar{\rho}_{N}(t_u)
\bar{\rho}_{N}(t_v) \bar{\rho}_{N}(t_w) dt_u dt_v dt_w \\
&\stackrel{Claim~\ref{clm:density_approx}}{=} (1+o(1)) N  \left({\nu \over \pi} \right)^2  \alpha^3
 \int_{[0,t]^3} e^{{\zeta \over 2}(t_u + t_v) + \zeta t_w} e^{-\alpha (t_u + t_v + t_w)} \times \\
&\left[\int_{D_z} \frac{1}{z_1^\beta+1}\frac{1}{z_2^\beta+1}\frac{1}{  (C_{w,v} z_1+C_{w,u} z_2)^\beta+1 } \, d z_1 d z_2 \right]
dt_u dt_v dt_w.
\end{split}
\end{equation}
The concentration of $\widehat{T}_t$ around its expected value can be shown using a second moment argument similar to that used in
Lemma~\ref{ET2} (we omit the details), from which we deduce that 
$$\widehat{T}_t = \E \left( \widehat{T}_t \right) (1+o_p(1)).  $$
But~(\ref{eq:Exp_differences}) implies that $\E \left( \widehat{T}_t \right) = \E \left( \widehat{T}_t' \right) + o \left( \E (\widehat{\Lambda}_t) \right)$. As $ \E (\widehat{\Lambda}_t) = \Theta \left( \E \left( \widehat{T}_t' \right) \right)$ we deduce that 
$$ \widehat{T}_t = \E \left( \widehat{T}_t' \right) (1+o_p(1)). $$
This combined with (\ref{eq:Lambda_t}) imply the statement of Theorem~\ref{thm:typicalglobalclustering}.

The value of $L_\infty (\beta,\zeta,\alpha)$ is also deduced as above using Proposition~\ref{prop_concentration} and taking 
$t=R/2 - \omega (N)$ which 
is equivalent (up to a $1+o(1)$ factor) to taking the integrals up to $t=\infty$.

\section{Proof of Proposition~\ref{prop:complete_triangles}}\label{sect:proof_prop_22}
In this section we prove separately the results for $\beta>1$, $\beta=1$ and $\beta<1$.
For $\beta\leq 1$, we will give only an upper bound { for $\E (\widehat{T})$} (cf. Sections~\ref{sect_proof=1},~\ref{sect_proof<1}). For $\beta > 1$, we
will consider two cases, namely $\zeta/ \alpha < 1$ and $\zeta /\alpha \geq 1$. In the former, we will show that
$\E (\widehat{T}) \asymp \E (\widehat{\Lambda})$. Note that the upper bound holds trivially, as $3\widehat{T} \leq  \widehat{\Lambda}$.
We will deduce only a matching lower bound in the next section. For $\zeta /\alpha \geq 1$, we will deduce an upper bound.

\subsection{Proof of Proposition~\ref{prop:complete_triangles}(i) ($\beta >1$)}\label{sect_proof>1}

\subsection*{Case $\zeta/\alpha < 1$}
We will deduce a lower bound on $\E (T(u,v;w)\vartheta(u,v;w))$ integrating $I(u,v;w)$ over the sub-domain of $D_t$ which
is $D_t':= \{(t_u,t_v,t_w) \ : \ 0 < t_u,t_v < t_w \}$. Note that in this case $C_{w,u}, C_{w,v} > 1$. Hence, we can bound
from below the double integral that appears in (\ref{Iuvw}) as follows:
\begin{equation*} \label{eq:Integral_Low}
\begin{split}
& \int_0^{\infty} \int_0^{\infty}\frac{1}{z_1^\beta+1}\frac{1}{z_2^\beta+1}\frac{1}{  (C_{w,v} z_1+C_{w,u} z_2)^\beta+1 } \, d z_1 d z_2  \\
& \geq \int_0^{\infty} \int_0^{\infty}
\frac{1}{(C_{w,v}z_1)^\beta+1}\frac{1}{(C_{w,u}z_2)^\beta+1}\frac{1}{  (C_{w,v} z_1+C_{w,u} z_2)^\beta+1 } \, d z_1 d z_2 \\
& = {1\over C_{w,v} C_{w,u}}~ \int_0^{\infty} \int_0^{\infty} {1\over x_1^\beta + 1}~{1\over x_2^\beta + 1}
\frac{1}{(x_1 + x_2)^\beta+1} dx_1 dx_2.
\end{split}
\end{equation*}
Therefore,
\begin{equation*}
I(u,v;w) \geq \left({\nu \over N} \right)^2 {e^{2\zeta t_w} \over C_{w,v}^2 C_{w,u}^2}
\int_0^{\infty} \int_0^{\infty} {1\over x_1^\beta + 1}~{1\over x_2^\beta + 1} \frac{1}{(x_1 + x_2)^\beta+1} dx_1 dx_2,
\end{equation*}
which in turn yields
\begin{equation*}
\begin{split}
\E (T(u,v;w)\vartheta(u,v;w)) \geq &
\left({\nu \over N} \right)^2
\left[ \int_{[0,\infty)^2} {1\over x_1^\beta + 1}~{1\over x_2^\beta + 1} \frac{ dx_1 dx_2}{(x_1 + x_2)^\beta+1} \right] \times
\\ &\int_{D_t'} {e^{2\zeta t_w} \over C_{w,v}^2 C_{w,u}^2} \bar{\rho}_N(t_u)\bar{\rho}_N(t_v)\bar{\rho}_N(t_w) dt_u dt_v dt_w.
\end{split}
\end{equation*}
We will show that for $\zeta /\alpha < 1$, the latter integral is $\Omega (1)$.
Indeed, we have
$$ {e^{2\zeta t_w} \over C_{w,v}^2 C_{w,u}^2} = e^{\zeta (t_u + t_v)},$$
whereby using Claim~\ref{clm:density_approx} (for large $N$) we obtain
\[
\begin{split}
& \int_{D_t'} {e^{2\zeta t_w} \over C_{w,v}^2 C_{w,u}^2} \bar{\rho}_N(t_u)\bar{\rho}_N(t_v)\bar{\rho}_N(t_w) dt_u dt_v dt_w  \gtrsim
{1\over 2} \int_{D_t'} e^{\zeta (t_u + t_v) - \alpha (t_u + t_v + t_w)} dt_u dt_v dt_w \\
& \asymp \int_0^{R/2 - \omega (N)} \int_0^{t_w} \int_0^{t_w} e^{(\zeta -\alpha)(t_u + t_v) - \alpha t_w} dt_u dt_v dt_w \\
& = \int_0^{R/2 - \omega (N)} \left[ \int_0^{t_w} e^{(\zeta -\alpha) t_u} dt_u \right]^2 e^{-\alpha t_w} dt_w \, \gtrsim\,  1.
\end{split}
\]
Thus, after recalling the definition of $\vartheta(u,v;w)$ from (\ref{eq:def_vartheta}), we deduce that
$$ \E (T(u,v;w)\vartheta(u,v;w)) \gtrsim {1\over N^2} \quad  \Rightarrow \quad \E (\widehat{T}') \gtrsim {N \choose 3} \E (T(u,v;w)\vartheta(u,v;w)) \gtrsim N.$$
Now (\ref{eq:equivalence}) implies that
$$ \E (\widehat{T}) \gtrsim N. $$

\subsection*{Case $1\leq \zeta/\alpha < 2$}

In this range we provide an upper bound on $ \E (\widehat{T'})$ and show that it is $o \left( \E (\widehat{\Lambda}) \right)$.
By (\ref{eq:equivalence}), this is clearly enough to deduce the second part of Proposition~\ref{prop:complete_triangles}(i).
We write
\[
\begin{split}
\E (\widehat{T}') & \leq N^3  \E (T(u,v;w)) \\
& = N^3 \left( \E (T(u,v;w)\vartheta(u,v;w)) +  \E (T(u,v;w)(1-\vartheta(u,v;w))) \right),
\end{split}
\]
where $u,v,w$ are three distinct vertices.
By the second part of Lemma~\ref{lem:Final_Expressions} the second term is $o\left(\E (\widehat{\Lambda})\right)$.
We will also show that
\begin{equation} \label{eq:1stTerm_ToShow}
N^3 \left( \E (T(u,v;w)\vartheta(u,v;w) \right) = o\left(\E (\widehat{\Lambda})\right).
\end{equation}

To this end, we split the domain of the integral of $I(u,v;w)$ into three sub-domains and bound $I(u,v;w)$ separately on each one of them.
In particular, we define
\[
\begin{split}
D_t^{(1)} & := \{ (t_u,t_v,t_w) \ : \ t_u, t_v > t_w \} \\
D_t^{(2)} & := \{ (t_u,t_v, t_w)\ : \ t_u \leq t_w \} \\
D_t^{(3)} & := \{ (t_u, t_v, t_w) \ : \ t_v \leq t_w\}.
\end{split}
\]
It is clear that the last two sub-domains are not disjoint but as we are interested only in upper bounds this does not create any issues. It is immediate to see that
\begin{equation} \label{eq:ToProve}
\begin{split}
\E (T(u,v;w)\vartheta(u,v;w)) & \leq \int_{D_t^{(1)}}  I(u,v;w) \bar{\rho}_N(t_u)\bar{\rho}_N(t_v)\bar{\rho}_N(t_w) dt_u dt_v dt_w \\
&
+ \int_{D_t^{(2)}}  I(u,v;w) \bar{\rho}_N(t_u)\bar{\rho}_N(t_v)\bar{\rho}_N(t_w) dt_u dt_v dt_w  \\
& +\int_{D_t^{(3)}}  I(u,v;w) \bar{\rho}_N(t_u)\bar{\rho}_N(t_v)\bar{\rho}_N(t_w) dt_u dt_v dt_w.
\end{split}
\end{equation}
We will bound from above each one of these three integrals. In fact, we will do so only for the first two -- the last one
can be treated exactly as the second one.

To bound the first integral, we use the following upper bound on $I (u,v;w)$:
\begin{equation*}
I (u,v;w) \leq \left({\nu \over N} \right)^2 {e^{2\zeta t_w} \over C_{w,v} C_{w,u}} \int_{D_z} {1\over z_1^{\beta} + 1} {1\over z_2^\beta +1} d z_1 d z_2.
\end{equation*}
Also,
$$  {e^{2\zeta t_w} \over C_{w,v} C_{w,u}}  = e^{\zeta t_w + {\zeta \over 2} (t_u + t_v)}.$$
We now integrate this quantity over $D_t^{(1)}$ applying Claim~\ref{clm:density_approx} as
follows
\[
\begin{split}
& \int_{D_t^{(1)}} e^{\zeta t_w + {\zeta \over 2} (t_u + t_v)}  \bar{\rho}_N(t_u)\bar{\rho}_N(t_v)\bar{\rho}_N(t_w) dt_u dt_v dt_w  \\
& \lesssim \int_0^{R/2 -\omega (N)} \int_{t_w}^{R/2 - \omega (N)} \int_{t_w}^{R/2 - \omega (N)}
e^{\zeta t_w + {\zeta \over 2} (t_u + t_v) - \alpha (t_u + t_v + t_w)} dt_u dt_v dt_w \\
&=
\int_0^{R/2 -\omega (N)} e^{\left( \zeta - \alpha \right) t_w} \left[ \int_{t_w}^{R/2 - \omega (N)}  e^{\left( \zeta /2 - \alpha \right) t_u} dt_u \right]^2 dt_w \\
& \stackrel{\zeta /\alpha < 2}{\lesssim}
\int_0^{R/2 -\omega (N)} e^{\left( \zeta - \alpha \right) t_w + 2 \left(\zeta/2 - \alpha \right)t_w}  dt_w =
\int_0^{R/2 -\omega (N)} e^{\left( 2\zeta - 3\alpha \right) t_w}  dt_w.
\end{split}
\]
If $\zeta/\alpha < 3/2$, then the above integral is $O(1)$, whereas if $\zeta /\alpha = 3/2$, then this is $O(R)$. Finally,
when $\zeta /\alpha > 3/2$, this is $O(N^{2 - 3 \alpha /\zeta})$. Hence,
\[
\begin{split}
\int_{D_t^{(1)}} & e^{\zeta t_w + {\zeta \over 2} (t_u + t_v)}  \bar{\rho}_N(t_u)\bar{\rho}_N(t_v)\bar{\rho}_N(t_w) dt_u dt_v dt_w
\lesssim
\begin{cases}
1, & \mbox{if $\frac{\zeta}{\alpha} = 1$} \\
N^{2 - 3\alpha /\zeta}, & \mbox{if $1 < \frac{\zeta }{\alpha }< 2$}
\end{cases}.
\end{split}
\]
Therefore,
\begin{equation} \label{eq:Part_I}
\begin{split}
N^3 & \int_{D_t^{(1)}}  I(u,v;w) \bar{\rho}_N(t_u)\bar{\rho}_N(t_v)\bar{\rho}_N(t_w) dt_u dt_v dt_w\\
&  \lesssim
\begin{cases}
N, & \mbox{if $\zeta /\alpha = 1$} \\
N^{3 - 3\alpha /\zeta}, & \mbox{if $1 < \zeta /\alpha < 2$}
\end{cases}.
\end{split}
\end{equation}
Note that both quantities are $o\left(\E ( \widehat{\Lambda} ) \right)$ (to see the latter note that
$3 - 3\alpha /\zeta < 2 - \alpha /\zeta$ which is equivalent to $1 < 2\alpha /\zeta$, that is, $\zeta /\alpha < 2$).

Now we consider the second integral in (\ref{eq:ToProve}). Note that on the sub-domain $D_t^{(2)}$ we have $C_{w,u}\geq 1$.
In this case, we bound the integral in (\ref{Iuvw}) as follows:
\[
\begin{split}
& \int_0^{\infty} \int_0^{\infty}
\frac{1}{z_1^\beta+1}\frac{1}{z_2^\beta+1}\frac{1}{  (C_{w,v} z_1+C_{w,u} z_2)^\beta+1 } \, d z_1 d z_2 \\
&
\leq \int_0^{\infty} \int_0^{\infty} \frac{1}{z_1^\beta+1}\frac{1}{  (C_{w,u} z_2)^\beta+1 } \, d z_1 d z_2 \\
&= {1\over C_{w,u}} \int_0^{\infty} \int_0^{\infty} \frac{1}{z_1^\beta+1}\frac{1}{  x_2^\beta+1 } \, d z_1 d x_2
= {1\over C_{w,u}} \left[ \int_0^{\infty}\frac{1}{z_1^\beta+1} dz_1 \right]^2.
\end{split}
\]
Hence,
\begin{equation}\label{eq:IUpper} I(u,v;w) \leq \left( {\nu \over N} \right)^2 {e^{2\zeta t_w} \over C_{w,v} C_{w,u}^2}
\left[ \int_0^{\infty}\frac{1}{z_1^\beta+1} dz_1 \right]^2,
\end{equation}
and
$$ {e^{2\zeta t_w} \over C_{w,v} C_{w,u}^2} = e^{{\zeta \over 2}(t_w + t_v) + \zeta t_u}.$$
To bound the second integral in (\ref{eq:ToProve}) we need to bound the integral of the above quantity over $D_t^{(2)}$.
Using Claim~\ref{clm:density_approx}, we have
\begin{equation*}
\begin{split}
&\int_{D_t^{(2)}} e^{{\zeta \over 2}(t_w + t_v) + \zeta t_u} \bar{\rho}_N(t_u)\bar{\rho}_N(t_v)\bar{\rho}_N(t_w) dt_u dt_v dt_w  \\
&\asymp \int_0^{R/2 - \omega (N)} \int_0^{R/2 - \omega (N)} e^{\left(\zeta/2  -\alpha \right) (t_w + t_v)} \left[ \int_{0}^{t_w} e^{(\zeta - \alpha)t_u} dt_u  \right] dt_v dt_w \\
&= \left[\int_0^{R/2 - \omega (N)} e^{\left( \zeta/2 -\alpha \right)t_v} dt_v \right]
\int_0^{R/2 - \omega (N)} e^{\left(\zeta/2  -\alpha \right) t_w} \left[ \int_{0}^{t_w} e^{(\zeta - \alpha)t_u} dt_u  \right] dt_w \\
&\stackrel{\zeta/\alpha < 2}{\asymp}  \int_0^{R/2 - \omega (N)} e^{\left(\zeta/2  -\alpha \right) t_w} \left[ \int_{0}^{t_w} e^{(\zeta - \alpha)t_u} dt_u  \right] dt_w.
\end{split}
\end{equation*}
Now, when $\zeta /\alpha = 1$, the above yields:
\begin{equation*}
\begin{split}
& \int_{D_t^{(2)}} e^{{\zeta \over 2}(t_w + t_v) + \zeta t_u} \bar{\rho}_N(t_u)\bar{\rho}_N(t_v)\bar{\rho}_N(t_w) dt_u dt_v dt_w \\
&  \asymp \int_0^{R/2 - \omega (N)} t_w e^{\left(\zeta/2  -\alpha \right) t_w} dt_w \asymp 1,
\end{split}
\end{equation*}
which by (\ref{eq:IUpper}) implies that
\begin{equation} \label{eq:2ndTerm}
\begin{split}
N^3 \int_{D_t^{(2)}}  I(u,v;w) \bar{\rho}_N(t_u)\bar{\rho}_N(t_v)\bar{\rho}_N(t_w) dt_u dt_v dt_w  \lesssim N =
o \left( \E (\widehat{\Lambda})\right).
\end{split}
\end{equation}
If $1<\zeta /\alpha < 2$, then
\begin{equation*}
\begin{split}
&\int_{D_t^{(2)}} e^{{\zeta \over 2}(t_w + t_v) + \zeta t_u} \bar{\rho}_N(t_u)\bar{\rho}_N(t_v)\bar{\rho}_N(t_w) dt_u dt_v dt_w  \\
& \lesssim \int_0^{R/2 - \omega (N)}e^{\left(3\zeta/2  -2\alpha \right) t_w} dt_w \asymp
\begin{cases}
1, & \mbox{if $1 <\zeta/\alpha < 4/3$} \\
R, & \mbox{if $\zeta/\alpha = 4/3$} \\
N^{3/2 - 2 \alpha/\zeta}, & \mbox{if $4/3 < \zeta /\alpha < 2$}
\end{cases}.
\end{split}
\end{equation*}
Therefore,
\begin{equation} \label{eq:2ndTerm_II}
\begin{split}
N^3 \int_{D_t^{(2)}}  I(u,v;w) \bar{\rho}_N(t_u)\bar{\rho}_N(t_v)\bar{\rho}_N(t_w) dt_u dt_v dt_w  \lesssim { R}N^{5/2 - 2\alpha /\zeta}.
\end{split}
\end{equation}
The latter is $o \left( \E (\widehat{\Lambda})\right)$, since $5/2 - 2\alpha /\zeta < 2 - \alpha /\zeta$ (which is equivalent to
$1/2 < \alpha /\zeta$, that is, $\zeta/\alpha <2$).
Hence, (\ref{eq:Part_I}), (\ref{eq:2ndTerm}), (\ref{eq:2ndTerm_II}) together with (\ref{eq:ToProve}) imply (\ref{eq:1stTerm_ToShow}).

\begin{Conjecture}
From the calculations seen in this section, we get two possible values for $\zeta/\alpha$ where the probability of triangles might have a sharp phase transition. We conjecture that the value where this is happening is $\zeta/\alpha=3/2$.
\end{Conjecture}

\subsection{Proof of Proposition \ref{prop:complete_triangles}(ii) ($\beta =1$)}\label{sect_proof=1}
In this section, we will prove (\ref{expected_T_1}) for the case $\beta = 1$.
{ The second part of Lemma~\ref{lem:Final_Expressions} implies}
$$ \E (\widehat{T}') = o \left( \E (\widehat{\Lambda} )\right), $$
which by (\ref{eq:equivalence}) implies Proposition~\ref{prop:complete_triangles} (ii).
%
%We focus on the proof of (\ref{eq:1stTerm_ToShow}).
We first bound the integral on the right-hand side of (\ref{Iuvw}) as follows:
\begin{equation} \label{eq:CorrectionInt}
\begin{split}
&\int_{D_z}  \frac{1}{z_1+1}\frac{1}{z_2 +1}\frac{1}{ C_{w,v} z_1+C_{w,u} z_2+1 } \, d z_1 d z_2 \\
&=
 \int_{1/(2\pi \omega(N))}^{A_{v,w} /2} \int_{1/(2\pi \omega (N))}^{A_{u,w} /2}
\frac{1}{z_1+1}\frac{1}{z_2 +1}
\frac{1}{ C_{w,v} z_1+C_{w,u} z_2+1 } \, d z_1 d z_2 \\
&\leq
\int_{1/(2\pi \omega (N))}^{A_{v,w}  /2} \int_{1/(2\pi \omega(N))}^{A_{u,w}  /2}
\frac{1}{z_1+1}\frac{1}{z_2 +1}
\frac{1}{ (C_{w,v} z_1+1)^{1/2} } \frac{1}{ (C_{w,u} z_2+1)^{1/2} } \, d z_1 d z_2 \\
&= \left[ \int_{0}^\infty  \frac{1}{z_1+1} \frac{1}{ (C_{w,v} z_1+1)^{1/2}}
~d z_1\right]
\left[\int_0^\infty \frac{1}{z_2 +1} \frac{1}{ (C_{w,u} z_2+1)^{1/2} }~d z_2 \right].
\end{split}
\end{equation}
Let us consider the first of these two integrals. We further bound it as follows:
\begin{equation} \label{eq:CorrectionIntI}
\begin{split}
& \int_{0}^\infty  \frac{1}{z_1+1} \frac{1}{ (C_{w,v} z_1+1)^{1/2}}
~d z_1 \\
&\leq
\begin{cases}
\int_{0}^\infty  \frac{1}{(z_1+1)^{3/2}}~d z_1, & \mbox{if $C_{w,v}\geq 1$} \\[1.3ex]
\int_{0}^\infty \frac{1}{z_1+1} \frac{1}{ (C_{w,v} z_1+ C_{w,v})^{1/2}}~dz_1=
{1\over C_{w,v}^{1/2}}~ \int_{0}^\infty  \frac{1}{(z_1+1)^{3/2}}~d z_1, &
\mbox{if $C_{w,v} < 1$}
\end{cases},
\end{split}
\end{equation}
The second integral in (\ref{eq:CorrectionInt}) is bounded analogously.
Now, we split $D_t$ into four sub-domains:
\begin{enumerate}
\item[1.] $t_u,t_v \leq t_w$ (that is, $C_{w,v},C_{w,u} \geq 1$);
\item[2.] $t_u \leq t_w$ but $t_v > t_w$ (that is, $C_{w,u} \geq 1$ and $C_{w,v}<1$);
\item[3.] $t_v \leq t_w$ but $t_u > t_w$ (that is, $C_{w,v} \geq 1$ and $C_{w,u}<1$);
\item[4.] $t_u, t_v > t_w$ (that is, $C_{w,v}, C_{w,u} < 1$).
\end{enumerate}
We denote by $D_i$ the domain considered in Case $i$, for $i=1,\ldots, 4$. Hence, we have
{ 
\begin{equation} \label{eq:1stsummand}
{N \choose 3}\E   \Bigl  ( T(u,v;w) \vartheta(u,v;w) \Bigr )  = {N \choose 3} \sum_{i=1}^4
\E   \Bigl  ( T(u,v;w)\vartheta(u,v;w)\mathbf{1}_{\{(t_u, t_v, t_w) \in D_i\}}
\Bigr ). 
%\stackrel{(\ref{eq:Case_I}), (\ref{eq:Case_II}),(\ref{eq:Case_IV})}{=} o \left( \E (\widehat{\Lambda} )\right).
\end{equation}
}
%We now focus on the proof of (\ref{eq:Case_I}), (\ref{eq:Case_II}), (\ref{eq:Case_IV}) which correspond to Cases 1,2,3 and 4.   
Setting
\[
\begin{split}
I'(u,v,w)
& := \left({\nu \over N} \right)^2 \frac{e^{2\zeta t_w}}{C_{w,v}C_{w,u}}
\max \left\{{1\over C_{w,v}^{1/2}},1 \right\}\max \left\{{1\over C_{w,u}^{1/2}},1 \right\}\times \\
& \quad \times
\left[ \int_0^\infty {1\over (z_1 +1)^{3/2}} dz_1 \right]^2,
\end{split}
\]
then (\ref{eq:CorrectionInt}) and (\ref{eq:CorrectionIntI}) imply that
$$ I(u,v,w) \leq I'(u,v,w).$$
We now consider each one of the four summands in (\ref{eq:1stsummand}) separately.

\subsubsection*{Case 1}
In this sub-domain, we have
$$ \frac{e^{2\zeta t_w}} {C_{w,v}C_{w,u}}  =
e^{\zeta t_w + {\zeta \over 2} \left( t_v + t_u \right)}. $$
We substitute this into the expression for $I'$ and we integrate over $D_1$ using
Claim~\ref{clm:density_approx}, thus obtaining
\begin{equation} \label{eq:Case_I}
\begin{split}
{N \choose 3}\E & \left( T (u,v;w) {\bf 1}_{(t_w, t_u, t_v) \in D_1} \right) \leq N^3 \int_{D_1} I'(u,v,w)
 dt_w dt_v d t_u\\
& \asymp  N\int_0^{R/2 - \omega (N)} \int_0^{t_w} \int_0^{t_w} e^{(\zeta -\alpha) t_w +
\left(\zeta/2 - \alpha \right)t_v + \left(\zeta/2 - \alpha \right)t_u} dt_u d t_v dt_w \\
&\asymp N\int_0^{R/2 - \omega (N)} e^{(\zeta -\alpha) t_w} \left[ \int_0^{t_w}
e^{\left(\zeta/2 - \alpha \right)t_u} dt_u \right]^2  dt_w \\
&\lesssim
N\int_0^{R/2 - \omega (N)} e^{(\zeta -\alpha) t_w} dt_w \lesssim
\begin{cases}
N, & \mbox{if $\zeta/\alpha < 1$} \\
RN, & \mbox{if $\zeta /\alpha = 1$} \\
N^{2-\alpha / \zeta}, & \mbox{if $\zeta /\alpha > 1$}
\end{cases}.
\end{split}
\end{equation}

\subsubsection*{Cases 2, 3}
Here, it suffices to consider only Case 2, where $t_u\leq t_w$ and $t_v > t_w$. Case 3 is treated in exactly the same way and gives the
same outcome. Here, $C_{w,u} \geq 1$ but $C_{w,v} < 1$. Hence by (\ref{eq:CorrectionIntI}) the factor that appears in $I'$ becomes
\begin{equation*}
\begin{split}
{e^{2\zeta t_w} \over C_{w,v}^{3/2} C_{w,u}}< {e^{2\zeta t_w} \over C_{w,v}^2 C_{w,u}} = e^{{\zeta \over 2}t_w + \zeta t_v + {\zeta \over 2} t_u}.
\end{split}
\end{equation*}
Thus, using again Claim~\ref{clm:density_approx} we have
\begin{equation} \label{eq:Case_II}
\begin{split}
& {N \choose 3}\E  \left( T (u,v;w) {\bf 1}_{(t_w, t_u, t_v) \in D_2} \right) \leq N^3 \int_{D_2} I'(u,v,w)
 dt_w dt_v d t_u\\
& \asymp  N \int_0^{R/2 - \omega (N)} \int_{t_w}^{R/2 - \omega (N)} \int_0^{t_w}
e^{\left({\zeta / 2}-\alpha \right)t_w + \left(\zeta - \alpha \right) t_v + \left({\zeta / 2} - \alpha \right) t_u}  d t_u dt_v dt_w \\
& \lesssim N \int_0^{R/2} \int_{0}^{R/2} \int_0^{R/2}
e^{\left({\zeta / 2}-\alpha \right)t_w + \left(\zeta - \alpha \right) t_v + \left({\zeta / 2} - \alpha \right) t_u}  d t_u dt_v dt_w \\
& = N \left[ \int_0^{R/2} e^{\left(\zeta - \alpha \right) t_v} dt_v \right] \left[ \int_{0}^{R/2}e^{\left({\zeta / 2} - \alpha \right) t_u} dt_u \right]^2
\lesssim \begin{cases}
N, & \mbox{if $\zeta/\alpha < 1$} \\
RN, & \mbox{if $\zeta /\alpha = 1$} \\
N^{2-\alpha / \zeta}, & \mbox{if $\zeta /\alpha > 1$}
\end{cases}.
\end{split}
\end{equation}
\subsubsection*{Case 4}
Now, the factor that appears in $I'$ becomes
\begin{equation*}
\begin{split}
{e^{2\zeta t_w} \over C_{w,v}^{3/2} C_{w,u}^{3/2}} = e^{{\zeta \over 2}t_w +{3\over 2}{\zeta \over 2} (t_v + t_u)}.
\end{split}
\end{equation*}
Hence,
\begin{equation} \label{eq:Case_IV}
\begin{split}
& {N \choose 3}\E  \left( T (u,v;w) {\bf 1}_{(t_w, t_u, t_v) \in D_4} \right) \leq N^3 \int_{D_4} I'(u,v,w)
 dt_w dt_v d t_u\\
& \asymp N \int_0^{R/2 - \omega (N)} \int_{t_w}^{R/2 - \omega (N)} \int_{t_w}^{R/2 - \omega (N)}
e^{\left({3\over 2}{\zeta \over 2} -\alpha \right)( t_u + t_v ) +
\left(\zeta /2- \alpha \right) t_w} dt_u dt_v dt_w \\
& \leq N \left[\int_0^{R/2 - \omega (N)} e^{(\zeta /2-\alpha)  t_w} dt_w \right]
\left[\int_{0}^{R/2} e^{\left(3\zeta/4 -\alpha \right)t_u}  dt_u \right]^2 \\
&
\lesssim \begin{cases}
N, & \mbox{if $\zeta/\alpha < 4/3$} \\
R^2N, & \mbox{if $\zeta /\alpha = 4/3$} \\
N^{1+3/2 - 2\alpha /\zeta }, & \mbox{if $\zeta /\alpha > 4/3$}
\end{cases}.
\end{split}
\end{equation}
The last exponent is equal to $2.5 - 2\alpha /\zeta < 2 - \alpha/\zeta$, as $1/2 < \alpha /\zeta$.

%\noindent
%Now, adding (\ref{eq:Case_I}), (\ref{eq:Case_II}) (or its counterpart for Case 3) and (\ref{eq:Case_IV}), we finally deduce (\ref{eq:1stsummand}).
{ Therefore, by plugging all these estimates into \eqref{eq:1stsummand}, we finally obtain
\[
{N \choose 3}\E   \Bigl  ( T(u,v;w) \vartheta(u,v;w) \Bigr )  = o \left( \E (\widehat{\Lambda} )\right).
\]
}

\subsection{Proof of Proposition~\ref{prop:complete_triangles}(ii) ($\beta <1$)}\label{sect_proof<1}

To prove Proposition~\ref{prop:complete_triangles}(ii) for $\beta < 1$, it also suffices to show (\ref{eq:1stTerm_ToShow}).
Recall (\ref{prop:equ_prob_b<1})
\[
\begin{split}
\E &( T (u,v;w)\vartheta(u,v;w)\mid t_u,t_v,t_w)\lesssim \\
& \frac{1}{A_{u,w}A_{v,w}} \int_{\frac{\omega(N)}{2\pi}}^{A_{v,w} /2} \int_{\frac{\omega(N)}{2\pi}}^{A_{u,w} /2}  
\frac{1}{z_1^\beta+1} \frac{1}{z_2^\beta+1} \frac{1}{(C_{w,v}z_1+C_{w,u}z_2)^\beta +1}dz_1 dz_2.
\end{split}
\]
The condition $z_2\geq \omega(N)/2\pi $ implies that
\[
\frac{1}{z_2^\beta+1}\lesssim \left ( \frac{1}{\omega(N)}\right )^\beta .
\]
Setting \ $ y=C_{w,v}z_1+C_{w,u}z_2$, we have
%\[
%y=C_{wv}z_1+C_{wu}z_2
%\]
%we have
\[
dz_1 dz_2= \frac{dz_1dy}{C_{w,u}}
\]
with
\[
y\in \left ( C_{w,v}z_1+C_{w,u}\frac{\omega(N)}{2\pi} , \, C_{w,v}z_1+C_{w,u}A_{v,w} /2\right ).
\]
So we obtain:
\[
\begin{split}
\E & ( T (u,v;w)\vartheta(u,v;w)\mid t_u,t_v,t_w)\\
& \lesssim \left( \frac{1}{\omega(N)}\right)^\beta \frac{1}{A_{u,w}A_{v,w}C_{w,u}}\int_{\frac{\omega(N)}{2\pi}}^{A_{u,w} /2}
\frac{1}{z_1^\beta}
\int_{C_{w,v}z_1+C_{w,u}\frac{\omega(N)}{2\pi}}^{C_{w,v}z_1+C_{w,u}A_{v,w}/2} \frac{1}{y^\beta} dy dz_1 .
\end{split}
\]
We extend the inner { integration interval} using an upper bound on $z_1$. In particular, since $z_1 \leq A_{u,w} /2$, we have
$C_{w,v}z_1 + C_{w,u}A_{v,w}  /2\leq  (C_{w,v}A_{u,w} + C_{w,u}A_{v,w})/2$. Note that
\begin{equation} \label{eq:CAProd}
\begin{split}
 C_{w,v} A_{u,w} &= \exp\left({\zeta \over 2} \left(t_w - t_v + R - t_u - t_w \right)\right) = A_{u,v}, \\
 C_{w,u} A_{v,w} &= A_{u,v}.
\end{split}
\end{equation}
Hence,
$$ C_{w,v}z_1 + C_{w,u}A_{v,w}/2  \leq A_{u,v}.$$ Also, since $z_1 \geq {\omega (N) \over 2\pi}$, we have
$C_{w,v}z_1 + C_{w,u}{\omega (N) \over 2\pi} \geq (C_{w,v}+C_{w,u})\frac{\omega(N)}{2\pi}$.
%Since the function $\frac{1}{y^\beta} $ is strictly positive, for $y>0$, we will obtain an upper bound assuming that
{ Then we will obtain an upper bound by extending the integration interval to}
\[
y\in \left ( (C_{w,v}+C_{w,u})\frac{\omega(N)}{2\pi} , \, A_{u,v}\right ).
\]
Hence, we obtain
\begin{equation*}
\begin{split}
\E & ( T (u,v;w)\vartheta(u,v;w) \mid t_u,t_v,t_w )\\
& \lesssim \left ( \frac{1}{\omega(N)}\right )^\beta \frac{1}{A_{u,w}A_{v,w}C_{w,u}}\int_{\frac{\omega(N)}{2\pi}}^{A_{u,w}/2}
\frac{1}{z_1^\beta} dz_1 \int_{(C_{w,v}+C_{w,u})\frac{\omega(N)}{2\pi}}^{A_{u,v}} \frac{1}{y^\beta} dy  \\
& \stackrel{(\ref{eq:CAProd})}{\lesssim} \left ( \frac{1}{\omega(N)}\right )^\beta \frac{1}{A_{u,w}A_{u,v}}
 A_{u,w}^{1-\beta} A_{u,v}^{1-\beta} \asymp \left ( \frac{1}{\omega(N)}\right )^\beta \frac{1}{(A_{u,v}A_{u,w})^\beta}.
\end{split}
\end{equation*}
%In the last inequality we used the fact that we are in the typical case, hence the dominant term is $(A_{uv}A_{uw})^{1-\beta}$.
Now we integrate over $(t_u, t_v, t_w) \in D_t=[0,R/2-\omega(N)]^3$ using Claim~\ref{clm:density_approx}, obtaining
\[
\begin{split}
\E & ( T (u,v;w)\vartheta(u,v;w) ) \\
& \lesssim \int_0^{R/2-\omega(N)} \int_0^{R/2-\omega(N)} \int_0^{R/2-\omega(N)} \left ( \frac{1}{\omega(N)}\right )^\beta \frac{e^{-\alpha (t_u+t_v+t_w)}}{(A_{u,v}A_{u,w})^\beta} dt_u dt_v dt_w.
%& = \Theta \left ( \left ( \frac{1}{\omega(N)}\right )^\beta \P (v \sim w \sim u)\right ).
\end{split}
\]
Elementary integration now yields:
\[
\E ( T (u,v;w)\vartheta(u,v;w)) \lesssim
\left \{ \begin{array}{ll}
(\omega(N))^{-\beta} N^{-2\beta} & \textnormal{if }\beta\zeta/\alpha<1\\
(\omega(N))^{-\beta} R N^{-2\beta} & \textnormal{if }\beta\zeta/\alpha=1\\
(\omega(N))^{-\beta} N^{-\beta-\alpha/\zeta}e^{-(\beta\zeta-\alpha)\omega(N)} & \textnormal{if }\beta\zeta/\alpha>1
\end{array} \right ..
\]
Multiplying the above by ${N \choose 3}$ and comparing the outcome with (\ref{expected_D<1})
we now deduce~Proposition~\ref{prop:complete_triangles}(ii) for $\beta < 1$ from (\ref{eq:equivalence}).

\section{Proof of Proposition~\ref{prop_asymptotic_atypical_triangles}}\label{sect:proof_prop_26}

In this section we bound from above the expected number of atypical triangles, that is, those triangles which contain at least one vertex of
type greater than $R/2 - \omega (N)$. We will only do this for the case $\zeta /\alpha > 1$, as for the case $\zeta /\alpha < 1$ the
 argument is straightforward. 

\subsection*{Case $\zeta/\alpha<1 $}

%As an easy consequence of Corollary \ref{cor:x0} we get that, a.a.s.\ we have
%\[
%\E(\widetilde T)\leq \E(\widetilde \Lambda)\leq 3\sum_{u\in \D}\P(t_u>x_0)=o(1).
%\]
Note that if $ \widetilde \Lambda\geq 1$, then there is a vertex of type at least $R/2-\omega(N)$.
But by Corollary \ref{cor:x0}, a.a.s.\ all vertices have type at most
\[
\frac{\zeta}{2\alpha}R+\omega(N)<\frac{R}{2}-\omega(N),
\]
for $N$ sufficiently large.
Hence, in this case, $3\widetilde T\leq \widetilde \Lambda=0 $, a.a.s.

\subsection*{Case $\zeta /\alpha \geq 1$}

As a preliminary observation, we point out that the probability that three vertices form an atypical triangle is bounded from above by the probability that they form an incomplete triangle.

For every triple $u,v,w$ of distinct vertices, we let ${\bf{1}}_{\Delta_a(u,v,w)}$ denote the indicator random variable that is equal to 1 if and only if
the vertices $u,v$ and $w$ form a triangle and at least one of these vertices is atypical. Similarly, we let
${\bf{1}}_{\Lambda_a(u,v;w)}$ be the indicator random variable that is equal to 1 if and only if  the vertices $u,v$ and $w$ form an incomplete triangle
with $w$ as the pivoting vertex and at least one of these vertices is atypical.
%Throughout this section we are going to make use of the following trivial inequality:

%The aim of this section is to give a bound on the expected number of atypical triangles. % and on the number of atypical incomplete triangles.
We split $\E({\bf 1}_{\Delta_a (u,v,w)})$ as follows:
\begin{equation}\label{eq:bound_triangles}
\begin{split}
\E ({\bf{1}}_{\Delta_a(u,v,w)}) & = 3\E ({\bf{1}}_{\Delta_a(u,v,w)}{\bf{1}}_{t_u,t_v\leq R/2-\omega(N);~ t_w>R/2-\omega(N)}) \\
& + \E ({\bf{1}}_{\Delta_a(u,v,w)}{\bf{1}}_{t_u,t_v,t_w> R/2-\omega(N)}) \\
& + 3\E ({\bf{1}}_{\Delta_a(u,v,w)}{\bf{1}}_{t_u \leq R/2-\omega(N);~ t_v,t_w>R/2-\omega(N)}).
\end{split}
\end{equation}
To deduce Proposition~\ref{prop_asymptotic_atypical_triangles}, it suffices to show that
\begin{equation} \label{eq:atyp_triangles_toprove}
\begin{split}
\left.
\begin{array} {l}
N^3 \E ({\bf{1}}_{\Delta_a(u,v,w)}{\bf{1}}_{t_u,t_v\leq R/2-\omega(N);~ t_w>R/2-\omega(N)}) \\
N^3  \E ({\bf{1}}_{\Delta_a(u,v,w)}{\bf{1}}_{t_u,t_v,t_w> R/2-\omega(N)}) \\
N^3 \E ({\bf{1}}_{\Delta_a(u,v,w)}{\bf{1}}_{t_u, t_v \leq R/2-\omega(N);~ t_w>R/2-\omega(N)})
\end{array} \right\}
= o\left(\E(\hat{\Lambda}) \right).
\end{split}
\end{equation}
Hence, the second part of Proposition~\ref{prop_asymptotic_atypical_triangles} will follow from Markov's inequality. 

We will estimate each one of these terms separately. 
We will be using the following general inequality:
\begin{equation}\label{eq_delta_a_lambda_a}
\E ({\bf{1}}_{\Delta_a(u,v,w)}{\bf{1}}_E)\leq \E ({\bf{1}}_{\Lambda_a(u,v;w)}{\bf{1}}_E),
\end{equation}
where $E$ denotes any event.
The event $E$ will be specified according to the case we consider.

Now recall the definitions of $\beta'$ and $\delta$:
\begin{equation*}%\label{def_delta}
\beta':=
\left \{ \begin{array}{ll}
\beta, & \textnormal{if } \beta < 1\\
1, & \textnormal{if } \beta \geq 1
\end{array} \right .
\qquad
\textnormal{and}
\qquad
\delta :=
\left \{
\begin{array}{ll}
0, & \textnormal{if } \beta \neq 1\\
1, & \textnormal{if } \beta = 1
\end{array}
\right ..
\end{equation*}
For sake of clarity, we will defer many of the technical calculations to Appendix \ref{sect_aux_calculations}.

\subsection{Case $t_u,t_v,t_w>R/2-\omega(N)$}\label{sect_3_bad}
In this case, we bound $\E ({\bf{1}}_{\Delta_a (u,v;w)})$ by the probability that $u,v,w$ have type greater than $R/2 - \omega (N)$:
\[
\begin{split}
\E  &({\bf{1}}_{\Delta_a(u,v;w)} {\bf 1}_{t_u,t_v,t_w>R/2-\omega(N)}) \lesssim 
% \E  ({\bf{1}}_{\Lambda_a(u,v;w)} {\bf 1}_{t_u,t_v,t_w>R/2-\omega(N)})
\\
& \int_{R/2-\omega(N)}^R \int_{R/2-\omega(N)}^R \int_{R/2-\omega(N)}^R e^{-\alpha(t_u+t_v+t_w)}dt_u dt_v dt_w.
\end{split}
\]
Now it is easy to check that
\[
t_u+t_v+t_w>\frac{3}{2}R - 3\omega(N).
\]
From these observations we can deduce the following:
\[
\begin{split}
\E & ({\bf{1}}_{\Delta_a(u,v,w)} {\bf 1}_{t_u,t_v,t_w>R/2-\omega(N)})
 \lesssim %\E  ({\bf{1}}_{\Lambda_a(u,v;w)} {\bf 1}_{t_u,t_v,t_w>R/2-\omega(N)})
\\
& \int_{R/2-\omega(N)}^R \int_{R/2-\omega(N)}^R \int_{R/2-\omega(N)}^R
e^{-\alpha(\frac{3}{2} R - 3\omega(N))}dt_u dt_v dt_w \\[1.1ex]
& \asymp R^{3}  e^{-{3\alpha \over \zeta} {\zeta R\over 2}} e^{3\alpha  \omega(N)} \asymp
R^{3}  N^{-{3\alpha \over \zeta}} e^{3\alpha  \omega(N)}.
\end{split}
\]
%which implies the statement for large $N$ and $\omega(N)$ growing slowly enough.

It suffices to show that if $\omega (N)$ is sufficiently slowly growing, then we have
\[
R^{3}  N^{3-{3\alpha \over \zeta}} e^{3\alpha  \omega(N)}\ll \E (\widehat{\Lambda}).
\]
Indeed, the above follows from the following conditions which are easy to verify (cf. Proposition~\ref{prop:incomplete_triangles}).
\begin{itemize}
\item[(i)] for $\beta\geq 1$ (i.e., $\beta'=1$) we have
\[
\left \{
\begin{array}{ll}
3-3\alpha/\zeta < 1, & \textnormal{ if } \zeta/\alpha=1\\
3-3\alpha/\zeta <2-\alpha/\zeta , & \textnormal{ if } \zeta/\alpha >1
\end{array}
\right .,
\]
as the latter is equivalent to $2\alpha/\zeta > 1$.
\item[(ii)] for $\beta <1$ (which, by the definition of the model, implies $\beta\zeta/\alpha<2$) we have
\[
\left \{
\begin{array}{ll}
3-3\alpha/\zeta < 3-2\beta, & \textnormal{ if } \beta\zeta/\alpha\leq 1\\
3-3\alpha/\zeta < 3-\beta-\alpha/\zeta , & \textnormal{ if } \beta\zeta/\alpha>1
\end{array}
\right .,
\]
as the former is equivalent to $\beta \zeta /\alpha < 3/2$ and the latter is equivalent to $\alpha /\beta \zeta > 1/2$.
\end{itemize}

\subsection{Case $t_u\leq R/2-\omega(N) $ and $t_v,t_w>R/2-\omega(N)$}\label{sect_2_bad}
Here we will make use of (\ref{eq_delta_a_lambda_a}) and bound the probability that three vertices form a triangle by the probability that they form an incomplete triangle. 
{ In particular, we shall be considering the case when the incomplete triangle is pivoted at $w$.}
%\subsubsection{Pivot at $w$}
Under this assumption we consider the following two sub-cases:
\begin{itemize}
\item[1.] $ t_u+t_w\leq R-\omega(N)$ with $ t_v+t_w> R-2\omega(N)$;
\item[2.] $ t_u+t_w> R-\omega(N)$ with $ t_v+t_w> R-2\omega(N)$.
\end{itemize}
Assume without loss of generality that $t_v>t_w $.  
The domain of integration of sub-case 1 becomes
\[
D_1:=\{ 0<t_u\leq R-t_w-\omega(N), \ t_w<t_v<R, \ R/2-\omega(N)<t_w<R\}.
\]
Note first that for any $u,v \in \V$ we have $\ln (A_{u,v})\lesssim \ln (N)\asymp R$.
Thus, applying Lemma~\ref{lemma_2.4_evolution} we obtain
\[
\begin{split}
& \E ({\bf{1}}_{\Lambda_a(u,v;w)} {\bf 1}_{t_u\leq R/2-\omega(N);\, t_v,t_w>R/2-\omega(N)}{\bf 1}_{(t_u,t_v,t_w) \in D_1}) \\
& \lesssim\int_{R/2-\omega(N)}^R \int_{t_w}^R \int_0^{R-t_w-\omega(N)}\!\left ( \frac{ R^{\delta}}{A_{u,w}}\right )^{\beta'\!\!}e^{-\alpha(t_u+t_v+t_w)}dt_u dt_v dt_w
 =: \phi_1 .
\end{split}
\]
Multiplying the above by $N^3$, the estimates in (\ref{phi_4}) imply (\ref{eq:atyp_triangles_toprove}).
In fact, it is easy to show that the following conditions hold.
\begin{itemize}
\item[(i)] For $\beta\geq 1$ we have
\[
\left \{
\begin{array}{ll}
5/2-2\alpha/\zeta<1, & \textnormal{ if }\zeta/\alpha=1\\
5/2-2\alpha/\zeta<2-\alpha/\zeta , & \textnormal{ if }\zeta/\alpha > 1
\end{array}
\right .,
\]
where the latter is equivalent to $\alpha /\zeta > 1/2$ (that is, $\zeta/\alpha < 2$). 
\item[(ii)] For $\beta<1$ we have
\[
\left \{
\begin{array}{ll}
3-\beta/2-2\alpha/\zeta< 3-2\beta, & \textnormal{ if }\beta\zeta/\alpha\leq 1\\
3-\beta/2-2\alpha/\zeta< 3-\beta-\alpha/\zeta , & \textnormal{ if }\beta\zeta/\alpha>1
\end{array}
\right .,
\]
where the former is equivalent to $\beta \zeta /\alpha < 4/3$ and the latter is equivalent to $\alpha / \beta \zeta > 1/2$ (that is, 
$\beta \zeta /\alpha < 2$). 
\end{itemize}

In sub-case 2, the domain of integration is
\[
D_2:=\{ R-t_w-\omega(N)\!<\!t_u\!<\! R/2-\omega(N),  t_w\!<\!t_v\!<\!R,  R/2-\omega(N)\!<\!t_w\!<\!R\}.
\]
Hence we get
\[
\begin{split}
\E & ({\bf{1}}_{\Lambda_a(u,v;w)} {\bf 1}_{t_u\leq R/2-\omega(N);\, t_v,t_w>R/2-\omega(N)}{\bf 1}_{(t_u,t_v,t_w) \in D_2}) \\
& \lesssim \int_{R/2-\omega(N)}^R \int_{t_w}^R \int_{R-t_w-\omega(N)}^{R/2} e^{-\alpha(t_u+t_v+t_w)}dt_u dt_v dt_w
=: \phi_2 .
\end{split}
\]
The above function is estimated in (\ref{phi_5}). Multiplying that by $N^3$ we obtain (\ref{eq:atyp_triangles_toprove}).
Indeed, the following inequalities are easy to verify.
\begin{itemize}
\item[(i)] for $\beta\geq 1$ we have
\[
\left \{
\begin{array}{ll}
3-3\alpha/\zeta<1, & \textnormal{ if } \zeta/\alpha=1\\
3-3\alpha/\zeta<2-\alpha/\zeta, & \textnormal{ if } \zeta/\alpha>1
\end{array}
\right . ,
\]
where the last inequality holds since $\alpha / \zeta > 1/2$ (that is, $\zeta /\alpha < 2$);
\item[(ii)] for $\beta< 1$ we have
\[
\left \{
\begin{array}{ll}
3-3\alpha/\zeta<3-2\beta, & \textnormal{ if } \beta\zeta/\alpha\leq 1\\
3-3\alpha/\zeta<3-\beta-\alpha/\zeta, & \textnormal{ if } \beta\zeta/\alpha>1
\end{array}
\right . ,
\]
where the former holds because $\beta \zeta /\alpha \leq 1 < 3/2$ and the latter because 
$\alpha / \zeta > 1/2$ and $\beta < 1$. 
\end{itemize}

\subsection{Case $t_u,t_v\leq R/2-\omega(N) $ and $t_w>R/2-\omega(N)$}\label{sect_1_bad}

Under our current assumptions, we have four possible sub-cases:
\begin{enumerate}
\item $ t_u+t_w\leq R-2\omega(N)$ with $ t_v+t_w\leq R-2\omega(N)$;
\item $ t_u+t_w\leq R-2\omega(N)$ with $ t_v+t_w> R-2\omega(N)$;
\item $ t_u+t_w> R-2\omega(N)$ with $ t_v+t_w\leq R-2\omega(N)$;
\item $ t_u+t_w> R-2\omega(N)$ with $ t_v+t_w> R-2\omega(N)$.
\end{enumerate}
We denote the $i$th domain by $D_i$.
We need to treat each situation separately, starting with sub-case 1. 
We shall use Lemma~\ref{lemma_2.4_evolution}:
\[
\begin{split}
\E & ({\bf{1}}_{\Delta_a(u,v,w)} {\bf 1}_{t_u,t_v\leq R/2-\omega(N);\, t_w>R/2-\omega(N)}{\bf 1}_{(t_u,t_v,t_w) \in D_1}) \leq \\[1.3ex]
\E & ({\bf{1}}_{\Lambda_a (v,w;u)} {\bf 1}_{t_u,t_v\leq R/2-\omega(N);\, t_w>R/2-\omega(N)}{\bf 1}_{(t_u,t_v,t_w) \in D_1})
\\[1.3ex]
& \lesssim \int_{R/2-\omega(N)}^{R-2\omega (N)} \int_0^{R-t_w-2\omega(N)}
\int_0^{R-t_w-2\omega(N)} \left ( \frac{R^{2\delta}}{A_{u,v} A_{u,w}}\right )^{\beta'}\times \\[1.3ex]
& \quad  e^{-\alpha(t_u+t_v+t_w)} dt_u dt_v dt_w%\\ &
=: \phi_3 .
\end{split}
\]
The asymptotic growth of $\phi_3$  is determined by the ratio $\beta'\zeta/\alpha $, as in (\ref{phi6}) in Appendix \ref{sect_aux_calculations}.
To deduce~(\ref{eq:atyp_triangles_toprove}) on this sub-domain, we multiply (\ref{phi6}) by $N^3$ and compare the resulting exponents of $N$ with those in Proposition~\ref{prop:incomplete_triangles}. For each case we have:
\begin{itemize}
\item[(i)] for $\beta\geq 1$ (that is, $\beta'=1$) we have
\[
\left \{
\begin{array}{ll}
3 - {3\over 2}\beta' - \alpha /\zeta < 1, & \textnormal{ if } \zeta/\alpha=1\\
3-\beta'/2-2\alpha/\zeta < 2-\alpha/\zeta , & \textnormal{ if } \zeta/\alpha>1
\end{array}
\right .,
\]
where the latter holds since $\alpha / \zeta > 1/2$ (that is, $\zeta /\alpha < 2$). 
\item[(ii)] for $\beta <1$ (where $\beta' =\beta$) we have
\[
\left \{
\begin{array}{ll}
3-{3\over 2}\beta- \alpha/\zeta < 3-2\beta, & \textnormal{ if } \beta\zeta/\alpha < 1\\
3-{5\over 2}\beta  < 3-2\beta, & \textnormal{ if } \beta\zeta/\alpha = 1\\
3-{\beta \over 2} - 2\alpha/\zeta < 3-\beta-\alpha/\zeta , & \textnormal{ if } \beta\zeta/\alpha>1
\end{array}
\right .,
\]
since $\alpha / \zeta > 1/2$ and $\beta < 1$.
\end{itemize}
%\smallbreak
%
%\noindent
Regarding Case 2 (as well as Case 3) we have the following:
\[
\begin{split}
\E & ({\bf{1}}_{\Delta_a(u,v,w)} {\bf 1}_{t_u,t_v\leq R/2-\omega(N);\, t_w>R/2-\omega(N)}{\bf 1}_{(t_u,t_v,t_w) \in D_2}) 
\leq \\[1.3ex]
 \E & ({\bf{1}}_{\Lambda_a (v,w;u)} {\bf 1}_{t_u,t_v\leq R/2-\omega(N);\, t_w>R/2-\omega(N)}{\bf 1}_{(t_u,t_v,t_w) \in D_2}) 
\lesssim \\[1.3ex]
& \int_{R/2-\omega(N)}^R
\int_{0}^{R-t_w-2\omega(N)} \int_{R - t_w -2\omega(N)}^{R/2-\omega(N)} \left ( \frac{R^{2\delta}}{A_{u,v}A_{u,w}}\right )^{\beta'}\times \\[1.3ex]
& \quad \times e^{-\alpha(t_u+t_v+t_w)} dt_v dt_u dt_w  =: \phi_4 .
\end{split}
\]
As in the previous case, this expression depends on the ratio $\beta'\zeta/\alpha $. The
statement follows multiplying (\ref{eq:phi7}) by $N^3$ and comparing the exponents of
$N$ with those in Proposition~\ref{prop:incomplete_triangles}. For each case we have:
\begin{itemize}
\item[(i)] for $\beta\geq 1$ we get
\[
\left \{
\begin{array}{ll}
2-2\alpha/\zeta <1, & \textnormal{ if }\zeta/\alpha=1\\
3-3\alpha/\zeta<2-\alpha/\zeta, & \textnormal{ if }1< \zeta/\alpha < 2
\end{array}
\right ..
\]
\item[(ii)] for $\beta<1$ we get
\[
\left \{
\begin{array}{ll}
3-\beta-2\alpha/\zeta<3-2\beta, & \textnormal{ if }\beta\zeta/\alpha\leq 1\\
3-3\alpha/\zeta<3-\beta-\alpha/\zeta, & \textnormal{ if }\beta\zeta/\alpha>1
\end{array}
\right ..
\]
\end{itemize}

Case 4 is treated in a similar way:
\[
\begin{split}
{ \E} & ({\bf{1}}_{\Delta_a(u,v,w)} {\bf 1}_{t_u,t_v\leq R/2-\omega(N);\, t_w>R/2-\omega(N)}{\bf 1}_{(t_u,t_v,t_w) \in D_4})  \\[1.3ex]
& \lesssim \int_{R/2-\omega(N)}^{R}\int_{R-t_w - 2\omega (N)}^{R/2-\omega(N)}\int_{R-t_w - 2\omega (N)}^{R/2-\omega(N)} \left ( \frac{ R^{\delta}}{A_{u,v}}\right )^{\beta'}\times \\[1.3ex]
& \quad \times e^{-\alpha(t_u+t_v+t_w)} dt_u dt_v dt_w=: \phi_5 .
\end{split}
\]
We estimate this in (\ref{var_phi3}) in Appendix A.
As above, the statement follows from the multiplication of (\ref{var_phi3}) by $N^3$.
For each case we have:
\begin{itemize}
\item[(i)] for $\beta\geq 1$ we get
\[
\left \{
\begin{array}{ll}
2-2\alpha/\zeta <1, & \textnormal{ if }\zeta/\alpha=1\\
3-3\alpha/\zeta<2-\alpha/\zeta, & \textnormal{ if }1< \zeta/\alpha < 2
\end{array}
\right ..
\]
\item[(ii)] for $\beta<1$ we get
\[
\left \{
\begin{array}{ll}
3-\beta-2\alpha/\zeta<3-2\beta, & \textnormal{ if }\beta\zeta/\alpha\leq 1\\
3-3\alpha/\zeta<3-\beta-\alpha/\zeta, & \textnormal{ if }\beta\zeta/\alpha>1
\end{array}
\right ..
\]
\end{itemize}

\section{Conclusions}
In this paper we give a precise characterization of the presence of clustering in random geometric graphs on the hyperbolic plane in terms
of its parameters. We focus on the range of parameters where these random graphs have a linear number of edges and their degree
distribution follows a power law. We quantify the existence of clustering, furthermore, in the part of the random graph that consists of 
vertices that have  { type at most $t$, where $0<t<R$},
we show that the clustering coefficient there is bounded away from 0. More importantly, we determine exactly how this quantity
depends on the parameters of the random graph.

The present work is a step towards establishing such random graphs as a suitable model for complex networks. Together 
with~\cite{ar:Foun13+} and~\cite{ar:Kosta}, our results show that for certain values of the parameters, such random graphs do capture
two of the fundamental properties of complex networks, namely: power-law degree distribution as well as clustering.

A natural next step in this direction is the study of the typical distances (in terms of \emph{hops}) between vertices. More precisely, one is
interested in investigating the distance between two typical vertices, and how the values of the parameters influence this quantity.
In other words, for which values of $\beta$ and $\zeta/\alpha$ is the resulting random graph what is commonly called a \emph{small world}?

\appendix

\section{Auxiliary Calculations}\label{sect_aux_calculations}
In this section we show the technical calculations needed to finish the proofs in Section \ref{sect:proof_prop_26}.
Recall that
\[
\beta'=
\left \{
\begin{array}{ll}
\beta & \textnormal{if } \beta < 1\\
1 & \textnormal{if } \beta \geq 1
\end{array}
\right . 
\qquad
\textnormal{and}
\qquad
\delta =
\left \{
\begin{array}{ll}
0 & \textnormal{if } \beta \neq 1\\
1 & \textnormal{if } \beta = 1
\end{array}
\right ..
\]
In the proof of Proposition~\ref{prop_asymptotic_atypical_triangles}, we defined the functions $\phi_1, \dots, \phi_5$, 
which we calculate explicitly in this section.

We start with
\[
\phi_1 \leq R^\delta \int_{R/2-\omega(N)}^R \int_{t_w}^R \int_0^{R-t_w-\omega(N)}\left ( \frac{1}{A_{u,w}}\right )^{\beta'}e^{-\alpha(t_u+t_v+t_w)}dt_u dt_v dt_w.
\]
Then we have
\[
\begin{split}
\phi_1 
& \lesssim \frac{R^\delta}{N^{\beta'}} 
\int_{R/2-\omega(N)}^R \int_{t_w}^R \int_0^{R-t_w-\omega(N)} 
e^{\left(\beta' \zeta / 2 - \alpha \right) t_u + (\beta' \zeta /2)t_w  - \alpha (t_v + t_w) } dt_u dt_v dt_w
\\
& \lesssim \frac{R^\delta}{N^{\beta'}} \int_{R/2-\omega(N)}^R \int_{t_w}^R e^{(\beta'\zeta/2)t_w}e^{-\alpha(t_v+t_w)} dt_v dt_w \\
& \lesssim \frac{R^\delta}{N^{\beta'}} \int_{R/2-\omega(N)}^R e^{(\beta'\zeta/2-2\alpha)t_w} dt_w
\lesssim \frac{R^\delta}{N^{\beta'}} e^{(\beta'\zeta/2-2\alpha)(R/2-\omega(N))}.
\end{split}
\]
Hence we have
\begin{equation}\label{phi_4}
\phi_1 \lesssim R^\delta N^{-\beta'/2-2\alpha/\zeta} e^{-(\beta'\zeta/2-2\alpha)\omega(N)} .
\end{equation}
We now consider
\[
\phi_2 \leq \int_{R/2-\omega(N)}^R \int_{t_w}^R \int_{R-t_w-\omega(N)}^{R/2} e^{-\alpha(t_u+t_v+t_w)}dt_u dt_v dt_w.
\]
A calculation similar to the previous case yields
\[
\begin{split}
\phi_2
& \lesssim \int_{R/2-\omega(N)}^R \int_{t_w}^R e^{-\alpha(R-t_w-\omega(N))} e^{-\alpha(t_v+t_w)} dt_v dt_w \\
& = e^{-\alpha(R-\omega(N))} \int_{R/2-\omega(N)}^R \int_{t_w}^R e^{-\alpha t_v} dt_v dt_w \\
& \lesssim N^{-2\alpha/\zeta} e^{\alpha \omega(N)} \int_{R/2-\omega(N)}^R e^{-\alpha t_w} dt_w
\lesssim N^{-2\alpha/\zeta} e^{\alpha \omega(N)} e^{-\alpha (R/2-\omega(N))}.
\end{split}
\]
And finally we obtain
\begin{equation}\label{phi_5}
\phi_2
\lesssim N^{-3\alpha/\zeta} e^{2\alpha \omega(N)}.
\end{equation}

\noindent
Now we consider
\[
\begin{split}
& \phi_3 =
\int_{R/2-\omega(N)}^{R-2\omega (N)} \int_0^{R-t_w-2\omega(N)}
\int_0^{R-t_w-2\omega(N)} \left ( \frac{R^{2\delta}}{A_{u,v}A_{u,w}}\right )^{\beta'}\times \\[1.2ex]
& \qquad e^{-\alpha(t_u+t_v+t_w)} dt_u dt_v dt_w \\[1.2ex]
&\leq {R^{2\delta}\over N^{2\beta'}}
\int_{R/2-\omega(N)}^{R-2\omega (N)} \! \int_0^{R-t_w-2\omega(N)}
\int_0^{R-t_w-2\omega(N)}
e^{\left({\beta'\zeta \over 2} - \alpha \right)(t_w + t_v)+ \left( \beta'\zeta -\alpha \right) t_u} \\[1.2ex]
& \qquad \times
dt_u dt_v dt_w \\[1.2ex]
&\lesssim
{R^{2\delta}\over N^{2\beta'}}
\int_{R/2-\omega(N)}^{R - 2\omega (N)}
e^{\left({\beta'\zeta \over 2} -\alpha \right)t_w}
\left[ \int_0^{R-t_w-2\omega(N)}
e^{(\beta'\zeta -\alpha)t_u} dt_u \right]
dt_w .
\end{split}
\]
Now the order of magnitude of this integral depends on the ratio $\beta'\zeta/\alpha$.
\smallbreak

\noindent
For $\beta'\zeta/\alpha\leq 1$, we have
\[
\begin{split}
\phi_3 
& \lesssim
{R^{2\delta + 1}\over N^{2\beta'}}
\int_{R/2-\omega(N)}^{R - 2\omega (N)}
e^{\left({\beta'\zeta \over 2} -\alpha \right)t_w} dt_w  \lesssim
{R^{2\delta + 1}\over N^{2\beta'}} e^{\left({\beta'\zeta \over 2} - \alpha \right)R/2}
e^{-(\beta'\zeta/2 - \alpha)\omega (N)} \\
&\asymp {R^{2\delta+1}N^{\beta'/2 - {\alpha /\zeta}}\over N^{2\beta'}}
e^{-(\beta'\zeta/2 - \alpha)\omega (N)}  =
R^{2\delta + 1}N^{-{3\over 2}\beta' - {\alpha /\zeta}} e^{-(\beta'\zeta/2 - \alpha)\omega (N)}.
\end{split}
\]
\smallbreak

\noindent
Finally, when $\beta'\zeta/\alpha>1$, we have
\[
\begin{split}
& \phi_3  \lesssim
{R^{2\delta} e^{(\beta'\zeta - \alpha)R}\over N^{2\beta'}}
e^{-2(\beta'\zeta - \alpha)\omega (N) }\int_{R/2-\omega(N)}^{R - 2\omega(N)}\!\!
e^{\left({\beta'\zeta \over 2} -\alpha\right)t_w- (\beta'\zeta -\alpha)t_w} dt_w \\
& \asymp
{R^{2\delta} N^{2(\beta' - \alpha /\zeta)}\over N^{2\beta'}}
e^{-2(\beta'\zeta - \alpha)\omega (N) }
\int_{R/2-\omega(N)}^{R-2\omega(N)}
e^{- \beta'\zeta t_w/2} dt_w \\
&\asymp
 {R^{2\delta} N^{2(\beta' - \alpha /\zeta)} e^{-{\beta'\zeta \over 2}~R/2}\over N^{2\beta'}}
 e^{-(3\beta'\zeta/2 - 2\alpha)\omega (N)} \\
 & \asymp
 {R^{2\delta} N^{2(\beta' - \alpha /\zeta) -{\beta' \over 2}}\over N^{2\beta'}}
 e^{-(3\beta'\zeta/2 - 2\alpha)\omega (N)} = R^{2\delta} N^{-{\beta' \over 2} - 2\alpha /\zeta} 
e^{-(3\beta'\zeta/2 - 2\alpha)\omega (N)}.
\end{split}
\]
Therefore,
\begin{equation} \label{phi6}
\phi_3 \lesssim
\begin{cases}
R^{2\delta+1}N^{-{3\over 2}\beta' - {\alpha /\zeta}} e^{-(\beta'\zeta/2 - \alpha)\omega (N)}, &
\mbox{if $\beta'\zeta / \alpha \leq 1$} \\
%\frac{R^{3\delta +3}}{N^{3\beta'}}, & \mbox{if $\beta' \zeta / \alpha = 1$} \\
R^{2\delta} N^{-{\beta' \over 2} - 2\alpha /\zeta} e^{-(3\beta'\zeta/2 - 2\alpha)\omega (N)}, &
\mbox{if $\beta' \zeta /\alpha > 1$} \\
\end{cases}.
\end{equation}
Now we consider the function
\[
\begin{split}
& \phi_4 \leq
R^{2\delta}\int_{R/2-\omega(N)}^R
\int_{0}^{R-t_w-2\omega(N)} \int_{R - t_w -2\omega(N)}^{R/2-\omega(N)} \left ( \frac{1}{A_{u,v}A_{u,w}}\right )^{\beta'}\times \\[1.2ex]
& \qquad  e^{-\alpha(t_u+t_v+t_w)} dt_v dt_u dt_w \\[1.2ex]
&\asymp
{R^{2\delta} \over N^{2\beta'}}
\int_{R/2-\omega(N)}^R
\int_{0}^{R-t_w-2\omega(N)} \int_{R - t_w -2\omega(N)}^{R/2-\omega(N)}
e^{(\beta'\zeta -\alpha)t_u} \times \\[1.3ex]
& \qquad e^{\left({\beta' \zeta \over 2}-\alpha \right)(t_v +t_w)} dt_v dt_u dt_w.
\end{split}
\]
Therefore, integrating with respect to $t_v$ recalling that $\beta' \zeta/2 < \alpha$ we obtain:
\[
\begin{split}
& \phi_4 \lesssim %\frac{R^\delta }{N^{2\beta'}} \int_0^{R/2-\omega(N)} e^{(\beta'\zeta/2-\alpha)t_v}e^{(\beta'\zeta-\alpha)t_v} dt_v .
{R^{2\delta} e^{\left({\beta' \zeta \over 2}-\alpha \right)R}\over N^{2\beta'}}
e^{-\left(\beta' \zeta-2\alpha \right)\omega (N)}\times \\[1.2ex]
& \qquad  \int_{R/2-\omega(N)}^R
\int_{0}^{R-t_w-2\omega(N)}
e^{(\beta'\zeta -\alpha)t_u-
\left({\beta' \zeta \over 2}-\alpha \right)t_w +
\left({\beta' \zeta \over 2}-\alpha \right)t_w} dt_u dt_w \\
&\asymp
{R^{2\delta}  N^{\beta' - 2\alpha /\zeta}\over N^{2\beta'}}
e^{-\left(\beta' \zeta-2\alpha \right)\omega (N)}
\int_{R/2-\omega(N)}^R
\int_{0}^{R-t_w-2\omega(N)}
e^{(\beta'\zeta -\alpha)t_u} dt_u dt_w.
\end{split}
\]
Now, the behavior of the latter integral depends on the value of $\beta' \zeta /\alpha$.
If $\beta' \zeta /\alpha \leq 1$, then
\begin{equation*}
\phi_4 \lesssim R^{2\delta + 2} N^{-\beta' - 2\alpha / \zeta}
e^{-\left(\beta' \zeta-2\alpha \right)\omega (N)}.
\end{equation*}
However, for $\beta' \zeta /\alpha > 1$ we have
\begin{equation*}
\begin{split}
& \phi_4 \lesssim {R^{2\delta}  N^{-\beta' - 2\alpha /\zeta}}
e^{-\left(\beta' \zeta-2\alpha \right)\omega (N)}
\int_{R/2-\omega(N)}^R
e^{(\beta'\zeta -\alpha)(R-t_w -2\omega (N))}  dt_w \\
&=  {R^{2\delta}  N^{-\beta' - 2\alpha /\zeta} e^{(\beta'\zeta -\alpha) R}}
e^{-\left(3\beta' \zeta-4\alpha \right)\omega (N)}
\int_{R/2-\omega(N)}^R
e^{-(\beta'\zeta -\alpha) t_w}  dt_w \\
& \lesssim {R^{2\delta}  N^{-\beta' - 2\alpha /\zeta} e^{(\beta'\zeta -\alpha) R- (\beta'\zeta -\alpha) R/2 }}
e^{-\left(2\beta' \zeta-3\alpha \right)\omega (N)}  \\
&\asymp {R^{2\delta}  N^{-\beta' - 2\alpha /\zeta} e^{(\beta'\zeta -\alpha) R/2}}
e^{-\left(2\beta' \zeta-3\alpha \right)\omega (N)} \\
& \asymp {R^{2\delta}  N^{-\beta' - 2\alpha /\zeta + \beta' - \alpha/\zeta}}
e^{-\left(2\beta' \zeta-3\alpha \right)\omega (N)} =
{R^{2\delta}  N^{- 3\alpha /\zeta}}
e^{-\left(2\beta' \zeta-3\alpha \right)\omega (N)}.
\end{split}
\end{equation*}
Therefore,
\begin{equation} \label{eq:phi7}
\begin{split}
\phi_4  \lesssim
\begin{cases}
R^{2\delta + 2} N^{-\beta' - 2\alpha / \zeta}
e^{-\left(\beta' \zeta-2\alpha \right)\omega (N)}, & \mbox{if $\beta' \zeta/\alpha \leq 1$} \\
{R^{2\delta}  N^{- 3\alpha /\zeta}}
e^{-\left(3\beta' \zeta-4\alpha \right)\omega (N)} , & \mbox{if $\beta' \zeta/\alpha > 1$}
\end{cases}.
\end{split}
\end{equation}
Finally, we consider
\[
\phi_5 =\!
\int_{R/2-\omega(N)}^{R}\!\int_{R-t_w - 2\omega (N)}^{R/2-\omega(N)}\!
\int_{R-t_w - 2\omega (N)}^{R/2-\omega(N)}\!\! \left ( \frac{ R^{\delta}}{A_{u,v}}\right )^{\beta'}\!\!\!\!
e^{-\alpha(t_u+t_v+t_w)} dt_u dt_v dt_w% \frac{R^\delta }{N^{\beta'}} \int_0^{R/2-\omega(N)}\int_{t_v}^{R/2-\omega(N)}\int_{R/2-\omega(N)}^{R-t_u-\omega(N)} e^{(\beta'\zeta/2)(t_v+t_u)}e^{-\alpha (t_u+t_v+t_w)}dt_w dt_u dt_v .
\]
The integral of the above expression is estimated as follows:
\begin{equation*}
\begin{split}
&  \int_{R/2-\omega(N)}^{R}\int_{R-t_w - 2\omega (N)}^{R/2-\omega(N)}
\int_{R-t_w - 2\omega (N)}^{R/2-\omega(N)} \left ( \frac{R^{\delta}}{A_{u,v}}\right )^{\beta'}
e^{-\alpha(t_u+t_v+t_w)} dt_u dt_v dt_w  \\
& \asymp {R^{\delta} \over N^{\beta'}}~\int_{R/2-\omega(N)}^{R}\int_{R-t_w - 2\omega (N)}^{R/2-\omega(N)}
\int_{R-t_w - 2\omega (N)}^{R/2-\omega(N)}
\!\!e^{{\beta'\zeta \over 2}\left( t_u + t_v \right)-\alpha \left(t_u+t_v+t_w \right)}
dt_u dt_v dt_w \\
&\lesssim
{R^{\delta}\over N^{\beta'}}~\int_{R/2-\omega(N)}^{R}
e^{-\alpha t_w + 2\left(\beta'\zeta /2 - \alpha \right)(R-t_w - 2 \omega (N))} dt_w \\
&= {R^{\delta}\, e^{2\left(\beta'\zeta /2 - \alpha \right)(R- 2 \omega (N))} \over N^{\beta'}}~
\int_{R/2-\omega(N)}^{R} e^{-(\beta' \zeta -  \alpha)t_w} dt_w \\
&= R^\delta {N^{2\beta' - 4\alpha /\zeta} \over N^{\beta'}} e^{-4(\beta' \zeta /2 - \alpha ) \omega (N)} 
\int_{R/2-\omega(N)}^{R} e^{-(\beta' \zeta -  \alpha)t_w} dt_w .
\end{split}
\end{equation*}
Hence, there are three cases according to the value of $\beta'\zeta /\alpha$, thus
obtaining
\begin{equation}  \label{var_phi3}
\phi_5  \lesssim \begin{cases}
R^\delta N^{-\beta'-2\alpha /\zeta}
e^{-4\left(\beta'\zeta /2 - \alpha \right) \omega (N)}, & \mbox{if $\beta'\zeta /\alpha < 1$ } \\
R^{\delta +1} N^{\beta' - 4\alpha /\zeta}
e^{-4(\beta' \zeta - 2\alpha)\omega (N)} & \\
\quad = R^{\delta +1} N^{- 3\alpha /\zeta}
e^{-4(\beta' \zeta - 2\alpha)\omega (N)}, & \mbox{if $\beta'\zeta /\alpha =1$} \\
R^\delta N^{-3\alpha /\zeta} e^{-(\beta' \zeta - 3\alpha)\omega (N)},
& \mbox{if $\beta'\zeta /\alpha > 1$}
\end{cases}.
\end{equation}

\section{Proof of Lemma~\ref{lem:relAngle}} \label{app:B}

We begin with the hyperbolic law of cosines: 
\begin{equation*} 
	\cosh (\zeta d(u,v)) = \cosh (\zeta (R- t_u )) \cosh (\zeta (R-t_v)) - \sinh (\zeta (R- t_u)) \sinh (\zeta (R-t_v)) 
\cos ( \theta_{u,v} ).
\end{equation*}

The right-hand side of the above becomes:
\begin{equation} \label{eq:coslaw} 
\begin{split} 
	&\cosh (\zeta (R- t_u) ) \cosh (\zeta (R-t_v)) - \sinh (\zeta (R- t_u )) \sinh (\zeta (R-t_v) ) \cos ( \theta_{u,v} ) = \\
	& {e^{\zeta (2R- (t_u + t_v))} \over 4} \left( \left(1+e^{-2\zeta (R-t_u)}\right)\left(1+e^{-2\zeta (R-t_v)}\right) 
      -\left(1-e^{-2\zeta (R-t_u)}\right)\left(1-e^{-2\zeta (R-t_v)}\right) \cos (\theta_{u,v}) \right) \\
	& = {e^{\zeta (2R- (t_u + t_v))} \over 4} \left(1- \cos (\theta_{u,v}) + \left(1+ \cos (\theta_{u,v}) \right) 
	\left( e^{-2\zeta (R-t_u)} +  e^{-2\zeta (R-t_v)}\right) + O\left( e^{-2\zeta (2R- (t_u + t_v))}\right)\right) 
\end{split}
\end{equation}
Therefore,
\begin{equation*}
	\begin{split}
    &\cosh (\zeta d(u,v)) \leq \\
	&  {e^{\zeta (2R- (t_u + t_v))} \over 4} \left(1- \cos (\theta_{u,v}) + 2 
	\left( e^{-2\zeta (R-t_u)} +  e^{-2\zeta (R-t_v)}\right)  + O\left( e^{-2\zeta (2R- (t_u + t_v))}\right)\right).
	\end{split}
\end{equation*}
Since $t_u+t_v < R-c_0$, the last error term is $O(N^{-4})$. 
Also, it is a basic trigonometric identity that $1- \cos (\theta_{u,v}) = 2\sin^2 \left( {\theta_{u,v} \over 2} \right)$. 
The latter is at most ${\theta_{u,v}^2 \over 2}$. 
 Therefore, the upper bound on $\theta_{u,v}$ yields:
  \begin{equation*}  
\begin{split}
      &\cosh (\zeta d(u,v)) \leq \\
      &{e^{\zeta (2R- (t_u + t_v))} \over 4} \left( {\theta_{u,v}^2 \over 2} + 2 
	\left( e^{-2\zeta (R-t_u)} +  e^{-2\zeta (R-t_v)}\right) + O\left( {1\over N^4}\right)\right) \\
	&\leq {e^{\zeta (2R- (t_u + t_v))} \over 4} \left( 2(1-\eps )^2 e^{\zeta (t_u+t_v -(1-\delta )R)} + 2 
	\left( e^{-2\zeta (R-t_u)} +  e^{-2\zeta (R-t_v)}\right) \right) + O\left( 1\right) \\
	&= (1-\eps)^2{e^{\zeta (1+\delta ) R} \over 2} + {1\over 2}\left(e^{\zeta ( t_u-t_v )} + e^{\zeta (t_v-t_u)} \right) + O(1)\\
	&< (1-\eps)^2{e^{\zeta (1+\delta) R} \over 2} + \eps {e^{\zeta (1+\delta) R} \over 2} + O(1) < {e^{\zeta (1+\delta) R} \over 2}, 
	\end{split}
 \end{equation*}  
for $N$ sufficiently large and $c_0$ such that $e^{-c_0}<\frac{1}{2}\eps$, since $t_u+t_v < (1- |\delta |) R-c_0$ and $t_u,t_v\geq0$.  
This implies that $t_u-t_v,t_v-t_u< (1+\delta ) R-c_0$ and, therefore, 
${1\over 2}\left(e^{\zeta (t_u-t_v)} + e^{\zeta (t_v-t_u)}\right)<{1\over 2}\left(e^{\zeta (1+\delta) R-c_0} + e^{\zeta 
(1+\delta) R-c_0}\right)<\eps{e^{\zeta (1+\delta) R}
\over 2}$. 
Also, since $\cosh (\zeta d(u,v))> {1\over 2} e^{\zeta d(u,v)}$, it follows that $d(u,v) < (1+\delta) R$. 

To deduce the second part of the lemma, we consider a lower bound on (\ref{eq:coslaw}) using the lower bound on $\theta_{u,v}$: 
\begin{equation} \label{eq:coslawLow} 
    \begin{split} 
    & \cosh (\zeta d(u,v))  \geq   {e^{\zeta (2R- (t_u + t_v))} \over 4} \left(1- \cos (\theta_{u,v}) \right) +O(1) \\
     & \geq {e^{\zeta (2R- (t_u + t_v))} \over 4} \left(1- \cos \left(2(1+\eps){e^{\frac{\zeta}{2} (t_u+t_v-(1-\delta )R)}}\right) \right)
+O(1). 
    \end{split}
\end{equation}
Using again that $1- \cos (\theta) = 2\sin^2 \left( {\theta  \over 2} \right)$ we deduce that 
 $$ 1- \cos \left( 2(1+\eps)e^{\frac{\zeta}{2}(t_u+t_v- (1-\delta) R)}\right)  = 2 \sin^2 \left( {1\over 2}~4(1+\eps)^2  
e^{\zeta (t_u+t_v-(1-\delta) R)} \right).$$
Since $t_u+t_v < (1-|\delta|) R - c_0$, it follows that $t_u+t_v- (1-\delta) R <- c_0$. 
So the latter is 
    \[ \begin{split}\sin & \left( {1\over 2}~4(1+\eps)^2  e^{\zeta(t_u+t_v- (1-\delta) R)}  \right)> 2\left(1+{\eps \over 2}\right)^2
e^{\zeta (t_u+t_v- (1-\delta) R)},
\end{split}\]
for $N$ and $c_0$ large enough, using the Taylor's expansion of the sine function around $0$.
Substituting this bound into (\ref{eq:coslawLow}) we have 
 $$ \cosh (\zeta d(u,v))  \geq \left(1+{\eps \over 2}\right)^2 {e^{\zeta (1+\delta) R} \over 2} + O(1).$$
Thus, if $d(u,v) \leq (1+\delta ) R$, the left-hand side would be smaller than the right-hand side which would lead to a contradiction.

\end{document}